\pdfoutput=1
\documentclass{amsart}
\usepackage{DaxStyle,DaxMacros}

\title[On homotopy groups of spaces of embeddings of an arc or a circle]{On homotopy groups of spaces of embeddings of an arc or a circle: the Dax invariant}

\author{Danica Kosanovi\'c}
\address{ETH Z\"urich, Department of Mathematics, R\"amistrasse 101, 8092 Z\"urich, Switzerland}
\email{danica.kosanovic@math.ethz.ch}

\date{July 16, 2022}

\begin{document}

\begin{abstract}
    We compute in many classes of examples the first potentially interesting homotopy group of the space of embeddings of either an arc or a circle into a manifold $M$ of dimension $d\geq4$.
    In particular, if $M$ is a simply connected 4-manifold the fundamental group of both of these embedding spaces is isomorphic to the second homology group of $M$, answering a question posed by Arone and Szymik. 
    The case $d=3$ gives isotopy invariants of knots in a 3-manifold, that are universal of Vassiliev type $\leq1$, and reduce to Schneiderman's concordance invariant.
\end{abstract}

\maketitle

\section{Introduction}\label{sec:intro}

Classical knot theory studies path components of the space of embeddings of a circle into a 3-manifold, and 2-knot theory is concerned with embeddings of a surface into a 4-manifold. One can also study codimension two knotting phenomena in higher dimensions, and in fact even more generally -- there exist smooth embeddings of codimension higher than two which are not mutually isotopic, as shown by Haefliger~\cite{Haefliger}.

In this and the forthcoming paper \cite{K-knotted} we consider a generalisation of knot theory in another direction: instead of studying only the set of components of a space of smooth embeddings, we study its \emph{homotopy groups}, when the source manifold is a single arc or a single circle, and the target is a compact smooth manifold of dimension $d\geq3$ (with or without boundary, and the arc is embedded neatly). For $d\geq4$ these groups give another notion of \emph{knottedness}, as their nontrivial elements are represented by multi-parameter families of embeddings which cannot be trivialised through such families. Interestingly, they also show useful for problems in low-dimensional topology, for example, in the recent work of Budney and Gabai \cite{Budney-Gabai} and ours with Teichner~\cite{KT-LBT}.

The present work is concerned with the lowest homotopy group potentially distinguishing embeddings from immersions, namely, ``knotted'' classes in degree $d-3$ for the ambient manifold of dimension $d$. We state our results for the respective cases of arcs and circles in Sections~\ref{intro-sec:arcs} and~\ref{intro-sec:circles}. In Section~\ref{sec:applications} we discuss applications, examples and the case $d=3$.

\subsection{Arcs}\label{intro-sec:arcs}
For $\D^1=[-1,1]$ and an oriented smooth $d$-dimensional manifold $X$ consider the spaces
\begin{equation}\label{eq:def-arcs}
\begin{aligned}
    \Arcs{X} &=\{K\colon\D^1\hra X \mid K(-1)=x_-,\,K(1)=x_+\}, \\ 
    \Imm_\partial(\D^1,X) &=\{K\colon\D^1\imra X \mid K(-1)=x_-,\,K(1)=x_+\},
\end{aligned}
\end{equation}
of embedded and immersed arcs respectively, which are neat and fixed on the boundary (i.e.\ $K$ is transverse to $\partial X$ and $K\cap\partial X=\{x_-,x_+\}$). Let us fix an arbitrary basepoint $\u\in\Arcs{X}$.

For immersions we have a map
\begin{equation}\label{eq:pu}
    p_\u\colon\ImArcs{X}\to\Omega X,
\end{equation}
which by Smale~\cite{Smale} induces isomorphisms on homotopy groups in degrees $n\leq d-3$:
\[
    \pi_np_\u\colon\pi_n(\ImArcs{X},\u)\cong\pi_n(\Omega X,\const_{x_-})\cong\pi_{n+1}(X,x_-).
\]
For $F\colon\S^n\to\ImArcs{X}$ the map $p_\u\circ F\colon\S^n\to\Omega X$ sends $\vec{t}$ to the loop $F(\vec{t})\cdot\u^{-1}$ based at $x_-$, where the inverse denotes the reverse of a path and the dot stands for the concatenation of paths. 

Following \cite{Dax} and \cite{Gabai-disks} we described in \cite{KT-LBT} a range of homotopy groups of $\Arcs{X}$.
\begin{theorem}[\cite{KT-LBT}]\label{thm:KT-arcs}
    Assume $d\geq 4$. Then there is a bijection $\pi_0\Arcs{X}\cong\pi_1X$, and for any basepoint $\u\in\Arcs{X}$ there are isomorphisms
    \[
        \pi_n(\Arcs{X},\u){\color{black}\xrightarrow[\cong]{\pi_{d-3}(\incl_X,\u)}}
        \pi_n(\ImArcs{X},\u)
        {\color{black}\xrightarrow[\cong]{\pi_{d-3}p_\u}}\pi_{n+1}X
    \]
    where $1\leq n\leq d-4$, and a group extension
    \begin{equation}\label{eq:thm1}
    \begin{tikzcd}[column sep=large]
        \faktor{\Z[\pi_1X\sm1]}{\dax_\u(\pi_{d-1}X)}\ar[tail,shift left]{r}{\partial\realmap} & \pi_{d-3}\big(\Arcs{X},\u\big)\arrow[shift left,dashed]{l}{\Dax} \ar[two heads]{r}{\pi_{d-3}p_\u} & \pi_{d-2}X.
    \end{tikzcd}
    \end{equation}
    For $d=4$ this extension is central, with an explicit commutator pairing.
\end{theorem}
Let us describe the maps appearing in the theorem, referring to Section~\ref{sec:arcs} for details. Firstly, $p_\u$ in \eqref{eq:thm1} is obtained from \eqref{eq:pu} by precomposing with the inclusion $\incl_X\colon\Arcs{X}\hra\Imm_\partial(\D^1,X)$. The homomorphism $\partial\realmap$ is is an explicit {\color{black}\emph{geometric realization}: for $g\in\pi_1X\sm1$ the family $\partial\realmap(g)$ has a piece of $u$ dragged around an embedded loop representing $g$, so that it comes back close to a meridian sphere to $\u$ and then swings around it; see Figure~\ref{fig:realmap} and \eqref{eq:realmap-def}. The \emph{Dax invariant} $\Dax$ is the inverse of $\partial\realmap$} on the subgroup $\ker(\pi_{d-3}p_\u)$, and is defined by {\color{black}picking a path through immersed arcs from the given $(d-3)$-parameter family of embedded arcs to the constant one $\const_\u$, and then counting double points that occur, together with associated loops; see \eqref{eq:Dax-def} for the precise formula}. Finally, computing $\Dax$ {\color{black}on} self-homotopies of the trivial family gives the homomorphism
\begin{equation}\label{eq:dax-u-def}
    \dax_\u\colon\pi_{d-1}X\ra\Z[\pi_1X\sm1].
\end{equation}
{\color{black}Namely, given $a\in\pi_{d-1}X$ we pick $F_A\colon\S^{d-2}\to\Imm_\partial(\D^1,X)$ representing it, that is $p_\u(F_A)=a$, and view $F_A$ as a self-homotopy of $\const_\u$, for which we can compute $\dax_\u(a)\coloneqq\Dax(F_A)$; see \eqref{eq:dax-def}.}
See Remark~\ref{rem:GW} for a relation of $\Dax$ and $\dax_\u$ to the Goodwillie--Weiss embedding calculus.
\begin{figure}[!htbp]
    \centering
       \begin{tikzpicture}[scale=0.73,every node/.style={scale=1.2},even odd rule]
        \clip (0.25,-1.5) rectangle (11.75,6);
        \draw[blue!90!black, very thick]
                (7.6,1.8) to[out=-140,in=-130, distance=1.8cm] (8.1,2);
        \draw[blue!80!black, very thick]
                (7.6,1.8) to[out=-120,in=-110, distance=3cm] (8.1,2);
        \fill[white] (8,0) circle (1cm);
        
        \draw[black,very thick] 
            (1,0) -- (2,0)
            (4,0) -- (11,0)
            (6.4,-0.7) node{\color{black!80!white}$\mathbb{S}^{d-2}_x$};
        \draw[gray,thick, dashed,postaction=decorate,decoration = {markings, mark = at position 0.2 with {\arrow{stealth}}}]
            (1,-0.08) -- (11,-0.08) node[below,pos=0.2]{$\mathsf{u}$};
        \draw[very thick]
            (1,-0.01) circle (1.6pt) node[below]{\small$\mathsf{u}(-1)$}
            (11,-0.01) circle (1.6pt) node[below]{\small$\mathsf{u}(1)$};
        \draw[red,very thick]
            (8, 0) node[below=1pt]{$x$} circle (1.5pt);

        \foreach \y in {8}{
            \shade[ball color=blue!5!white,opacity=0.55] (\y,0) circle (1cm);
            \draw[black!60!yellow] (\y-1,0) arc (180:360:1cm and 0.5cm);
            \draw[black!60!yellow,dotted] (\y-1,0) arc (180:0:1cm and 0.5cm);
            \draw[black!60!yellow,semithick] (\y,0) circle (1cm);
            }
        \draw[blue!60!black, very thick]
                        (7.6,1.8) to[out=-110,in=-100, distance=1.15cm] (7.6,-1.06)
                        (8,1.06) to[out=40,in=-120, distance=0.07cm] (8.1,2);
        \draw[black, ultra thick]
                (7.6,1.8) to[out=-90,in=-100, distance=1.15cm] (8,-1.06)
                (8.2,1.06) to[out=40,in=-120, distance=0.07cm] (8.1,2);
        \draw[black, very thick, postaction=decorate,decoration = {markings, mark = at position 0.1 with {\arrow{stealth}} , mark = at position 0.9 with {\arrowreversed{stealth}}}]
                        (2,0) to[out=0,in=-100, distance=0.5cm]
                        (3,3.15) to[out=80,in=98, distance=4cm] (8.1,2)
                        (4,0) to[out=180,in=-100, distance=0.5cm]
                        (4,3) to[out=80,in=94, distance=3cm] (7.6,1.8);
        \fill[color=black!20!white] (5.6,3) circle (0.55cm);
        \draw[thick,postaction = decorate,decoration = {markings, mark = at position 0.76 with {\arrowreversed{>}} }] (5.6,3) circle (0.65cm) node[left=0.36cm]{$g$};
        
        \draw[purple!50!black, very thick]
                (7.6,1.8) to[out=-80,in=-80, distance=1.4cm] (8.8,-0.8)
                (8.5,0.95) to[out=100,in=-80, distance=0.1cm] (8.1,2)
        ;
        \draw[purple!65!black, very thick]
              (7.6,1.8) to[out=-55,in=-55, distance=3.5cm] (8.1,2);
        \draw[purple!80!black, very thick]
                (7.6,1.8) to[out=-40,in=-40, distance=2.2cm] (8.1,2);
        \draw[black, thick,-latex]
                (8.1,2) -- (7.6,1.8);
        \draw[black, very thick]
                (9.25,-1.2) node[purple!60!black]{\textbf{$\alpha_t$}}
                (8.8,4) node{\textbf{$\partial\mathfrak{r}(g)(\vec{t})$}};
    \end{tikzpicture}~\quad
    \begin{tikzpicture}[scale=0.73,every node/.style={scale=1.2},even odd rule]
    \clip (0.25,-1.5) rectangle (11.75,6);
        \draw[black,very thick]
            (4,0) -- (7.33,0) 
            (7.65,0) -- (11,0);
        \draw[red,very thick, postaction=decorate,decoration = {markings, mark = at position 0.8 with {\arrow{stealth}}}]
            (8,-0.08) -- (7.61,-0.08) 
            (7.3,-0.08) -- (1,-0.08)
            node[below,pos=0.6]{$\mathsf{u}_{\leq\theta^+_x}^{-1}$};
        \draw[very thick]
            (1,-0.01) circle (1.6pt) node[below]{\small$\mathsf{u}(-1)$}
            (11,-0.01) circle (1.6pt) node[below]{\small$\mathsf{u}(1)$};
            
        \draw[black,very thick]
                (7.5,1.8) to[out=-95,in=-120, distance=1.25cm] (8,0);
        \draw[red,dashed,very thick]
            (1,0) -- (2,0)
            (8,0) to[out=60,in=-80, distance=0.7cm] (8.1,2) 
            (8.25,4) node{\textbf{$\rho_{\leq\theta^-_x}$}};
        \draw[red,dashed,very thick, postaction=decorate,decoration = {markings, mark = at position 0.5 with {\arrow{stealth}}}]
                        (2,0) to[out=0,in=-100, distance=0.5cm]
                        (3,3.15) to[out=80,in=98, distance=4cm] (8.1,2);
        \draw[black,very thick]
                        (4,0) to[out=180,in=-100, distance=0.5cm]
                        (4,3) to[out=80,in=85, distance=2.8cm] (7.5,1.8);
                        
        \fill[color=black!20!white] (5.6,3) circle (0.55cm);
        \draw[thick,postaction = decorate,decoration = {markings, mark = at position 0.76 with {\arrowreversed{>}} }] (5.6,3) circle (0.65cm) node[left=0.36cm]{$g$};
        
        \fill[red,draw,very thick]
            (8, -0.1) node[below]{$x$}
            (8, 0) circle (1.7pt);
\end{tikzpicture}
    \caption{{\color{black}
    \emph{Left.} The family $\partial\realmap(g)(\vec{t})\in\Arcs{X}$ for several values $\vec{t}\in\S^{d-3}$ and $d=4$ (coloured arcs are in past or future). \emph{Right.} The single immersed arc $\rho$ in a homotopy from $\partial\realmap(g)$ to $\const_\u$ has one double point $x=\rho_{\theta_x^-}=\rho_{\theta_x^+}$, with sign $+1$ and loop $\rho_{\leq\theta_x^-}\cdot \rho_{\leq\theta_x^+}^{-1}\simeq\rho_{\leq\theta_x^-}\cdot \u_{\leq\theta_x^+}^{-1}\simeq g$, so $\Dax(\partial\realmap(g))=g$.}}
    \label{fig:realmap}
\end{figure}
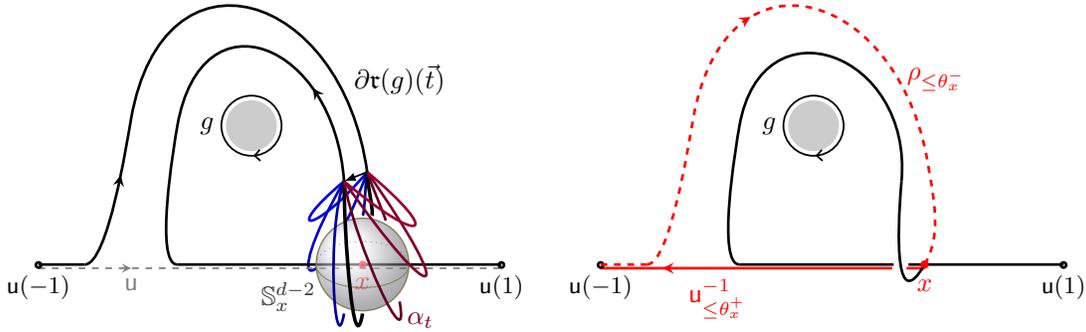

Note that if $\pi_1X=1$ we have the isomorphism
$\pi_{d-3}p_\u\colon\pi_{d-3}(\Arcs{X},\u)\xrightarrow{\cong}\pi_{d-2}X$,
so the first potentially interesting group (i.e.\ one which would distinguish embeddings from immersions) turns out \emph{not} to be so interesting. On the other hand, for $\pi_1X\neq1$ it remains to understand the image of \eqref{eq:dax-u-def} and also describe the extension \eqref{eq:thm1}. For the latter see Questions~\ref{question1} and~\ref{question2} below, while the former is studied in Theorem~\ref{thm:dax-compute}, which gives simplifying formulae for $\dax_\u$. We use them to compute $\ker(\pi_{d-3}p_\u)$ for several classes of target manifold $X$ in Sections~\ref{intro-sec:examples} and~\ref{intro-sec:more-examples}.

In Theorem~\ref{thm:dax-compute} we consider $\dax_\u$ actually for any $d\geq3$, and compare it to the homomorphism
\[
    \dax\coloneqq\dax_{\u_-}
\]
for a fixed arc $\u_-\colon\D^1\hra X$ with endpoints $\u_-(-1)=x_-$ and $\u_-(1)=x_-'$, a point close to $x_-$ (so $\u$ and $\u_-$ live in different spaces $\Arcs{X}$), and so that \emph{$\u_-$ is isotopic into $\partial X$ rel.\ endpoints}.

{\color{black}We need to fix some more notation.} Let $ga\in\pi_nX$ denote the usual action of $g\in\pi_1X\coloneqq\pi_1(X,x_-)$ on $a\in\pi_nX\coloneqq\pi_n(X,x_-)$ for $n\geq1$, and $g \bm{k}\in\pi_1(X,\partial X)$ the action of $g\in\pi_1X$ on the set of arcs $\bm{k}\in\pi_1(X,\partial X)\coloneqq\pi_1(X,\partial X,x_-)$, by precomposition at $x_-$. We will express $\dax_\u$ in terms of $\dax$ and the well-known algebraic invariant, the \emph{equivariant intersection pairing}:
\[
    \lambda\colon\pi_{d-1}X\times\pi_1(X,\partial X)\to\Z[\pi_1X].
\]
{\color{black}This is an intersection pairing on the homology of the universal cover, recalled in} Section~\ref{subsec:formulae}. In fact, we use $\lambdabar$, defined as $\lambda$ minus its term at $1\in\pi_1X$. Finally, we write $\lambda(\bm{k},a)\coloneqq(-1)^{d-1}\ol{\lambda(a,\bm{k})}$ for $a\in\pi_{d-1}X$ and $\bm{k}\in\pi_1(X,\partial X)$, using the involution on the group ring $\Z[\pi_1X]$ linearly extending $\ol{g}\coloneqq g^{-1}$. The following result{\color{black}s} will be proven in Section~\ref{subsec:formulae}, and applied in  Section~\ref{sec:applications}.
\begin{mainthm}\label{thm:dax-compute}
    Assume $d\geq 3$ and denote by $\bm{\u}$ the homotopy class of $\u$ in $\pi_1(X,\partial X)$. For any $a\in\pi_{d-1}X$ and $g\in\pi_1X$ we have
    \begin{enumerate}[label=(\Roman*)]
    \item\label{eq:A}
        $\dax_\u(a)=\dax(a)+\lambdabar(a,\bm{\u})$,
    \item\label{eq:B}
        $\dax(ga)=g\dax(a)\ol{g}-\lambdabar(g a,g)+\lambdabar(g,g a)$.
    \end{enumerate}
\end{mainthm}
\begin{maincor}\label{cor:dax-compute}
    \begin{enumerate}
\item\label{eq:cor1}
    If $a$ has an embedded representative, then $\dax_\u(a)=\lambdabar(a,\bm{\u})$.
\item\label{eq:cor2}
        $\dax_{g\u}(a)-\dax_\u(a)=\lambdabar(a,g\bm{\u})-\lambdabar(a,\bm{\u})$.
\item\label{eq:cor3}
        $\dax_{g\u}(g a)-g\dax_\u(a)\ol{g}=\lambdabar(g,ga)$.
\item\label{eq:cor4}
        $\dax_\u(g a)-g\dax_\u(a)\ol{g}=\lambdabar(g a,\bm{\u})-\lambdabar(g a,g\bm{\u})+\lambdabar(g,g a)$.
\item\label{eq:cor5}
    If $a$ has an embedded representative, then
        $\dax_\u(g a)=\lambdabar(g a,\bm{\u})-\lambdabar(g a,g)+\lambdabar(g,g a)$.
    \end{enumerate}
\end{maincor}
\begin{remark}
    We have $\pi_0\Arcs{X}\cong\pi_0\Map_\partial(\D^1,X)=\{f\colon\D^1\to X\mid f(\partial\D^1)=\{x_{\pm}\}\}/\simeq$ ($\cong\pi_1X$ via $p_\u$) if $d\geq4$, so $\Dax$ and $\dax_\u$ depend only on the homotopy class $\bm{\u}\in\pi_0\Map_\partial(\D^1,X)$. However, we define $\Dax$ also for $d=3$, when it can depend on $\u$ itself, see Section~\ref{intro-sec:3d}.
\end{remark}
\begin{remark}\label{rem:KT-mu3}
    In \cite{KT-LBT} we exhibit for $d=4$ a relation between $\dax$ and Wall's self-intersection invariant $\mu_3$, also important for computations; for example, using it we showed that every finitely generated abelian group is realised as $\ker(\pi_1p_\u)$. See also Remark~\ref{rem:KTcor} and Theorem~\ref{thm:concordance} below.
\end{remark}
 
\subsection{Circles}\label{intro-sec:circles}
For a $d$-dimensional manifold $N$ (with or without boundary) we study the space $\Emb(\S^1,N)$ of knots in $N$, by comparing it to the space of immersions via the inclusion
\[
    \incl_N\colon\Emb(\S^1,N)\hra\Imm(\S^1,N).
\]
Given an arbitrary basepoint $\s\colon\S^1\hra N$, pick a small open ball $D^d=\mathrm{int}(\D^d)\subseteq N$ around the point $\s(e)$. Let $\u\colon\D^1\hra N\sm D^d$ be the neatly embedded arc obtained by restricting $\s$ so that $\u(\D^1)=\s(\S^1)\cap(N\sm D^d)$, see Figure~\ref{fig:Phi}. In Section~\ref{sec:circles} we will relate the homotopy type of $\Emb(\S^1,N)$ based at $\s$ to the space $\Arcs{N\sm D^d}$ based at $\u$, and use it to prove the following {\color{black}(the first part seems to be well known, but we provide a proof for the sake of completeness, see Section~\ref{subsec:immersed-circles}).}
\begin{mainthm}\label{thm:circles-main}
Assume $d\geq 4$ {\color{black}and denote by $\bm{\s}\in\pi_1N=\pi_1(N,\s(e))$ the homotopy class of $\s$.}
\begin{enumerate}[label=(\Roman*)]
    \item 
    There is a bijection of $\pi_0\Imm(\S^1,N)$ with the set of conjugacy classes of $\pi_1N$, and for any basepoint $\s\in\Imm(\S^1,N)$ and $1\leq n\leq d-3$ there is a group extension
    \begin{equation}\label{eq:B-imm}
    \begin{tikzcd}
        \faktor{\pi_{n+1}(N)}{\langle a-\bm{\s}\,a\rangle} \ar[tail]{r} & 
        \pi_n\big(\Imm(\S^1,N),\s\,\big) \ar[two heads]{r}{} & 
        \big\{b\in\pi_nN \mid b=\bm{\s}\,b\big\}.
    \end{tikzcd}
    \end{equation}
    Moreover, if $\s$ is nullhomotopic, this extension \emph{splits}.
    \item\label{thm:circles-mainII}
    For $0\leq n\leq d-4$ the map $\incl_N$ induces isomorphisms $\pi_n(\Emb(\S^1,N),\s)\cong \pi_n(\Imm(\S^1,N),\s)$ while for $n=d-3$ there is a group extension
    \begin{equation}\label{eq:B-emb}
    \begin{tikzcd}[column sep=large,row sep=tiny]
        \faktor{\Z[\pi_1N\sm1]}{rel_{\s}}\ar[tail,shift left]{r}{\partial\realmap} & \pi_{d-3}\big(\Emb(\S^1,N),\s\,\big)\arrow[shift left,dashed]{l}{\Dax} \ar[two heads]{r}{\pi_{d-3}(\incl_N,\s)} & \pi_{d-3}\big(\Imm(\S^1,N),\s\,\big),
    \end{tikzcd}
    \end{equation}
    for $rel_{\s}\coloneqq\dax_\u(\pi_{d-1}(N\sm D^d))\oplus\big\{\Dax(\delta^{whisk}_{\s}(c)) \mid c=\bm{\s}\,c\in\pi_{d-2}N\big\}$, where the family of embedded arcs $\delta^{whisk}_{\s}(c)\in\ker\pi_{d-3}(\incl_{N\sm D^d},\u)$ is obtained from $\s$ by isotoping $\s(e)$ around $c$. 
    Moreover, if $\s$ is nullhomotopic, then $\delta^{whisk}_{\s}(c)$ vanishes for all $c$.
\end{enumerate}
\end{mainthm}
{\color{black}We can use Corollary~\ref{cor:dax-compute} to identify a big class of relations in $rel_\s$.}
\begin{cor}\label{cor:gPhi}
    In the above setting, let $\Phi\colon\S^{d-1}=\partial\D^d\hra N\sm D^d$ be a parametrization of the boundary of the removed ball, and $\bm{\Phi}\in\pi_{d-1}(N\sm D^d)$ its homotopy class. Then for any $g\in\pi_1N$ we have 
$\dax_\u(g\,\bm{\Phi})=(-1)^{d-1}\ol{g}-g\ol{\bm{\s}}\pmod{1}$. In particular, these expressions belong to $rel_{\s}$.
\end{cor}
\begin{proof}
    Since $\bm{\Phi}$ has an embedded representative, Corollary~\ref{cor:dax-compute}\ref{eq:cor5} gives 
    \begin{equation}\label{eq:gPhi-proof}
        \dax_\u(g\bm{\Phi})=\lambdabar(g\bm{\Phi},\bm{\u})-\lambdabar(g\bm{\Phi},g)+\lambdabar(g,g\bm{\Phi}).
    \end{equation}
    We claim that $\lambda(\bm{\Phi},\bm{\u})=1-\ol{\bm{\s}}$, so $\lambda(g\bm{\Phi},\bm{\u})=g\lambda(\bm{\Phi},\bm{\u})=g-g\ol{\bm{\s}}$, by a property of $\lambda$ (see Lemma~\ref{lem:lambda}).
        Indeed, there are two intersection points of a pushoff $\Phi'\colon\S^{d-1}\hra N\sm D^d$ and $\bm{\u}$ as in Figure~\ref{fig:Phi}.
    The first point is close to $\u(-1)$, with the group element clearly $1\in\pi_1N$ and positive sign, and the other one is close to $\u(+1)$, with the negative sign (only the orientation of $\u$ changed) and the group element given by the dashed arc followed by $\u^{-1}$, so homotopic to $(\s\cap D^d)^{-1}\cdot\u^{-1}=\s^{-1}$.
\begin{figure}[!htbp]
    \centering
    \begin{tikzpicture}[scale=0.8,every node/.style={scale=0.87},even odd rule]
        \clip (-0.4,-1.9) rectangle (4.1,0.9);

        \fill[color=black!30!white] (3.35,-.5) circle (0.2cm);
        \draw[black,postaction = decorate,decoration = {markings, mark = at position 0.27 with {\arrowreversed{>}} }] (3.35,-.5) circle (0.25cm) node[left=0.15cm]{\footnotesize$\bm{\mathsf{s}}$};

        \draw[green!70!black,ultra thick,postaction = decorate,decoration = {markings, 
            mark = at position 0.08 with {\arrow{stealth}} }]
            (1,-.5)  to[out=90,in=90, distance=1.02cm] (4.01,-.5)
            to[out=-90,in=-90, distance=1.02cm] (1,-.5) ;
        \draw[black,thick,postaction = decorate,decoration = {markings, 
            mark = at position 0.28 with {\arrow{stealth}} }]
            (1.7,0.2)  to[out=15,in=90, distance=0.8cm] (4.05,-.5)
            to[out=-90,in=-15, distance=0.8cm] (1.7,-1.2) ;
        \draw[thick,green!70!black]
            (0.9,-0.18) node{$\mathsf{s}$};
        \draw[thick,black]
            (2.85,0.45) node{$\mathsf{u}$};
        \fill[green!70!black,fill=black] 
            (1,-.5) circle (1pt) 
            (0.7,-.6) node{\scriptsize$\mathsf{s}(e)$};

        \foreach \y in {1}{
            \shade[ball color=black!5!white,opacity=0.3] (\y,-.5) circle (1cm);
            \draw[] (\y-1,-.5) arc (180:360:1cm and 0.5cm);
            \draw[black!60!yellow,dotted] (\y-1,-.5) arc (180:0:1cm and 0.5cm);
            \draw[black!60!yellow,semithick] (\y,-.5) circle (1cm);
            
            \shade[ball color=black!5!white,opacity=0.15] (\y,-.5) circle (1.3cm);
            \draw[black!60!yellow,opacity=0.5] (\y-1.3,-.5) arc (180:360:1.3cm and 0.65cm);
            \draw[black!60!yellow,dotted,opacity=0.5] (\y-1.3,-.5) arc (180:0:1.3cm and 0.65cm);
            \draw[black!60!yellow,semithick,opacity=0.5] (\y,-.5) circle (1.3cm);
            }
        \draw[black!60!yellow]
            (0.1,-1.2) node{\small$\Phi$};
        \draw[black!60!yellow,opacity=0.8]
            (-0.15,-1.5) node{\small$\Phi'$};
    
        \draw[fill,black,every node/.style={scale=0.5},thick]
            (1.7,-1.2) circle (1pt) node[above]{$\mathsf{u}({+}1)$}
            (1.722,0.2) circle (1pt) node[below]{$\mathsf{u}({-}1)$};
        \draw[thick,red,densely dashed]
            (1.74,0.2) -- (2.05,0.27) 
            to[out=-44,in=44,distance=0.45cm] (2.05,-1.27);
        \fill[red] 
            (2.05,0.27) circle (1.4pt)
            (2.07,0.5) node{\scriptsize$\bm+$} 
            (2.05,-1.27) circle (1.4pt) (2.12,-1.45) node{\scriptsize$\bm-$} ;
    \end{tikzpicture}
    \caption{The setting of Theorem~\ref{thm:circles-main} and Corollary~\ref{cor:gPhi}. {\color{black}Removing a $d$-ball from a manifold $N$, leaves a boundary component $\Phi$ and turns an embedded circle $\s$ into a neatly embedded arc $\u$. The push-off $\Phi'$ is a $(d-1)$-sphere that intersects $\u$ in two points of opposite sign.}}
    \label{fig:Phi}
\end{figure}
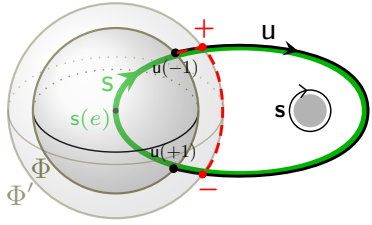

    By the same argument we also find $\lambda(g\bm{\Phi},g)=g\lambda(\bm{\Phi},g)=g(1-\ol{g})=g-1$, so $\lambdabar(g\bm{\Phi},g)=g\pmod{1}$, and by definition $\lambdabar(g,g\bm{\Phi})=(-1)^{d-1}\ol{g}\pmod{1}$. Plugging this into \eqref{eq:gPhi-proof} we have
    \[
        \dax_\u(g\bm{\Phi}) = g - g\ol{\bm{\s}} - g + (-1)^{d-1}\ol{g} = (-1)^{d-1}\ol{g} - g\ol{\bm{\s}}\pmod{1}.\qedhere
    \]
\end{proof}
\begin{remark}
    See \cite[Lem.4.10]{KT-LBT} for a similar computation (when $\u$ is nullhomotopic). Note that the class $g\,\bm{\Phi}\in\pi_{d-1}(N\sm D^d)$ is rarely trivial: in fact, this is the case if and only if $N$ is a simply connected rational homology sphere (in which case $g=\bm{\s}=1$), see Lemma~\ref{lem:QHS}.
\end{remark}
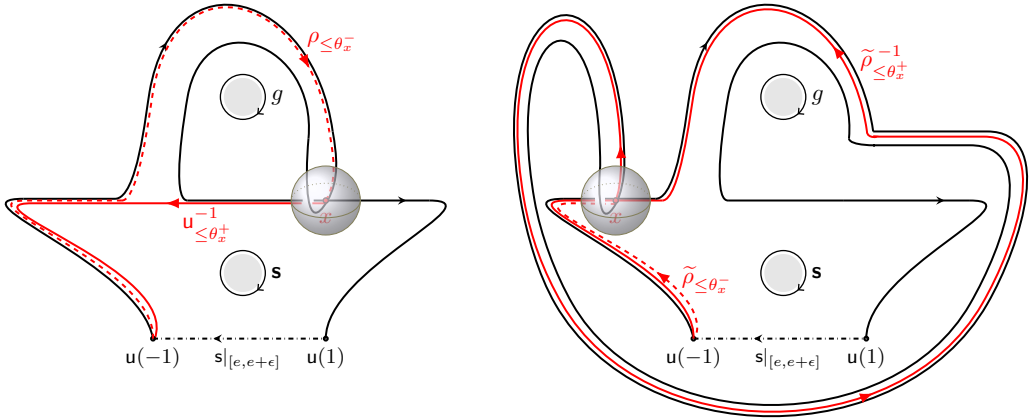
\begin{figure}[!htbp]
    \centering
    \begin{tikzpicture}[scale=0.73,every node/.style={scale=1.4},even odd rule]
    \clip (-1.4,-6.5) rectangle (12,6);
        \draw[black,very thick]
            (-0.75,0.05) -- (1.9,0.05)
            (4,0) -- (7.33,0) 
            (8.05,0) to[out=60,in=-80, distance=0.7cm] (8.2,2);
        \draw[black,very thick,postaction=decorate,decoration = {markings, mark = at position 0.8 with {\arrow{stealth}}}]
            (7.65,0) -- (11,0);
        \draw[black,very thick]
            (-0.75,0.05)
            to[out=180,in=100, distance=2.2cm] (3,-4)
            (11,0) 
            to[out=0,in=90, distance=2cm] (8,-4)
            ;
        \draw[black, dashdotted, very thick,postaction=decorate,decoration = {markings, mark = at position 0.35 with {\arrowreversed{stealth}}}]
            (3,-4) -- (8,-4) node[below,pos=0.55]{\footnotesize$\mathsf{s}|_{[e,e+\epsilon]}$};
        \fill[color=black!10!white] (5.6,-2.1) circle (0.55cm);
        \draw[thick,postaction = decorate,decoration = {markings, mark = at position 0.86 with {\arrowreversed{>}} }] (5.6,-2.1) circle (0.65cm) node[right=0.3cm]{$\bm{\mathsf{s}}$};

        \draw[red,very thick, postaction=decorate,decoration = {markings, mark = at position 0.3 with {\arrow{Latex}}}]
            (8,-0.08) -- (7.61,-0.08) 
            (7.3,-0.08) -- (-0.5,-0.08)
            node[below=-2pt,pos=0.35]{$\mathsf{u}_{\leq\theta^+_x}^{-1}$}
            to[out=180,in=70, distance=2.1cm]
             (3.05,-4);
        \draw[very thick]
            (3,-4) circle (1.6pt) node[below]{\small$\mathsf{u}(-1)$}
            (8,-4) circle (1.6pt) node[below]{\small$\mathsf{u}(1)$};
            
        \draw[black,very thick]
                (7.5,1.8) to[out=-95,in=-120, distance=1.25cm] (8.05,0);
        \draw[red,dashed,very thick]
            (2.1,0) -- (-0.65,0) 
            to[out=180,in=87, distance=2.1cm]
            (3,-4)
            (8,0) to[out=60,in=-80, distance=0.7cm] 
            (8.1,2) 
            (8.2,4.6) node{\textbf{$\rho_{\leq\theta^-_x}$}};
        
        \draw[black,very thick, postaction=decorate,decoration = {markings, mark = at position 0.4 with {\arrow{stealth}}}]
                        (1.9,0.05) to[out=0,in=-100, distance=0.65cm]
                        (2.9,3.15) to[out=80,in=98, distance=4.1cm] (8.2,2);

        \draw[black,very thick]
                        (4,0) to[out=180,in=-100, distance=0.5cm]
                        (4,3) to[out=80,in=85, distance=2.8cm] (7.5,1.8);

        \draw[red,dashed,very thick, postaction=decorate,decoration = {markings, mark = at position 0.84 with {\arrow{Latex}}}]
                        (2.1,0) to[out=0,in=-100, distance=0.5cm]
                        (3,3.15) to[out=80,in=98, distance=4cm] (8.1,2);
                        
        \fill[color=black!10!white] (5.6,3) circle (0.55cm);
        \draw[thick,postaction = decorate,decoration = {markings, mark = at position 0.86 with {\arrowreversed{>}} }] (5.6,3) circle (0.65cm) node[right=0.3cm]{$g$};
        
        \fill[red,draw,very thick]
            (8, -0.1) node[below]{$x$}
            (8, 0) circle (1.7pt);
        \draw[thick] (8, 0) circle (2.1pt);
            
        \foreach \y in {8}{
            \shade[ball color=blue!5!white,opacity=0.65] (\y,0) circle (1cm);
            \draw[black!60!yellow] (\y-1,0) arc (180:360:1cm and 0.5cm);
            \draw[black!60!yellow,dotted] (\y-1,0) arc (180:0:1cm and 0.5cm);
            \draw[black!60!yellow,semithick] (\y,0) circle (1cm);
            }
\end{tikzpicture}~\quad
    \begin{tikzpicture}[scale=0.73,every node/.style={scale=1.4},even odd rule]
    \clip (-3,-6.5) rectangle (13,6);
        \draw[black,very thick]
            (-0.75,0.05) -- (1.9,0.05)
            (4,0) -- (7.33,0) 
            (8.2,2) -- (11.5,2)
            to[out=0,in=60, distance=2cm] (11.5,-3)
            to[out=-120,in=-60, distance=5cm] (-1,-3) to[out=120,in=-80, distance=1cm]
            (-2,0) to[out=100,in=90, distance=6.5cm]
            (1,1)
            ;
        \draw[black,very thick]
            (7.5,1.8)  to[out=-90,in=0, distance=0.23cm] (8.2,1.6) -- (11,1.6)
            to[out=0,in=60, distance=2cm] (11,-3)
            to[out=-120,in=-60, distance=4.5cm] (-0.5,-3) to[out=120,in=-80, distance=1cm]
            (-1.6,0) to[out=100,in=90, distance=5.5cm]
            (0.3,1)
            ;

        \draw[black,very thick,postaction=decorate,decoration = {markings, mark = at position 0.8 with {\arrow{stealth}}}]
            (7.33,0) -- (11,0);
        \draw[black,very thick]
            (-0.75,0.05)
            to[out=180,in=100, distance=2.2cm] (3,-4)
            (11,0) 
            to[out=0,in=90, distance=2cm] (8,-4)
            ;
        \draw[black, dashdotted, very thick,postaction=decorate,decoration = {markings, mark = at position 0.35 with {\arrowreversed{stealth}}}]
            (3,-4) -- (8,-4) node[below,pos=0.55]{\footnotesize$\mathsf{s}|_{[e,e+\epsilon]}$};
        \fill[color=black!10!white] (5.6,-2.1) circle (0.55cm);
        \draw[thick,postaction = decorate,decoration = {markings, mark = at position 0.86 with {\arrowreversed{>}} }] (5.6,-2.1) circle (0.65cm) node[right=0.3cm]{$\bm{\mathsf{s}}$};

        \draw[red,dashed,very thick, postaction=decorate,decoration = {markings, mark = at position 0.68 with {\arrowreversed{Latex}}}]
            (0.75,-0.1) --
            (-0.5,-0.08)
            to[out=180,in=70, distance=2.1cm]
             (3.05,-4);
        \draw[black,very thick]
            (3,-4) circle (1.6pt) node[below]{\small$\mathsf{u}(-1)$}
            (8,-4) circle (1.6pt) node[below]{\small$\mathsf{u}(1)$};
            
        \draw[red,very thick,
        ]
            (2.1,0) -- (-0.65,0) 
            to[out=180,in=87, distance=2.1cm]
            (3,-4)
            (8.55,4) node{\textbf{$\widetilde{\rho}_{\leq\theta^+_x}^{\;-1}$}}
            (3.4,-2.3) node{$\widetilde{\rho}_{\leq\theta^-_x}$};
        
        \draw[black,very thick, postaction=decorate,decoration = {markings, mark = at position 0.4 with {\arrow{stealth}}}]
                        (1.9,0.05) to[out=0,in=-100, distance=0.65cm]
                        (2.9,3.15) to[out=80,in=98, distance=4.1cm] (8.2,2);

        \draw[black,very thick]
                        (4,0) to[out=180,in=-100, distance=0.5cm]
                        (4,3) to[out=80,in=85, distance=2.8cm] (7.5,1.8);

        \draw[red,very thick, postaction=decorate,decoration = {markings,
        mark = at position 0.18 with {\arrowreversed{Latex}},mark = at position 0.5 with {\arrowreversed{Latex}},mark = at position 0.965 with {\arrowreversed{Latex}}}]
                        (2.1,0) to[out=0,in=-100, distance=0.5cm]
                        (3,3.05) to[out=80,in=98, distance=4cm] (8.02,2) to[out=-90,in=0, distance=0.23cm] (8.2,1.85)
                        
            -- (11.4,1.85)
            to[out=0,in=60, distance=2cm] 
            (11.35,-3)
            to[out=-120,in=-60, distance=4.8cm] (-0.8,-3) to[out=120,in=-80, distance=1.1cm]
            (-1.9,0) to[out=100,in=90, distance=6.3cm]
            (0.9,1) to[in=60,out=-90, distance=0.3cm]
            (0.7,-0.1);
                        
        \fill[color=black!10!white] (5.6,3) circle (0.55cm);
        \draw[thick,postaction = decorate,decoration = {markings, mark = at position 0.86 with {\arrowreversed{>}} }] (5.6,3) circle (0.65cm) node[right=0.3cm]{$g$};

        \fill[white]
            (0.15,-0.2) rectangle (0.42,0.2);
        \draw[black,very thick]
                (0.3,1) to[out=-95,in=-120, distance=0.9cm] (0.75,-0.1)
                to[out=60,in=-90, distance=0.3cm] (1,1) ;
                
        \fill[red,draw,very thick]
            (0.75, -0.1) node[below]{$x$}
            (0.75, 0) circle (1.7pt);
        \draw[thick] (0.75, 0) circle (2.1pt);
        
        \foreach \y in {0.75}{
            \shade[ball color=blue!5!white,opacity=0.65] (\y,0) circle (1cm);
            \draw[black!60!yellow] (\y-1,0) arc (180:360:1cm and 0.5cm);
            \draw[black!60!yellow,dotted] (\y-1,0) arc (180:0:1cm and 0.5cm);
            \draw[black!60!yellow,semithick] (\y,0) circle (1cm);
            }
\end{tikzpicture}
    \caption{\color{black}
    \emph{Left.} The $(d-3)$-family of circles $c(g)\coloneqq\partial\realmap(g)\cdot\s|_{[e,e+\epsilon]}$ in $N$. \emph{Right.} After an isotopy of $c(g)$ in $N$, the restriction to $N\sm D^d$ is a family of arcs with $\Dax(\partial\realmap(g)^{new})=(-1)^{d-3}[\widetilde{\rho}_{\leq\theta^-_x}\cdot\widetilde{\rho}_{\leq\theta^+_x}^{-1}]=(-1)^{d-3}\bm\s^{-1}g^{-1}$.}
    \label{fig:explaining-circles}
\end{figure}
{\color{black}
Finally, to understand where the relations $\ol{g}=(-1)^{d-1}g\ol{\bm{\s}}$ come from, consider $\u=\s|_{\S^1\sm[e,e+\epsilon]}$ restricted from $\s\colon\S^1\hra N$ and the class $\partial\realmap(g)$ on $\u$, as in Figure~\ref{fig:explaining-circles}. Then in $N$ we can slide the tip of the finger along $\s^{-1}$, to reach the final position as on the left of the figure. Whereas $\Dax(\partial\realmap(g))=g$ as explained in Figure~\ref{fig:realmap}, for the newly obtained arc family we have $\Dax(\partial\realmap(g)^{new})=(-1)^{d-3}\ol{\bm\s}\ol{g}$. Thus, for $\Dax$ for circles to be well defined we must have $(-1)^{d-3}\ol{\s g}-g=0$, or equivalently $(-1)^{d-1}\ol{h}-h\ol{\s}=0$ for all $h\in\pi_1N$.
}

\subsection*{Outline}
Sections~\ref{intro-sec:examples} and \ref{intro-sec:more-examples} contain examples for $d\geq4$, while Section~\ref{intro-sec:3d} studies $d=3$, the relation to finite type and concordance invariants, and states Theorem~\ref{thm:circles-3d}. In Section~\ref{sec:arcs} we discuss arcs, proving Theorem~\ref{thm:dax-compute} and Corollary~\ref{cor:dax-compute}. In Section~\ref{sec:circles} we discuss circles, proving Theorems~\ref{thm:circles-main}~and~\ref{thm:circles-3d}.

\subsection*{Acknowledgements} 
I wish to thank Peter Teichner for many helpful discussions related to the Dax invariant, which in particular helped me sort out the statement of Theorem~\ref{thm:dax-compute}. Thank you also to Greg Arone for an email correspondence, and Pedro Boavida for comments. A part of this work was done at LAGA, Université Sorbonne Paris Nord, funded by FSMP.

\section{Applications}\label{sec:applications}

\subsection{Simply connected examples}\label{intro-sec:examples}

The part \ref{cor:arcs-simply-conn} of the next corollary is immediate from Theorem~\ref{thm:KT-arcs}. The part \ref{cor:circles-simply-conn} follows from  \ref{cor:arcs-simply-conn} using Corollary~\ref{cor:iso-of-circles-from-arcs} and the fact from Theorem~\ref{thm:circles-main} that the sequence \eqref{eq:B-imm} splits if $\s$ is nullhomotopic.

\begin{cor}\label{cor:1-trivial}
Let $X$ and $N$ be $d$-manifolds with $d\geq4$ and $\partial X\neq\emptyset$.
\begin{enumerate}
\item\label{cor:arcs-simply-conn}
    If $\pi_1X=1$ then embedded and immersed arcs in $X$ cannot be distinguished by $(d-3)$-parameter families: $\pi_{d-3}\Arcs{X}\cong\pi_{d-3}\ImArcs{X}\cong\pi_{d-2}X$.
\item\label{cor:circles-simply-conn}
    Similarly, if $\pi_1N=1$ then the natural inclusion induces an isomorphism
    \[
        \pi_{d-3}\Emb(\S^1,N)\cong\pi_{d-3}\Imm(\S^1,N)\cong \pi_{d-3}N\oplus\pi_{d-2}N.
    \]
\end{enumerate}
\end{cor}
For example, if $d\geq4$ then $\pi_{d-3}\Arcs{\D^d}\cong0$. In fact, Budney shows in \cite[Prop.~3.9]{Budney-family} that the first homotopy group that distinguishes $\Arcs{\D^d}$ and $\ImArcs{\D^d}$ is in degree $2(d-3)$, namely $\pi_{2(d-3)}\Arcs{\D^d}\cong\Z$, while $\pi_{2(d-3)}\ImArcs{\D^d}\cong\pi_{2d-5}\S^{d-1}$ is finite. 
\begin{remark}\label{rem:GW}
    Whereas Budney describes the generator in the corresponding homology group $H_{2(d-3)}(\Arcs{\D^d};\Z)\cong\Z$ (which was computed by Turchin), we can directly write down a generating map
    \[
        \S^{2(d-3)}=\S^{d-3}\wedge\S^{d-3}\to\Arcs{\D^d}.
    \]
    It sends $(\vec{t}_1,\vec{t}_2)$ for $\vec{t}_i\in\S^{d-3}$ to the knot obtained by replacing a subinterval of the basepoint arc $\u$ with the arc obtained as the embedded commutator $[\mu_1(\vec{t}_1),\mu_2(\vec{t}_2)]$. Here $\mu_i\colon\S^{d-2}\hra\D^d$ are two different meridian spheres for $\u$, and $\mu_i(\vec{t}_i)$ are the arcs foliating them. 
    
    This is based on constructions from \cite{K-thesis-paper} and will appear in \cite{K-knotted}, where we more generally use \emph{gropes} of degree $n\geq1$ to give generators of the kernel of the surjection $\pi_{n(d-3)}\ev_n$, for the maps $\ev_n\colon\Arcs{X}\to T_n\Arcs{X}$ to the Goodwillie--Weiss tower. The present example is the case $n=2$, while the realisation map $\realmap$ from Theorem~\ref{thm:KT-arcs} is precisely degree $n=1$. 
    
    In fact, $\Dax$ is precisely the isomorphism induced by $\ev_2$, from $\ker(\pi_{d-3}\ev_1)$ to $\ker(\pi_{d-3}p_2)$, where $p_2\colon T_2\to T_1=\ImArcs{X}$. The latter is isomorphic to the quotient of $\pi_{d-3}\fib(p_2)\cong\Z[\pi_1X]$ by the image of the connecting map $\Omega T_1\to\fib(p_2)$, which agrees with $1\oplus\dax_\u$; cf.\ Remark~\ref{rem:GW-d3}.
\end{remark}

If $N$ is simply connected, Corollary~\ref{cor:circles-simply-conn} says that the {\color{black}lowest homotopy group that could distinguish embedded from immersed circles in $N$, actually does \emph{not}}. In particular, if $N$ is a simply connected 4-manifold, we have $\pi_0\Emb(\S^1,N)=1$ and $\pi_1\Emb(\S^1,N)\cong\pi_2(N)\cong H_2(N;\Z)$. 

This answers in negative a question of Arone and Szymik \cite{AS} whether for a simply connected 4-manifold the inclusion of embedded into immersed circles can fail to be injective on $\pi_1$. Injectivity was shown for a certain class of 4-manifolds by~\cite{Moriya}.

We thank Daniel Ruberman for pointing out the following application of this result. It is a classical theorem of Wall \cite{Wall} that for $N$ a closed simply connected oriented 4-manifold with indefinite $\lambda_N$, every automorphism of the intersection form $(H_2(N\#\S^2\times\S^2),\lambda)$ is \emph{realised by a diffeomorphism} of $N\#\S^2\times\S^2$. To prove this Wall takes an unknot $\s\colon\S^1\hra N$ and a tubular neighbourhood $\nu\s\colon\S^1\times\D^3\hra N$ and considers the composite $\rho$ of the connecting map
\[
    \pi_1(\Emb(\S^1\times\D^3,N),\nu\s)\to\pi_0\Diff_\partial(N\sm\nu\s)
\]
for the fibration sequence $\Diff_\partial(N\sm\nu\s)\to\Diff(N)\to \Emb(\S^1\times\D^3,N)$, and the map
\[
    \pi_0\Diff_\partial(N\sm\nu\s)\to\pi_0\Diff(N\#\S^2\times\S^2),
\]
defined using the inclusion $N\sm\nu\s\hra (N\sm\nu\s)\cup\D^2\times\S^2= N\#\S^2\times\S^2$ and extending by the identity. 
Wall then picks lifts $f_w\in\pi_1(\Emb(\S^1\times\D^3,N),\nu\s)$ of classes $w\in\pi_2N$ under the forgetful maps
\begin{equation}\label{eq:Wall-spherical}
    \pi_1(\Emb(\S^1\times\D^3,N),\nu\s)\to\pi_1(\Emb(\S^1,N),\s)\to\pi_2N.
\end{equation}
More explicitly, each class $f_w$ is given as a self-isotopy of $\s$ obtained by foliating a generically immersed $2$-sphere $S_w\colon\S^2\imra N$ representing $w$. Namely, keep $\s$ fixed near $e\in\S^1$ while the rest of $\S^1$ performs a ``lasso'' move around $S_w$ (so that in every moment of time the lasso is embedded).

However, $w\in\pi_2N\cong H_2(N;\Z)\cong H^2(N;\Z)$ does lift to such a class $f_w$ if and only if it satisfies $w^2=0\in\Z/2\cong H^4(N;\Z/2)$ (this can be seen using the long exact sequence for the fibration $\Emb(\S^1\times\D^3,N)\to\Emb(\S^1,N)$ whose fibre is homotopy equivalent to $\Omega SO(3)$). In fact, if $w^2=1$ then we instead obtain a diffeomorphism $\rho(f_w)\colon N\#\S^2\times\S^2\to N\#\S^2\widetilde{\times}\S^2$, where $\S^2\widetilde{\times}\S^2$ is the nontrivial $\S^2$ bundle over $\S^2$ (conversely, if such a diffeomorphism exists then $\lambda_N$ is odd).

Thus, for every $w$ of even square Wall obtains a diffeomorphism $\rho(f_w)$ of $N\#\S^2\times\S^2$, which he then uses to generate all automorphisms of $\lambda_{N\#\S^2\times\S^2}$.
On the other hand, we note that in \eqref{eq:Wall-spherical} the first map is injective since $\pi_2SO(3)=0$, and the second map is an isomorphism by Corollary~\ref{cor:1-trivial}\ref{cor:circles-simply-conn}. Therefore, we deduce the following observation about Wall's construction.
\begin{cor}
    For a closed simply connected oriented 4-manifold $N$, we have the identification $\pi_1(\Emb(\S^1\times\D^3,N){\color{black},\nu\s)}=\{f_w\mid w\in\pi_2N,\, w^2=0\pmod2\}$, so the image of Wall's homomorphism $\rho\colon\pi_1(\Emb(\S^1\times\D^3,N),\nu\s)\to\pi_0\Diff(N\#\S^2\times\S^2)$ consists precisely of the classes $\rho(f_w)$.
\end{cor}

\subsection{Further examples}\label{intro-sec:more-examples}
 {\color{black}The following corollaries are immediate from Theorems~\ref{thm:KT-arcs} and~\ref{thm:circles-main}. For the first we use that the set of relations $rel_{\s}$ contains $\dax_\u(\pi_{d-1}(N\sm D^d))$ and $\pi_1(N\sm D^d)=\pi_1N$. Recall that we denote $\ol{g}\coloneqq g^{-1}$.}
{\color{black}\begin{cor}
\begin{enumerate}
\item\label{cor:iso-of-circles-from-arcs}
    For the inclusion $\incl_N$ of embedded into immersed circles in $N$ to induce an isomorphism $\pi_{d-3}(\incl_N,\s)$,
    it suffices that the analogous inclusion $\incl_{N\sm D^d}$ for arcs in $N\sm D^d$ induces an isomorphism 
    $\pi_{d-3}(\incl_{N\sm D^d},\u)$.
\item\label{cor:iso-of-circles-and-arcs}
    For the group $\ker\pi_{d-3}(\incl_N,\s)$ to be isomorphic to $\ker\pi_{d-3}(\incl_{N\sm D^d},\u)=\ker\pi_{d-3}p_\u$ it suffices that $\s$ is nullhomotopic or that $\{c\in\pi_{d-2}N \mid c=\bm{\s}\,c\}=\{0\}$.
\end{enumerate}
\end{cor}}

\begin{cor}\label{cor:d-1-trivial}
Let $X$ and $N$ be $d$-manifolds with $d\geq4$ and $\partial X\neq\emptyset$.
\begin{enumerate}
    \item\label{cor:arcs-d-1-conn}
    If $\pi_{d-1}X=0$ then there is a short exact sequence
    \[
    \Z[\pi_1X\sm1]\hra\pi_{d-3}(\Arcs{X},\u)\twoheadrightarrow\pi_{d-2}X.
    \]
    \item\label{cor:circles-d-1-conn}
    If $\pi_{d-2}N=\pi_{d-1}N=0$ then there is a short exact sequence
    \[
    \faktor{\Z[\pi_1N]}{\langle1\rangle\oplus\langle \ol{g}-(-1)^{d-1}g\ol{\bm{\s}}\mid g\in\pi_1N\rangle}\hra\pi_{d-3}(\Emb(\S^1,N),\s)\twoheadrightarrow\pi_{d-3}N.
    \]
\end{enumerate}
\end{cor}
In particular, if $X$ is aspherical, i.e.\ $\pi_{*>1}X=0$, then $\pi_{d-3}\Arcs{X}\cong\Z[\pi_1X\sm1]$. For example,
\[
    \pi_{d-3}(\Arcs{\S^1\times\D^{d-1}},\u)\cong\Z[\Z\sm 0],\quad \text{for any } \u.
\]
Similarly, in the case of circles we can take an aspherical (closed or not) manifold $N$, for example
\[
    \pi_{d-3}(\Emb(\S^1,\S^1\times\D^{d-1}),t^{k_0})\cong \faktor{\Z[t,t^{-1}]}{\langle t^0, t^{-k}-(-1)^{d-1}t^{k-k_0}\mid k\in\Z\rangle}.
\]
{\color{black}As another example,} a compact irreducible 3-manifold $Y$ with infinite $\pi_1Y$ is aspherical (by the Hurewicz theorem applied to its (noncompact) universal cover), so  $N\coloneqq\S^1\times Y$ is as well and
\[
    \pi_1(\Emb(\S^1,\S^1\times Y),\s)\cong \faktor{\Z[\Z\times\pi_1Y]}{\langle1\rangle\oplus\langle \ol{g}-(-1)^{d-1}g\ol{\bm{\s}}\mid g\in\Z\times\pi_1Y\rangle}.
\]
The following is a next simplest case, also computed in \cite{Budney-Gabai} by different means. 
\begin{cor}\label{cor:BG}
    Let $\s=t^{W_0}\in\Z\{ t\} \cong\pi_1(\S^1\times\S^{d-1})$ with $W_0\geq0$. Then we have for $d\geq5$:
    \begin{equation}\label{eq:BG}
    \pi_{d-3}(\Emb(\S^1,\S^1\times\S^{d-1}),\s)\cong\faktor{\Z[t,t^{-1}]}{\langle{\color{black}t^0}, t^{-1}+\dots+t^{-(W_0-1)},\,t^{-k}-(-1)^{d-1}t^{k-W_0}\mid k\in\Z\rangle}
    \end{equation}
    while for $d=4$ the group $\pi_1(\Emb(\S^1,\S^1\times\S^3),\s)$ is the sum of the displayed group with $\Z$.
\end{cor}
\begin{proof}
    Denote $N\coloneqq\S^1\times\S^{d-1}$ and $X\coloneqq N\sm D^d$. As $\pi_n(N)=0$ for $2\leq n\leq d-2$, Theorem~\ref{thm:circles-main} implies $\pi_n(\Imm(\S^1,N),\s)=0$ for $0\leq n\leq d-3$, except that $\pi_1(\Imm(\S^1,N),\s)\cong\pi_1N$. 
    
    Moreover, $\pi_{d-1}X\cong\pi_{d-1}(\S^1\vee\S^{d-1})\cong\Z[t,t^{-1}]$, generated by $i_2\colon \{pt\}\times\S^{d-1}\subseteq X\subseteq N$. However, a more convenient generator is $\Phi\colon\partial D^d\hra X$, and the long exact sequence of the pair $(N,X)$ yields
    \[
        \pi_{d-1}X\cong\pi_{d-1}N\oplus \partial_*\pi_d(N,X)\cong\Z\langle \bm{i_2}\rangle\, \oplus\, \Z[t,t^{-1}]\langle \bm{\Phi}\rangle.
    \]
    Thus, if $d-3>1$ then Theorem~\ref{thm:circles-main} implies (using ${\color{black}\pi_{d-3}(\Imm(\S^1,N),\s)=0}$):
    \[
        \pi_{d-3}\Emb(\S^1,N)\cong\faktor{\langle t^k,t^{-k}\mid k>0\rangle}{\langle \dax_\u(\bm{i_2}),\dax_\u(t^k\bm{\Phi})\mid k\in\Z\rangle}
    \]
    while for $d-3=1$ we have this plus the group $\Z\cong\pi_1(\Imm(\S^1,N),\s)$. It is straightforward to compute $\lambda(\bm{i_2},t^{W_0})=1+t^{-1}+\dots+t^{-(W_0-1)}$, so $\dax_\u(\bm{i_2})=\lambdabar(\bm{i_2},t^{W_0})=t^{-1}+\dots+t^{-(W_0-1)}$ {\color{black}by Corollary~\ref{cor:dax-compute}}, while Corollary~\ref{cor:gPhi} gives $\dax_\u(t^k\bm{\Phi})=(-1)^{d-1}t^{-k}-t^k\bm{\s}^{-1}=(-1)^{d-1}t^{-k}-t^{k-W_0}$.
\end{proof}
\begin{remark}
    Note that the resulting abelian group in \eqref{eq:BG} is simply a $\Z^\infty$, {\color{black}unless $d$ and $W_0$ are both odd, when it is $\Z/2\oplus\Z^\infty$. Namely}, the generators $t^{-k}$ are redundant for $k> W_0/2$ by the second relation, {\color{black}which also gives} $(-1)^{d-1}t^{-W_0/2}=t^{-W_0/2}$ if $W_0$ even. This is empty unless $d$ is also even, in which case it says that $t^{-W_0/2}$ is 2-torsion. {\color{black}However, if $d$ and $W_0$ are even $t^{-W_0/2}$ vanishes by the first relation (for $W_0$ odd it is empty). On the other hand, if $d$ is odd, the first relation says that for $W_0$ even $t^{-W_0/2}$ is redundant, whereas for $W_0$ odd $t^{-1}+\dots+t^{-(W_0-1)/2}$ is 2-torsion}. 
    We just mention that one can get a similar result for $W_0<0$ (using $\lambda(\bm{i_2},t^{W_0})=-(t+\dots+t^{-W_0})$). 
    
    Finally, let us observe that $\pi_{d-2}(\S^1\times\S^{d-1})=0$, so Corollary~\ref{cor:iso-of-circles-and-arcs} applies, giving an isomorphism: $\pi_{d-3}(\Emb(\S^1,\S^1\times\S^{d-1}\sm\D^d),\u)\cong\pi_{d-3}(\Emb(\S^1,\S^1\times\S^{d-1}),\s)$.
\end{remark}

{\color{black}
\begin{remark}
    Let us compare Corollary~\ref{cor:BG} to the computation of Budney and Gabai in \cite[Sec.2]{Budney-Gabai}. Their group $\pi_{d-3}(\Emb(\S^1,\S^1\times\S^{d-1}),t^{W_0})\cong\Lambda_n^{W_0}$ is described as the quotient of the Laurent polynomial ring $\Z[t,t^{-1}]$ by a different set of relations, namely $\langle t^k-(-1)^d t^{W_0-1-k},t^0,t^1\rangle$. However, $\varphi(t^k)=t^{k+W_0}-t^{k+W_0-1}$ is a well-defined map from our group to $\Lambda_n^{W_0}$, and an explicit inverse $\psi$ is given by: $\psi(t^k)\coloneqq t^{k-W_0}+\dots+t^1+t^{0}$ for $k\geq W_0$, $\psi(t^k)\coloneqq -t^{k-(W_0-1)}-\dots-t^{-1}-t^0$ for $0\leq k\leq W_0-1$, and $\psi(t^k)\coloneqq (-1)^d \psi(t^{W_0-1-k})$ for $k\leq0$.
    
    Moreover, their class $\theta_{t^{W_0},k}\in\pi_{d-3}(\Emb(\S^1,\S^1\times\S^{d-1}),t^{W_0})$ is related to our realization map: the class $\partial\realmap(g)^{new}$ on the right of Figure~\ref{fig:explaining-circles} is up to a sign (not fixed in \cite{Budney-Gabai}) the class $\theta_{t^{W_0},k}$ for $s=t^{W_0}$ and $t^k=gs$ (the guiding arc for the finger), so $g=t^{k-W_0}$. Indeed, in their family the double point occurs \emph{before} the root of the finger. Whereas the cocircular method gives $W_2(\theta_{t^{W_0},k})=t^k-t^{k-1}$, we have $\Dax(\partial\realmap(t^{k-W_0})^{new})=(-1)^{d-3}t^{-k}=t^{k-W_0}$ as explained in that figure. Since $\varphi(t^{k-W_0})=t^k-t^{k-1}$, the two computations agree.
\end{remark}  
\begin{remark}
    As a further example, consider the ``self-referential'' family from \cite[Fig.13-15]{Gabai-disks}. Firstly, Gabai's $\tau_g$ is precisely our $\partial\realmap_u(g)$, and the foliation $\mathsf{selfref}_u(g)$ of his self-referential disk $D_g$ is the family depicted in Figure~\ref{fig:self-referential}.
    This is given similarly as $\partial\realmap(g)$, except that after following a guiding arc representing $g$, we come back and swing around the sphere $S$ linking the root of the finger, instead of just the basepoint arc. Therefore, the homotopy which pulls back using the ball that $S$ bounds produces two double points. The associated loops are computed in Figure~\ref{fig:self-referential}: $g$ at $x_2$ as before, and $-(-1)^{d-1}g^{-1}$ at $x_1$ since not only the guiding arc is reversed (giving $-1$) but also now the second sheet is the one moving (giving $(-1)^{d-1}$), cf.\ the proof of Corollary~\ref{cor:gPhi}. Thus,
    \[
        \Dax(\mathsf{selfref}_u(g))=g-(-1)^{d-1}\ol{g}.
    \]
    Indeed, Gabai has $d=4$ and $\Dax(\mathsf{selfref}_u(g))=g+\ol{g}$.
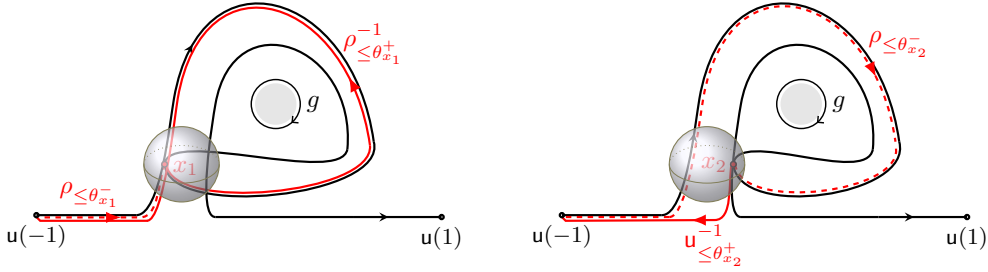
\begin{figure}[!htbp]
    \centering
    \begin{tikzpicture}[scale=0.73,every node/.style={scale=1.4},even odd rule]
    \clip (-1.7,-1.5) rectangle (11.6,6);
        \draw[black,very thick]
            (-0.75,0.05) -- (1.9,0.05)
            (4,0) -- (7.65,0) 
            ;
        \draw[black,very thick,postaction=decorate,decoration = {markings, mark = at position 0.4 with {\arrow{stealth}}}]
            (7.65,0) -- (10,0);

        \draw[red,very thick]
            (2.2,-0.1) -- (-0.65,-0.1) -- (-0.75,0)
             (2.7,1.4) to[out=-85,in=-85, distance=1.3cm] (8.1,1.9)
            (8.2,4.6) node{\textbf{$\rho_{\leq\theta^+_{x_1}}^{-1}$}};
        \draw[very thick]
           (-0.75,0.05) circle (1.6pt) node[below]{\small$\mathsf{u}(-1)$}
           (10,0)  circle (1.6pt) node[below]{\small$\mathsf{u}(1)$};
            
        \draw[black,very thick]
                (2.65,1.4) to[out=-85,in=-75, distance=1.5cm] (8.2,2)
                (2.65,1.4) to[out=85,in=-100, distance=1cm] (7.5,1.8);
        \draw[red,dashed,very thick, postaction=decorate,decoration = {markings, mark = at position 0.5 with {\arrowreversed{Latex}}}]
            (2.65,1.4) to[out=-110,in=0, distance=0.5cm]
            (2,-0.01) -- (-0.65,-0.01) 
            node[above,pos=0.5]{$\rho_{\leq\theta^-_{x_1}}$}
            ;
        
        \draw[black,very thick, postaction=decorate,decoration = {markings, mark = at position 0.4 with {\arrow{stealth}}}]
                        (1.9,0.05) to[out=0,in=-100, distance=0.65cm]
                        (2.9,3.15) to[out=80,in=98, distance=4.1cm] (8.2,2);

        \draw[black,very thick]
                        (4,0) to[out=180,in=-100, distance=0.5cm]
                        (4,3) to[out=80,in=85, distance=2.8cm] (7.5,1.8);

        \draw[red,very thick, postaction=decorate,decoration = {markings, mark = at position 0.84 with {\arrowreversed{Latex}}}]
                (2.18,-0.1) to[out=0,in=-100, distance=0.45cm]
                (3,3.15) to[out=80,in=98, distance=4cm] (8.1,1.9);
                        
        \fill[color=black!10!white] (5.6,3) circle (0.55cm);
        \draw[thick,postaction = decorate,decoration = {markings, mark = at position 0.86 with {\arrowreversed{>}} }] (5.6,3) circle (0.65cm) node[right=0.3cm]{$g$};
        
        \fill[red,draw,very thick]
            (2.6, 1.3) node[right]{$x_1$}
            (2.65, 1.4) circle (1.7pt);
        \draw[thick] (2.65, 1.4) circle (2.1pt);
            
        \foreach \y in {3.1}{
            \shade[ball color=blue!5!white,opacity=0.65] (\y,1.4) circle (1cm);
            \draw[black!60!yellow] 
            (\y-1,1.4) arc (180:360:1cm and 0.5cm);
            \draw[black!60!yellow,dotted] (\y-1,1.4) arc (180:0:1cm and 0.5cm);
            \draw[black!60!yellow,semithick] (\y,1.4) circle (1cm);
            }
\end{tikzpicture}~\quad
\begin{tikzpicture}[scale=0.73,every node/.style={scale=1.4},even odd rule]
    \clip (-1.7,-1.5) rectangle (11.6,6);
        \draw[black,very thick]
            (-0.75,0.05) -- (1.9,0.05)
            (4,0) -- (7.65,0) 
            ;
        \draw[black,very thick,postaction=decorate,decoration = {markings, mark = at position 0.4 with {\arrow{stealth}}}]
            (7.65,0) -- (10,0);

        \draw[red,very thick, postaction=decorate,decoration = {markings, mark = at position 0.4 with {\arrow{Latex}}}]
            (3.75,1.4) to[out=-100,in=0, distance=0.5cm]
            (3.3,-0.08)  node[below=-2pt]{$\mathsf{u}_{\leq\theta^+_{x_2}}^{-1}$} -- (-0.65,-0.1) -- (-0.75,0)
            
            ;
        \draw[very thick]
            (-0.75,0.05) circle (1.6pt) node[below]{\small$\mathsf{u}(-1)$}
            (10,0)  circle (1.6pt) node[below]{\small$\mathsf{u}(1)$};
            
        \draw[black,very thick]
                (3.8,1.4) to[out=-85,in=-75, distance=1.5cm] (8.2,2)
                (3.8,1.4) to[out=85,in=-100, distance=1cm] (7.5,1.8);
        \draw[red,dashed,very thick]
            (2.1,0) -- (-0.65,0)
            (3.85,1.4) to[out=-85,in=-75, distance=1.3cm] (8.1,1.9)
            (8.2,4.6) node{\textbf{$\rho_{\leq\theta^-_{x_2}}$}};
        
        \draw[black,very thick, postaction=decorate,decoration = {markings, mark = at position 0.2 with {\arrow{stealth}}}]
                        (1.9,0.05) to[out=0,in=-100, distance=0.65cm]
                        (2.9,3.15) to[out=80,in=98, distance=4.1cm] (8.2,2);

        \draw[black,very thick]
                        (4,0) to[out=180,in=-100, distance=0.5cm]
                        (4,3) to[out=80,in=85, distance=2.8cm] (7.5,1.8);

        \draw[red,dashed,very thick, postaction=decorate,decoration = {markings, mark = at position 0.85 with {\arrow{Latex}}}]
                (2.1,0) to[out=0,in=-100, distance=0.5cm]
                (3,3.15) to[out=80,in=98, distance=4cm] (8.1,2);
                        
        \fill[color=black!10!white] (5.6,3) circle (0.55cm);
        \draw[thick,postaction = decorate,decoration = {markings, mark = at position 0.86 with {\arrowreversed{>}} }] (5.6,3) circle (0.65cm) node[right=0.3cm]{$g$};
        
        \fill[red,draw,very thick]
            (3.9, 1.4) node[left]{$x_2$}
            (3.8, 1.4) circle (1.7pt);
        \draw[thick] (3.8, 1.4) circle (2.1pt);
            
        \foreach \y in {3.1}{
            \shade[ball color=blue!5!white,opacity=0.65] (\y,1.4) circle (1cm);
            \draw[black!60!yellow] 
            (\y-1,1.4) arc (180:360:1cm and 0.5cm);
            \draw[black!60!yellow,dotted] (\y-1,1.4) arc (180:0:1cm and 0.5cm);
            \draw[black!60!yellow,semithick] (\y,1.4) circle (1cm);
            }
\end{tikzpicture}
    
    \caption{\color{black}\emph{Left.} The double point loops $g_{x_1}=[\rho_{\leq\theta_{x_1}^-}\cdot\rho_{\leq\theta_{x_1}^+}^{-1}]$ and $g_{x_2}=[\rho_{\leq\theta_{x_2}^-}\cdot\rho_{\leq\theta_{x_2}^+}^{-1}]=[\rho_{\leq\theta_{x_2}^-}\cdot\u_{\leq\theta_{x_2}^+}^{-1}]$.}
    \label{fig:self-referential}
\end{figure}

    In particular, for $X=\S^1\times\S^{d-1}\sm D^d$ the previous remark gives $\varphi\Dax(\mathsf{selfref}_{t^{W_0}}(t^k))=\varphi(t^k-(-1)^{d-1}t^{-k})=t^{k+W_0}-t^{k+W_0-1}+(-1)^d(t^{-k+W_0}-(-1)^dt^{-k+W_0-1})=t^{k+W_0}-t^{k+W_0-1}+t^{k-1}-t^{k}$.
    Compare to a similar family $\alpha_{t^{W_0},k}\in\pi_{d-3}(\Emb(\S^1,\S^1\times\S^{d-1}),t^{W_0})$ in \cite[Fig.4,Prop.2.6]{Budney-Gabai}.
\end{remark}
}

For the next example, observe that $\lambdabar(a,g)$ vanishes if $X=\D^k\times Y$ for $1\leq k\leq d$.
\begin{cor}
    Given a $(d-k)$-dimensional manifold $Y$ for $k\geq1$, let $\mathcal{S}(Y)$ denote a collection of $\Z[\pi_1Y]$-generators for $\pi_{d-1}Y$. Then
    \[
        \faktor{\Z[\pi_1Y\sm1]}{\langle g\dax(a)\ol{g}\mid a\in\mathcal{S}(Y),g\in\pi_1Y\rangle}\hra \pi_{d-3}\big(\Arcs{\D^k\times Y},1\big) \sra \pi_{d-2}Y.
    \]
\end{cor}
\begin{remark}\label{rem:KTcor}
    For a 4-manifold $X=\D^k\times Y$ the cases $k=2,3,4$ are {\color{black}completely computed} by the discussion so far. For $k=1$ we can say more thanks to \cite[Thm.B and Cor.1.6]{KT-LBT}.
    If a 3-manifold $Y$ has $H_3\wt{Y}=0$, then every $a\in\pi_3X$ is the composite of the Hopf map $\S^3\to\S^2$ and a map $S\colon\S^2\to X$, and we have $\dax(a)=\lambdabar(S,S)$. This vanishes since any $S$ can be pushed off itself in the $\D^1$-direction.
    Therefore, for $X=\D^1\times Y$ we have $\dax=0$ and the exact sequence as in Corollary~\ref{cor:d-1-trivial}\ref{cor:arcs-d-1-conn},
    unless $Y$ is closed with finite fundamental group. In that case, we have $\pi_3X\cong H_3\wt{Y}\cong \Z$ generated by the map $a_Y\colon\S^3=\wt{Y}\to\{pt\}\times Y\subseteq X$, and computing $\dax(a_Y)$ remains an interesting problem!
\end{remark}

\subsection{Arcs in dimension three}\label{intro-sec:3d}
For $d=3$ we consider the induced map on sets of components $\pi_0\incl_X\colon\pi_0\Arcs{X}\to\pi_0\ImArcs{X}\cong\pi_1X$, so $\ker(\pi_{d-3}(\incl_X,\u))$ translates into the set
\[
    \KK(X,\bm{\u})\coloneqq (\pi_0\incl_X)^{-1}(\bm{\u})\subseteq\pi_0\Arcs{X},
\]
of those isotopy classes of (long) knots $K\colon\D^1\hra X$ in the $3$-manifold $X$ homotopic to the fixed knot $\u\colon\D^1\hra X$. Counting double point loops during a generic homotopy from $K$ to $\u$ still gives
\begin{equation}\label{eq:Dax-dim3}
    \Dax{\color{black}=\Dax_\u}\colon\KK(X,\bm{\u})\sra \faktor{\Z[\pi_1X]}{\dax_\u(\pi_2X)}
\end{equation}
which is now only a map of sets, see~(\ref{eq:Dax-def}).
However, it is still surjective (perform crossing changes along given group elements) and the related map $\dax_\u\colon\pi_2X\to\Z[\pi_1X]$, which applies $\Dax_\u$ to self-homotopies, is a homomorphism of abelian groups (see Lemma~\ref{lem:Dax-dim3}). Additionally, in this 3-dimensional setting Corollary~\ref{cor:dax-compute} shows particularly useful thanks to the following classical result.
\begin{theorem}[Corollary of the Sphere Theorem, see {\cite[Prop.~3.12]{Hatcher-3}}]\label{thm:cor-sphere-thm}
    In a compact orientable 3-manifold $X$ there exists a finite collection $\mathcal{S}(X)$ of disjointly embedded $2$-spheres generating $\pi_2X$ as a $\Z[\pi_1X]$-module. 
    Each sphere in the collection is either a connected sum sphere in the prime decomposition of $X$, or coming from a $\S^1\times\S^2$ factor, or it is parallel to a sphere in $\partial X$. 
\end{theorem}

\begin{cor}\label{cor:dim3}
    For a $3$-manifold $X$ with $\mathcal{S}(X)$ as in the theorem, there is a surjection of sets
    \begin{equation}\label{eq:pi0ev2}
       \begin{tikzcd}
            \Dax_\u\colon\KK(X,\bm{\u}) \ar[two heads]{r}{} & \faktor{\Z[\pi_1X\sm1]}{\langle\lambdabar(ga,\bm{\u})-\lambdabar(ga,g)+\ol{\lambdabar(ga,g)} \mid \,g\in\pi_1X,\,a\in \mathcal{S}(X) \,\rangle}
        \end{tikzcd}
    \end{equation}
    In particular, $\dax_\u(\pi_2X)$ depends only on the homotopy class $\bm{\u}$ (but $\Dax_\u$ could depend on~$\u$).
\end{cor}
\begin{proof}
    This follows immediately from Lemma~\ref{lem:Dax-dim3}, Corollary~\ref{cor:dax-compute}\ref{eq:cor5} and Theorem~\ref{thm:cor-sphere-thm}.
\end{proof}

\begin{theorem}\label{thm:Vassiliev}
    The map $\Dax_\u$ is a \emph{universal} Vassiliev type $\leq1$ invariant {\color{black}of knotted arcs in $X$}.
\end{theorem}
\begin{proof}
An invariant $v\colon\KK(X,\bm{\u})\to A$, with values in an abelian group $A$, is of type $\leq1$ if and only if for an arc $J$ with a single double point $v(res(J))\coloneqq v(J^+)-v(J^-)$ depends only on the isotopy class of $J$, where $J_{\pm}$ are two knots obtained by resolving the double point. In particular, this is true for $v=\Dax_\u$ since $\Dax_\u(res(J))$ is the homotopy class of the double point loop.

To show that $\Dax_\u$ is universal, we will find for any $v$ as above a homomorphism $w_v\colon\im(\Dax_\u)\to A$ such that $v(K)=v(\u)+w_v\circ\Dax_\u(K)$ for all $K\in\KK(X,\bm{\u})$. 
We simply define $w_v$ by linear extension, $w_v(g)\coloneqq v(res(J_g))$ for an immersed arc $J_g$ with a single double point and loop $g=\Dax_\u(J_g^+)$. A homotopy from $K$ to $\u$ is a sequence of crossing changes, so $K-\u=\sum_ires(J_i)$, where $J_i$ is the immersed arc occurring during the $i$-th crossing change. Thus, $v(K)-v(\u)=\sum_iv(res(J_i))=w_v\circ\Dax_\u(K)$. To see that $w_v$ vanishes on relations $R_1-R_2=0$ defining the target of $\Dax_\u$, {\color{black}i.e.\ $R_1-R_2\in\dax_\u(\pi_2X)$}, recall that each such relation arises from having two different homotopies $h_i$ from $K$ to $\u$, with $\Dax_\u(h_i)=R_i$, so $w_v(R_1)=w_v\circ\Dax_\u(K)=w_v(R_2)$.
\end{proof}
\begin{remark}\label{rem:GW-d3}
    We saw in Remark~\ref{rem:GW} that $\Dax$ agrees with the second Goodwillie--Weiss knot invariant $\pi_0\ev_{2}$ (and $\ev_1=\incl_X$). In fact, $\pi_0\ev_{n}$ is invariant under the $(n-1)$-equivalence relation of knots due to Gusarov and Habiro \cite{K-thesis-paper}. For $M=\D^3$ this is equivalent to being of Vassiliev type $\leq n-1$, but in general 3-manifolds the analogous connection is an open problem for $n\geq3$.
\end{remark}

Finally, we study the isotopy invariant $\Dax_\u$ modulo the relation of concordance of knotted arcs.
\begin{defn}
    Two knots $K_0,K_1\in\Arcs{X}$ are \emph{concordant} if there exists a neat embedding $Q\colon \D^1\times[0,1]\hra X\times[0,1]$ with $Q|_{\D^1\times\{i\}}=K_i$ and $Q(\partial \D^1\times\{t\})=\u(\partial \D^1)\times\{t\}\subseteq\partial X\times\{t\}$.
    Let $\mathcal{C}(X,\bm{\u})$ be the set of concordance classes of knotted arcs in $X$ which are \emph{homotopic} to~$\u$.
\end{defn}
    The next theorem says that for $\u\in\Emb_\partial(\D^1,X)$ homotopic into $\partial X$ rel.\ endpoints, $\Dax_\u$ reduces to Schneiderman's \cite{Schneiderman-algebraic} concordance invariant
    \[
    \mu_2\colon\mathcal{C}(X,\bm{\u})\to\faktor{\Z[\pi_1X\sm1]}{\langle \ol{g}-g\mid g\in\pi_1X\rangle}.
    \]
    For a concordance class $[K]\in\mathcal{C}(X,\bm{\u})$ this is defined as Wall's self-intersection invariant $\mu_2(H)$ of any \emph{immersed} concordance $H\colon \D^1\times[0,1]\imra X\times[0,1]$ from $K$ to $\u$, e.g.\ take for $H$ the trace of any homotopy $h$ from $K$ to $\u$. The count of double point loops in $H$ is well defined only modulo $\ol{g}-g$, as there is no preference between the two intersecting sheets of $H$. For a proof that $\mu_2$ is well defined see \cite[Thm.1]{Schneiderman-algebraic}, where the analogous invariant was defined for \emph{nullmotopic circles}.
    
\begin{theorem}\label{thm:concordance}
    For a compact orientable $3$-manifold $X$, \emph{if $\u\in\Emb_\partial(\D^1,X)$ is homotopic into $\partial X$ rel.\ endpoints}, then there is a commutative diagram of sets
    \[\begin{tikzcd}[column sep=large]
        \KK(X,\bm{\u}) \ar[two heads]{d}{\Dax_\u}\ar[two heads]{r} 
        & 
        \mathcal{C}(X,\bm{\u}) \ar[two heads]{d}{\mu_2}
        \\
        \faktor{\Z[\pi_1X\sm1]}{\langle \ol{\lambdabar(ga,g)}- \lambdabar(ga,g)\mid g\in\pi_1X,a\in \mathcal{S}\rangle}\ar[two heads]{r}{q} 
        &
        \faktor{\Z[\pi_1X\sm1]}{\langle \ol{g}-g\mid g\in\pi_1X\rangle}
    \end{tikzcd}
        \]
\end{theorem}
\begin{proof}[Proof of Theorem~\ref{thm:concordance}]
    The proof is the same as in \cite[Thm.5.11]{KT-LBT}:
    for $K\in\Arcs{X}$ and a homotopy $h$ from $K$ to $\u$, we defined the class $\Dax_\u(K)=\Dax_\u(h)\in\Z[\pi_1X]$ as the sum of signed double point loops occurring in $h$. This is equal to $\mu_2(H)$ for one choice of sheets: this is the sum of signed double point loops occurring in $h$ when seen as an immersed annulus $H$ in $X\times[0,1]$.
\end{proof}

In other words, $\lambda(ga,g)$ becomes zero if calculated in $X\times[0,1]$ since we can push $ga$ and $g$ off each other using the $[0,1]$-direction; however, in this setting we have to mod out $g-\ol{g}$ since in a concordance we have forgotten the foliations by arcs. As an example, in $X=N\sm\D^3$ the type $\leq1$ invariant $\Dax_\u$ simply agrees with $\mu_2$, since we have $\Phi\colon\S^2=\partial\D^3\hra X$ for which $\lambdabar(g\bm{\Phi},g)=g$ (see Corollary~\ref{cor:gPhi}). However, if $X$ is an irreducible 3-manifold then the target of $\Dax_\u$ is $\Z[\pi_1X\sm1]$, so it is an isotopy invariant which is in general strictly stronger than the concordance invariant $\mu_2$.

\subsection{Circles in dimension three}
Consider now the set $\KK(N,\bm{\s})\coloneqq(\pi_0\incl_N)^{-1}(\bm{\s})$ consisting of isotopy classes of knots $\S^1\hra N$ which are in the free homotopy class $\bm{\s}$ of a fixed knot $\s$. Denote by $\zeta(\bm{\s})\coloneqq\{b\in\pi_1N\mid b=\bm{\s}\,b\}$ the centralizer of $\bm{\s}\in\pi_1N=\pi_1(N,\s(e))$.
\begin{mainthm}\label{thm:circles-3d}
    There is a map $\Dax_\s$ from the set $\KK(N,\bm{\s})$ to the quotient of the group
    \[\faktor{\Z[\pi_1N\sm1]}{
        \big\langle
        \ol{g}-g\ol{\bm{\s}},\; 
        \lambdabar(ga,\bm{\s})-\lambdabar(ga,g)+\ol{\lambdabar(ga,g)}\mid g\in\pi_1N, a\in \mathcal{S}(N\sm\D^3)\;
         \,\big\rangle
        }
    \]
    by the following set-theoretic action of $\zeta(\bm{\s})$: let $b\in\zeta(\bm{\s})$ act by $r \mapsto b\cdot r\cdot b^{-1} + \Dax_\s(\delta^{whisk}_{\s}(b))$.
\end{mainthm}

\begin{remark}\label{rem:Vassiliev-circles}
    The proof of Theorem~\ref{thm:Vassiliev} together with Remark~\ref{rem:Vassiliev-circles2} implies that the invariant $\Dax_\s$ of from Theorem~\ref{thm:circles-3d} is a \emph{universal Vassiliev invariant of type $\leq1$ of knots in $N$}. For example, the type $\leq1$ invariants from \cite{Kirk-Livingston} factor through $\Dax$, but one can now construct many more!
\end{remark}

Recall from Theorem~\ref{thm:circles-main} that $\delta^{whisk}_{\s}(b)$ is the knot obtained by dragging a piece of $\s$ around a loop $\beta$ representing $b$. Thus, $\delta^{whisk}_{\s}(b)\simeq \beta\s\beta^{-1}\simeq\s$ and $\Dax(\delta^{whisk}_{\s}(b))$ counts double points in a based homotopy witnessing this, i.e.\ in a disk bounded by (an embedded version of) the commutator~$[\beta,\s]$.

In particular, if $\s$ is freely nullhomotopic then all $\delta^{whisk}_{\s}(b)$ and $\lambdabar(ga,\bm{\s})$ vanish, so we are only modding out by $\ol{g}-g$ and conjugation. The same argument as in Theorem~\ref{thm:concordance} then shows that $\Dax_\s$ (itself!) agrees with Schneiderman's invariant $\mu_2$ for $\mathcal{C}(N,\bm{\s})$ {\color{black}in this case}. One could also prove that $\Dax_\s$ reduces to $\mu_2$ modulo concordance for any $\s$, but we omit this as the target of $\mu_2$ is in general somewhat complicated, see \cite{Schneiderman-algebraic}. We just note that $\Dax_\s(\delta^{whisk}_{\s}(\bm{\s}^n))=0$ for all $n$.

\begin{remark}
    Our set of relations should be compared to Schneiderman's generating set of relations $\Phi(\s)$ from \cite[Sec.4.7]{Schneiderman-algebraic}, and our action of the centralizer $\zeta(\bm{\s})$ with his in \cite[Sec.4.7.2]{Schneiderman-algebraic}. By \cite[Prop.5.4.1]{Schneiderman-algebraic} the set $\Phi(\s)$ has a part coming from spheres (corresponding to our relations coming from $\mathcal{S}(X)$), and a ``toroidal'' part (note that if $b\in\zeta(\bm{\s})$ then there is an immersed torus $\S^1\times\S^1\imra N$ with $\s=\S^1\times pt$). As observed there, it is only this latter part that is sensitive to the isotopy class of $\s$, while the rest depends only on its homotopy class $\bm{\s}$. This is also the case in our relations, where only $\Dax_\s(\delta^{whisk}_{\s}(b))$ can depend on $\s$.
\end{remark}

\subsection{Open problems}

There are some natural questions that to our knowledge remain open.
\begin{question}\label{question1}
    Does the image of $\dax_\u$ depend only on the homotopy type of a manifold $X$?
\end{question}
\begin{question}\label{question2}
    What is the isomorphism class of the central extension \eqref{eq:thm1}? Does it depend only on the homotopy type of a 4-manifold $X$?
\end{question}
We saw {\color{black}in Corollary~\ref{cor:dim3}} that the answer to the first question is affirmative for $d=3$. However, we suspect that there is an example of a pair of homotopy equivalent nondiffeomorphic 4-manifolds $X,X'$ for which the extensions~\eqref{eq:thm1} (or \eqref{eq:B-emb} for the circle case) are distinct. Namely, the Dax invariant is intimately related to the configuration space of two points in $X$ (i.e.\ the second stage of the embedding calculus tower, see Remark~\ref{rem:GW}), and a result of Longoni and Salvatore~\cite{Longoni-Salvatore} shows that those spaces can distinguish nondiffeomorphic (3-)manifolds. Let us point out, however, that by \cite{AS} these extensions agree for homeomorphic smooth manifolds.

\section{Knotted arcs}\label{sec:arcs}
{\color{black}Throughout this section $X$ is} a connected oriented smooth manifold of dimension $d\geq 3$ with nonempty boundary. We pick $x_-,x_+\in\partial X$ and let $x_-$ be the basepoint. For $\D^1=[-1,1]$ we consider spaces $\Arcs{X}$ and $\ImArcs{X}$ of embedded and immersed arcs as in \eqref{eq:def-arcs}, with boundary condition $\{x_-,x_+\}$.

Note that $\pi_0\Arcs{\D^3}\cong\pi_0\Emb(\S^1,\D^3)\cong\pi_0\Emb(\S^1,\S^3)$ is the set of isotopy classes of classical knots. On the other hand, if $d\geq4$ we simply have bijections $\pi_0\Arcs{X}\cong\pi_1X\cong\pi_0\ImArcs{X}$, so it is natural to try to determine the lowest homotopy group distinguishing embeddings from immersions. In other words, we are asking about the connectivity of the inclusion
\begin{equation}\label{eq:arcs-stable-range}
    \incl_X\colon\Arcs{X}\hra \ImArcs{X}.
\end{equation}
By general position the map $\incl_X$ is $(d-4)$-connected, i.e.\ an isomorphism in the so-called \emph{stable} range $[0,d-3)$. A range of nontrivial relative homotopy groups that comes after the stable range has been studied by Dax \cite{Dax}, following the work of Haefliger and Hatcher--Quinn, and is called the \emph{metastable} range $[d-3,2d-6)$. This was then further developed into a theory of embedding calculus by Goodwillie, Klein and Weiss \cite{GKW}, where the stable and metastable range only correspond to respectively the first and second stage of a whole tower of spaces{\color{black}. Remarkably, when $d\geq4$ this tower} completely describes the homotopy type of $\Arcs{X}$. 

In joint work \cite{KT-LBT} we used these ideas to describe explicitly the first potential difference in homotopy groups, {\color{black} obtaining Theorem~\ref{thm:KT-arcs}; we next provide the outline of its proof in order to fix notation.}  In Proposition~\ref{lem:Dax-dim3} we will also extend to the case $d=3$, not discussed in \cite{KT-LBT}, which was stated in \eqref{eq:Dax-dim3} above.

Using the long exact sequence of the pair \eqref{eq:arcs-stable-range}, for $\u\in\Arcs{X}$ the kernel of the surjection $\pi_{d-3}(\incl_X,\u)\colon\pi_{d-3}(\Arcs{X},\u)\sra\pi_{d-3}(\ImArcs{X},\u)$ is given as the cokernel of
\begin{equation}\label{eq:deltaX}
    \delta_{\incl_X}\colon\pi_{d-2}(\ImArcs{X},\u)\to \pi_{d-2}^{Imm,Emb}(X,\u)\coloneqq\pi_{d-2}(\ImArcs{X},\Arcs{X},\u).
\end{equation}
Note that for $d=3$ the kernel of $\pi_0(\incl_X,\u)$ is interpreted as the preimage set $\KK(X,\bm{\u})\coloneqq(\pi_0\incl_X)^{-1}(\bm{\u})$, while the mentioned cokernel is the set of orbits of the action of the group $\pi_1(\ImArcs{X},\u)$ on the set $\pi_1^{Imm,Emb}(X)$ via $\delta_{\incl_X}$ (postconcatenate a loop to a path).

Therefore, we have two tasks: first to compute the relative homotopy group $\pi_{d-2}^{Imm,Emb}(X,\u)$ (which will be done using the Dax invariant) and then the connecting map $\delta_{\incl_X}$ and its image.

\subsection{The Dax invariant}\label{subsec:dax-invariant}
Let $\I=[0,1]$.
A class in the relative homotopy group $\pi_{d-2}^{Imm,Emb}(X,\u)$ is represented by a map $F\colon \I^{d-2}\to \ImArcs{X}$ which takes $\I^{d-3}\times\{0\}$ into the subspace $\Arcs{X}$, and has constant value $\u$ on $\partial\I^{d-3}\times\I\cup\I^{d-3}\times\{1\}$, the rest of the boundary of the cube $\I^{d-2}$. {\color{black}For the proof of the following lemma, see for example \cite[329]{Dax} (and \cite[331]{Dax} for the relative case).}

\begin{lemma}\label{lem:Dax}
    After an arbitrary small deformation of $F$ preserving the boundary conditions, we can assume that the map $\wt{F}\colon\; \I^{d-2}\times\D^1 \to \I^{d-2}\times X$, given by $(\vec{t},\theta)\mapsto (\,\vec{t},\,F(\vec{t})(\theta)\,)$
    is an immersion with only isolated transverse double points.
\end{lemma}
The double points of $\wt{F}$ must be of the form $(\vec{t}_i,x_i)\in\I^{d-2}\times X$, $i=1,\dots,N$, for some $\vec{t}_i\in \I^{d-2}$ and $x_i=F(\vec{t}_i)(\theta^-_i)=F(\vec{t}_i)(\theta^+_i)$, for some $\theta^-_i<\theta^+_i\in \D^1$. Define the double point loop at $x_i$ by
\begin{equation}\label{eq:dp-loop}
    g_{x_i}\coloneqq \big[F(\vec{t}_i)_{\leq\theta_i^-}\cdot F(\vec{t}_i)_{\leq\theta_i^+}^{-1}\big]\quad\in\pi_1X\coloneqq\pi_1(X,x_-),
\end{equation}
where for an arc $\alpha\colon\D^1\imra X$ and $\theta\in\D^1$ we use the notation $\alpha_{\leq\theta}\coloneqq\alpha|_{[-1,\theta]}$.
See Figure~\ref{fig:dax-loop}.
Moreover, let $\varepsilon_{x_i}\in\{-1,1\}$ be the local orientation of $\wt{F}$ at $x_i$, obtained by comparing orientations of tangent bundles to both sheets, namely the image of the derivatives $d\wt{F}|_{(\vec{t}_i,\theta^-_i)}\oplus d\wt{F}|_{(\vec{t}_i,\theta^+_i)}$ with the tangent space of $\I^{d-2}\times X$ at $(\vec{t}_i,x_i)$.

Following \cite{Dax,Gabai-disks} define for $d\geq3$ the \emph{Dax invariant} by
\begin{equation}\label{eq:Dax-def}
    \Dax\colon \pi_{d-2}^{Imm,Emb}(X,\u)\to\Z[\pi_1X],\quad  
    \Dax(\bm{F})\coloneqq\sum_{i=1}^N\varepsilon_{x_i}g_{x_i}\in\Z[\pi_1X].
\end{equation}
This is well defined and does not depend on the choice of a perturbation of the map $F$, but only on its homotopy class $\bm{F}$. For $d\geq4$ the left hand side is a group and $\Dax$ is clearly a group homomorphism since the corresponding maps $F$ get stacked in the $\I^{d-2}$-direction.
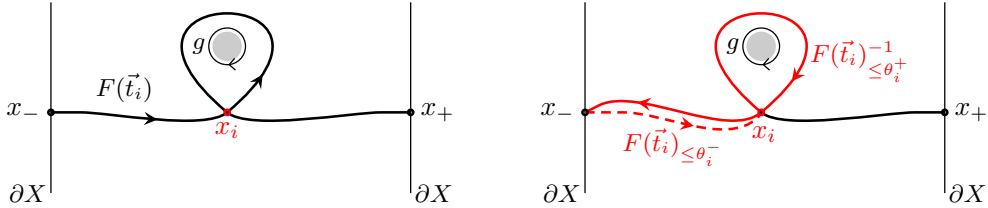
\begin{figure}[!htbp]
    \centering
        \begin{tikzpicture}[scale=0.8,every node/.style={scale=0.85},even odd rule]
        \clip (-0.7,-1.3) rectangle (5.55,1.6);
        \draw 
            (0,1.5) -- (0,-1.1) node[left=-2pt]{\small$\partial X$}  
            (4.9,1.5) -- (4.9,-1.1)  node[right=-2pt]{\small$\partial X$};
        \draw[black,thick,postaction=decorate,decoration = {markings, mark = at position 0.6 with {\arrow{stealth}}}]
            (0,0) -- 
            (0.4,0) to[out=0, in=-140, distance=0.4cm]  (2.4,0)
            (1,0) node[above]{\small$F(\vec{t}_i)$};
        \draw[black,thick]
            (2.4,0) to[out=-40,in=180, distance=0.4cm] 
            (4.5,0) -- (4.9,0)
            ;
        \draw[black,thick,postaction=decorate,decoration = {markings, mark = at position 0.18 with {\arrow{stealth}}}]
            (2.4,0)
            to[out=40,in=140, distance=2.8cm]
            (2.4,0) ;
        \draw[red,thick]
            (2.4,0) circle (1pt) node[below]{$x_i$};
        \draw[black,thick]
            (0,0) circle (1pt) node[left]{$x_-$} 
            (4.9,0)  circle (1pt) node[right]{$x_+$};
            
        \fill[color=black!20!white] 
            (2.4,0.9) circle (0.2cm);
        \draw[black,postaction = decorate,decoration = {markings, mark = at position 0.76 with {\arrowreversed{>}} }] 
            (2.4,0.9) circle (0.25cm) node[left=0.155cm]{\small$g$};
    \end{tikzpicture}~\quad\quad
\begin{tikzpicture}[scale=0.8,every node/.style={scale=0.85},even odd rule]
        \clip (-0.7,-1.3) rectangle (5.55,1.6);
        \draw 
            (0,1.5) -- (0,-1.1) node[left=-2pt]{\small$\partial X$}  
            (4.9,1.5) -- (4.9,-1.1)  node[right=-2pt]{\small$\partial X$};
        \draw[red,densely dashed,thick,postaction=decorate,decoration = {markings, mark = at position 0.6 with {\arrow{stealth}}}]
            (0,0) -- 
            (0.5,0)
            to[out=0, in=-140, distance=0.6cm]  (2.35,-0.08)      --(2.4,0)
            (1.2,-0.08) node[below]{\small$F(\vec{t}_i)_{\leq\theta^-_i}$};
        \draw[red,thick,postaction=decorate,decoration = {markings, mark = at position 0.3 with {\arrowreversed{stealth}}}]
            (0,0) -- (0.2,0.1) to[out=25, in=-140, distance=0.6cm]  (2.4,0)
            (3.75,0.7) node[]{\small$F(\vec{t}_i)_{\leq\theta^+_i}^{-1}$};
        \draw[black,thick]
            (2.4,0) to[out=-40,in=180, distance=0.4cm] 
            (4.5,0) -- (4.9,0);
        \draw[red,thick,postaction=decorate,decoration = {markings, mark = at position 0.15 with {\arrowreversed{stealth}}}]
            (2.4,0) 
            to[out=40,in=140, distance=2.8cm]
            (2.4,0) ;
        \draw[black,thick]
            (0,0) circle (1pt) node[left]{$x_-$} 
            (4.9,0)  circle (1pt) node[right]{$x_+$}
            ;
        \draw[red,thick,fill=black]
            (2.4,0) circle (1pt) (2.44,-0.3) node[]{$x_i$};
            
        \fill[color=black!20!white] 
            (2.4,0.9) circle (0.2cm);
        \draw[black,postaction = decorate,decoration = {markings, mark = at position 0.76 with {\arrowreversed{>}} }] 
            (2.4,0.9) circle (0.25cm) node[left=0.155cm]{\small$g$};
    \end{tikzpicture}
    \caption{\emph{Left.} An immersed arc $F(\vec{t}_i)$ with a double point $x_i=F(\vec{t}_i)(\theta_i^-)=F(\vec{t}_i)(\theta_i^+)$. 
    \emph{Right.} The red double point loop $g_{x_i}$ is homotopic to $g$ (first follow the dashed arc, then the solid red arc).}
    \label{fig:dax-loop}
\end{figure}

Let us describe an explicit (partial) inverse 
\begin{equation}\label{eq:realmap-def}
    \realmap\colon\Z[\pi_1X]\to\pi_{d-2}^{Imm,Emb}(X,\u).
\end{equation}
The class $\realmap(g)$ is represented by the map $\realmap(g)\colon\I^{d-2}=\I^{d-3}\times\I\to\ImArcs{X}$ given as the $(d-3)$-parameter family of embedded arcs $\realmap(g)(\vec{t},0)\in\Arcs{X}$ for $\vec{t}\in\I^{d-3}$, obtained as the ``family finger move'' of $\u$ along $g$, together with paths through immersed arcs $\realmap(g)(\vec{t},t)$ from this family $\realmap(g)(\vec{t},0)$ to the constant family $\const_\u$ equal to $\u$, using the meridian ball as in Figure~\ref{fig:realmap}. In more detail, we pick a meridian sphere $\S^{d-2}_x$ for $\u$, and foliate it by arcs $\alpha_{\vec{t}}$ using its description as the suspension of a $(d-3)$-sphere. We fix $\mytheta_L<\mytheta_R\in\D^1$ and drag the interval $\u([\mytheta_L,\mytheta_R])$ around $g$, then connect sum it into the arc $\alpha_{\vec{t}}$. For the immersed arcs $\realmap(g)(\vec{t},t)$ we instead connect sum into an arc foliating the meridian ball that $\S^{d-2}_x$ bounds.

For $d\geq4$ we extend $\realmap$ to the group ring $\Z[\pi_1X]$ linearly. Namely, subdivide the $s\in\I$ factor into $N$ parts, and define $\realmap(\sum_{i=1}^Ng_i)$ by applying the above foliated finger move along $g_i$ for the corresponding values of $s$. For $d=4$ one needs to checks that this is independent of the order of indices $i$, so that $\realmap(\sum_{i=1}^Ng_i)$ is a well-defined class in $\pi_2^{Imm,Emb}(X,\u)$, which is not necessarily abelian; this follows by an Eckmann--Hilton argument, see \cite{KT-LBT}.

However, for $d=3$ we can only define a set-theoretic map
\begin{equation}\label{eq:realmap-dim3}
    \realmap\colon\Z[\pi_1X]\to\pi_1^{Imm,Emb}(X,\u)=\pi_1(\ImArcs{X},\Arcs{X},\u).
\end{equation}
Namely, for $r\in\Z[\pi_1X]$ we first write $r=\sum_{i=1}^Ng_i$ for \emph{some order} $i=1,2,\dots,N$. Subdivide $[\mytheta_L,\mytheta_R]=[\mytheta_L^1,\mytheta_R^1]\cup\dots\cup [\mytheta_L^N,\mytheta_R^N]$ and also pick some points $x_i\in\D^1$ so that $\mytheta_R<x_1<\dots<x_N$. Then define $\realmap(r)\colon(\I,\{0\},\{1\})\to(\ImArcs{X},\Arcs{X},\u)$ as the path which, as $s\in\I$ decreases, performs one by one finger move of $[\mytheta_L^i,\mytheta_R^i]$ along $g_i$ across the meridian disk to $\u$ at $x_i$.
\begin{remark}
    For every $d\geq3$ an Eckmann--Hilton argument also shows that $\realmap(\sum_{i=1}^Ng_i)$ is homotopic to the map which does the finger moves ``in parallel'', i.e.\ maps $\vec{t}\in \I^{d-3}$ to the \emph{concatenation of the respective arcs} for each subinterval $[\mytheta_L^i,\mytheta_R^i]$. This uses that the adjoint is defined on $\I^{d-3}\times\I\times\D^1$ and that finger moves can be chosen to be contained in disjoint $d$-balls in $X$.
\end{remark}
In the family $\realmap(g)$ there is only one immersed arc, and it has a unique double point, with the loop $+g$, see Figure~\ref{fig:realmap} {\color{black}for the loop and \cite[Thm.4.5]{KT-LBT} for the sign}. This implies that $\Dax\circ\realmap=\Id$,
so $\Dax$ is surjective for all $d\geq3$. For $d\geq4$ the results of Dax imply that  $\realmap\circ\Dax=\Id$ as well, so:
\begin{theorem}[\cite{Dax}, \cite{Gabai-disks}, \cite{KT-LBT}]\label{thm:dax}
    For $d\geq4$ the map $\Dax$ is an isomorphism of groups; for $d=3$ it is a surjection of sets.
\end{theorem}
\begin{remark}\label{rem:rel-Dax}
    One can define the relative Dax invariant $\Dax(f_0,f_1)$, if $f_i\in\pi_{d-3}\Arcs{N\sm D^d}$ are such that $p_\u(f_0)=p_\u(f_1)$. Then $\Dax(f_0,f_1)=\Dax(f_1^{-1}\cdot f_0)$,
    see~\cite[Sec.5.1.1]{KT-LBT}.
\end{remark}

\subsection{The connecting map}\label{subsec:dax-compute}
Let us now study the composite 
\[\begin{tikzcd}
    \pi_{d-2}(\ImArcs{X},\u)\ar{r}{\delta_{\incl_X}} & \pi_{d-2}^{Imm,Emb}(X,\u)\ar[two heads]{r}{\Dax} & \Z[\pi_1X].
\end{tikzcd}
\]
\begin{lemma}\label{lem:Dax-dim3}
    The map $\Dax\circ\delta_{\incl_X}$ is a homomorphism of groups also for $d=3$.
\end{lemma}
\begin{proof}
    The argument is the same as for showing that $\Dax$ is a homomorphism when $d\geq4$: given $a_i\colon(\I,\partial\I)\to(\ImArcs{X},\u)$ for $i=1,2$, the loop $a_1\cdot a_2$ is obtained by concatenating in the $\I$-direction, so $\Dax\circ\delta_{\incl_X}(a_1\cdot a_2)$ is computed as first counting double point loops in the family $F_{a_1}$, then in the family $F_{a_2}$, so it equals $\Dax\circ\delta_{\incl_X}(a_1)+\Dax\circ\delta_{\incl_X}(a_2)$.
\end{proof}

For any $d\geq3$ recall from \eqref{eq:pu} that the map $p_\u\colon\ImArcs{X}\to\Omega X$, defined by $p_\u(\alpha)=\alpha\cdot\u^{-1}$, induces isomorphisms on $\pi_n$ for $n\leq d-3$, by Smale. Moreover, for $n=d-2$ we have
\begin{equation}\label{eq:pu2}
\begin{tikzcd}
    \Z \rar[tail] &
    \pi_{d-2}(\ImArcs{X},\u)\ar[two heads]{rr}{\pi_{d-2}p_\u} &&
    \pi_{d-1}X,
\end{tikzcd}
\end{equation}
a short exact sequence with the first map sending $1\in\Z$ to the ``interior twist'', see \cite{KT-LBT}. 
There we also show that the value of $\Dax\circ\delta_{\incl_X}$ on the interior twist is $1\in\Z[\pi_1X]$. That was for $d\geq4$, but the proof goes through for $d=3$ and is easy: a crossing change on $\u$ along a trivial group element can be isotoped to $\u$.

Therefore, picking any set-theoretic section of $\pi_{d-2}p_\u$ (\emph{we will fix a particular choice soon}) and composing it with $\Dax\circ\delta_{\incl_X}$ we obtain a well-defined homomorphism
\begin{equation}\label{eq:dax-def}
    \dax_\u\colon\pi_{d-1}X\to\Z[\pi_1X\sm1].
\end{equation}
In other words, two sections $l_1,l_2$ differ by an element of $\Z$, so $\dax_\u(a)=\Dax\circ\delta_{\incl_X}\circ l_i(a)\pmod{1}$ is well defined. We conclude from this, \eqref{eq:deltaX} and Theorem~\ref{thm:dax} that there are surjective maps
\[
    \Dax\colon\ker(\pi_{d-3}(\incl_X,\u))\sra\faktor{\Z[\pi_1X]}{\dax_\u(\pi_{d-1}X)}.
\]
For $d\geq4$ they are isomorphisms of groups, while for $d=3$ they are surjections of sets with the source understood as $\KK(X,\bm{\u})=(\pi_0\incl_X)^{-1}(\u)$. This completes our recollection of the proof of Theorem~\ref{thm:KT-arcs}, and we next turn to computing $\dax_\u$ and proving Theorem~\ref{thm:dax-compute}.

\subsubsection{Computing \texorpdfstring{$\dax_\u$}{dax}}
By definition, $\dax_\u(a)\in\Z[\pi_1X\sm1]$ for $a\in\pi_{d-1}X$ is computed as follows. Firstly, represent $a$ by a map $A\colon\S^{d-2}\to \Omega X$, and pick any lift $F_A\colon(\I^{d-2},\partial)\to(\ImArcs{X},\u)$ of $A$, meaning that $p_\u\circ F_A=A$. Secondly, compute $\Dax(F_A)\in\Z[\pi_1X]$, and then disregard any potential $g_{x_i}=1$ in the sum to obtain $\dax_\u(a)$. We now describe a construction of a lift $F_A$, by first choosing a convenient map $A$, which is in terms of a so-called $\u$-whiskered representative $A_\u$. 

Recall that we fix a subinterval $[\mytheta_L,\mytheta_R]\subseteq\D^1=[-1,1]$, and that for an arc $\alpha\colon[c,d]\hra X$ we denote by $\alpha^{-1}$ its time reversal, and for $\theta\in\D^1$ write $\alpha_{\leq\theta}\coloneqq \alpha|_{[c,\theta]}$ and $\alpha_{\geq\theta}\coloneqq \alpha|_{[\theta,d]}$.
\begin{defn}\label{def:A-FA}
    A map $A_\u\colon\I^{d-2}\times[\mytheta_L,\mytheta_R]\to X$ is called a $\u$\emph{-whiskered} representative of a class $a\in\pi_{d-1}(X;x_-)\cong\pi_{d-2}(\Omega X,\u_{\leq\mytheta_R}\cdot\u_{\leq\mytheta_R}^{-1})$ if the following conditions hold (see Figure~\ref{fig:u-whiskered}).
\begin{enumerate}
    \item\label{cond1} 
    Each $A_\u(\vec{t})\colon[\mytheta_L,\mytheta_R]\hra X$ for $\vec{t}\in\I^{d-2}$ is an immersion, whose interior intersects $\u$ transversely and only to the right of $\u(\mytheta_R)$ on $\u$.
    \item\label{cond2}
    For $\vec{t}\in\partial \I^{d-2}$
    we have $A_\u(\vec{t})=\u|_{[\mytheta_L,\mytheta_R]}$.
    \item\label{cond3} 
    The map $(\I^{d-2},\partial)\to(\Omega X,\u_{\leq\mytheta_R}\cdot\u_{\leq\mytheta_R}^{-1})$ defined by the formula
    $\vec{t}\;\mapsto\; \u_{\leq\mytheta_L}\cdot A_\u(\vec{t}) \cdot\u_{\leq\mytheta_R}^{-1}$
    is continuous and represents the homotopy class $a$.
    \qedhere
\end{enumerate}
\end{defn}
\begin{remark}\label{rem:u-whiskered-exists}
    Any $a\in\pi_{d-1}(X;x_-)$ has a $\u$-whiskered representative $A_\u$. Namely, we first represent $a$ by an immersion $A\colon(\I^{d-1},\partial\I^{d-1})\to (X,x_-)$, which intersects $\u$ transversely in finitely many points $\u(\theta_i)$ such that $\theta_i>\mytheta_R$, as in Figure~\ref{fig:u-whiskered} (where $w$ is the whisker to the basepoint $x_-$). Moreover, we can reparametrize $A$ so that writing $\I^{d-1}=\I^{d-2}\times\D^1$ gives adjoint maps $A(\vec{t})\colon\D^1\to X$ for $\vec{t}\in\I^{d-2}$, each of which is an immersed arc from $x_-$ to itself, cf.\ Lemma~\ref{lem:Dax}. 

Finally, we precompose the whisker $w$ of $A$ by the loop $\u\cdot\u^{-1}$, then perform a homotopy until each $A(\vec{t})$ agrees with $\u_{\leq\mytheta_L}$ on $[-1,\mytheta_L]$ and with $\u_{\geq\mytheta_R}\cdot\u^{-1}$ on $[\mytheta_R,1]$. The restriction {\color{black}of this final family} to $[\mytheta_L,\mytheta_R]$ is the desired family $A_\u$.
\end{remark}
\begin{figure}[!htbp]
    \centering
    \begin{tikzpicture}[scale=0.8,every node/.style={scale=0.9},even odd rule]
        \clip (-0.7,-1.3) rectangle (5.55,2.6);
        \draw 
            (0,1.5) -- (0,-1.1) node[left=-2pt]{\small$\partial X$} (4.9,1.5) -- (4.9,-1.1)  node[right=-2pt]{\small$\partial X$};

        \draw[blue,thick]
            (1,2.5) to[out=-50,in=160, distance=1cm]
            (3,1) to[out=-20,in=90, distance=1cm] (4.7,-1) 
            (0.6,1.6) node{$w$};
        \draw[blue] 
            (3,1) to[out=-10,in=-90, distance=1cm] (4,2.5) node[below=1cm]{$A$};
        \draw[blue, dotted] 
            (3,1) to[out=140,in=-90, distance=0.7cm] (2.9,2.6);
        \draw[thick,blue,postaction=decorate,decoration = {markings, mark = at position 0.6 with {\arrow{stealth}}}] 
            (0,0) to[out=20,in=-140, distance=0.7cm]  (1.25,2.2);
    
        \draw[black, very thick]
            (0,0) -- (1,0)  
            (2.6,0) -- (4.9,0);
        \draw[green!70!black,thick]
            (1,0) to[out=0,in=140, distance=0.9cm] (1.9,1.6) 
            to[out=-40,in=-80, distance=0.4cm] (2.8,1.66);
        \fill[white]
            (2.5,1.46) rectangle (2.35,1.2);
        \draw[green!70!black,thick,postaction=decorate,decoration = {markings, mark = at position 0.6 with {\arrow{stealth}}}]    
            (2.8,1.66)
            to[out=100,in=180, distance=0.77cm] (2.6,0) ;
        
        \draw[green!70!black,fill=green!70!black]
            (1.05,0) circle (1.4pt) node[below]{\footnotesize$\mathsf{u}(\uptheta_L)$} 
            (2.6,0) circle (1.4pt) 
            (2.65,0) node[below]{\footnotesize$\mathsf{u}(\uptheta_R)$}
            (2.75,0.25)
            node[above]{\small$A_{\mathsf{u}}(\vec{t})$}; 
        \draw[black,thick]
            (0,0) circle (1.2pt) node[left]{$x_-$} 
            (4.9,0)  circle (1.2pt) node[right]{$x_+$};
        \draw[gray,dashed,thick,postaction=decorate,decoration = {markings, mark = at position 0.76 with {\arrow{stealth}}}]
            (0,-0.07) -- (4.9,-0.07) 
            (3.8,0) node[below]{$\mathsf{u}$};
    \end{tikzpicture}
    \captionof{figure}{A $\u$-whiskered representative $A_\u$ of $a$ for a value $\vec{t}$ is shown in green. The corresponding immersed arc $F_{A_\u}(\vec{t})$ is the union of the green arc $A_\u(\vec{t})$ and the solid horizontal arcs.}
    \label{fig:u-whiskered}
\end{figure}
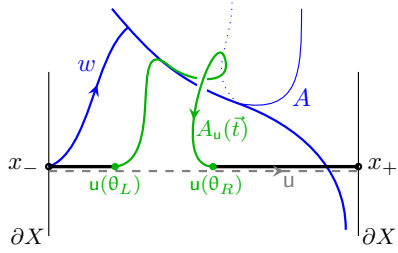
\begin{defn}
    Given a class $a\in\pi_{d-1}X$ and its $\u$-whiskered representative $A_\u$,  define
\[F_{A_\u}\colon(\I^{d-2},\partial)\to(\ImArcs{X},\u),\quad
    F_{A_\u}(\vec{t})\coloneqq\u_{\leq\mytheta_L}\cdot A_\u(\vec{t})\cdot\u_{\geq\mytheta_R}.\qedhere
\]
\end{defn}
Strictly speaking, $F_{A_\u}(\vec{t})$ might not be smooth at the points $\u(\mytheta_L)$ and $\u(\mytheta_R)$, but we can assume that $A_\u$ is chosen so that this is the case, cf.\ Figure~\ref{fig:u-whiskered}.
Moreover, the condition \ref{cond2} for $A_\u$ ensures that $F_{A_\u}(\vec{t})=\u$ for $\vec{t}\in\partial \I^{d-2}$, so $F_{A_\u}$ indeed represents a class in $\pi_{d-2}(\ImArcs{X},\u)$.
\begin{lemma}\label{lem:computing-dax-u}
    The homotopy class of $p_\u\circ F_{A_\u}\colon(\I^{d-2},\partial)\to(\Omega X,\u\cdot\u^{-1})$ is precisely $a\in\pi_{d-1}X$, so $\dax_\u(a)=\Dax(F_{A_\u})$. Moreover, any self intersection of the arc $F_{A_\u}(\vec{t})$ arises either as a self-intersection $x_i$ of an arc $A_\u(\vec{t}_i)$, or as an intersection $y_j$ of an arc $A_\u(\vec{t}_j)$ with $\u_{\geq\mytheta_R}$, and the corresponding double point loops in the two cases can be calculated as:
    \[\begin{cases}
        g_{x_i}=\u_{\leq\mytheta_L}\cdot A_\u(\vec{t}_i)_{\leq\theta_i^-}\cdot A_\u(\vec{t}_i)_{\leq\theta_i^+}^{-1}\cdot\u_{\leq\mytheta_L}^{-1} & \text{ for }x_i\in A\cap A,\\
        g_{y_j}=\u_{\leq\mytheta_L}\cdot A_\u(\vec{t}_j)_{\leq\theta_j^-}\cdot\u_{\leq\theta_j^+}^{-1} & \text{ for }y_j\in A\cap\u.
    \end{cases}
    \]
\end{lemma}
\begin{proof}
    By definition of $p_\u$ and $F_{A_\u}$
    we have $p_\u\circ F_{A_\u}(\vec{t})=\u_{\leq\mytheta_L}\cdot A_\u(\vec{t})\cdot\u_{\geq\mytheta_R}\cdot\u^{-1}$.
    There is an obvious homotopy $\u_{\geq\mytheta_R}\cdot\u^{-1}\simeq\u_{\leq\mytheta_R}$ rel.\ endpoints, so the first claim follow from Definition~\ref{def:A-FA}~\ref{cond3}. 
        For the second claim, first note that the double points of the arcs $F_{A_\u}(\vec{t})$ can arise either
    \begin{enumerate}
    \item[-]
    as the intersections of $\u_{\leq\mytheta_L}$ with $A_\u(\vec{t}_j)$, or
    \item[-]
    as the intersections of $\u_{\leq\mytheta_L}$ with $\u_{\geq\mytheta_R}$, or
    \item[(x)]
    as the self-intersections $x_i$ of $A_\u(\vec{t}_i)$ for some $\vec{t}_i$, or
    \item[(y)]
    as the intersections $y_j$ of $A_\u(\vec{t}_j)$ with
    $\u_{\geq\mytheta_R}$.
\end{enumerate}

However, by Definition~\ref{def:A-FA}\ref{cond1} points in $A\pitchfork\u$ occur in $\u$ to the right of $\u(\mytheta_R)$, so the first case does not arise. The second case is obviously not possible. 

In both remaining cases we have $F_{A_\u}(\vec{t}_i)_{\leq\theta_i^-}=\u_{\leq\mytheta_L}\cdot A_\u(\vec{t}_i)_{\leq\theta_i^-}$. In the $x$-case we similarly have $ F_{A_\u}(\vec{t}_i)_{\leq\theta_i^+}=\u_{\leq\mytheta_L}\cdot A_\u(\vec{t}_i)_{\leq\theta_i^+}$, so the formula for $g_{x_i}$ follows from its definition \eqref{eq:dp-loop}. For the $y$-case observe that the arc $F_{A_\u}(\vec{t}_j)_{\leq\theta_j^+}$ is homotopic rel.\ endpoints to $\u_{\leq\theta_j^+}$ since each $A_\u(\vec{t}_j)$ is homotopic to $A_\u(\vec{t})=\u|_{[\mytheta_L,\mytheta_R]}$, for $\vec{t}\in\partial\I^{d-2}$.
\end{proof}

\subsubsection{Computing $\dax$}
Right before the statement of Theorem~\ref{thm:dax-compute} we have introduced the notation
\[
    \dax\coloneqq\dax_{\u_-}\colon\pi_{d-1}X\to\Z[\pi_1X\sm1]
\]
for a fixed arc $\u_-\colon\D^1\hra X$ isotopic into $\partial X$ rel.\ endpoints and with $\u_-(-1)=x_-$.
We choose
\begin{equation}\label{def:short-arc}
    \u_-=\u_{\leq\mytheta_R}\cdot\u'_{\geq\mytheta_R},
\end{equation}
as in the left part of Figure~\ref{fig:u-minus-whiskered}.
Namely, $\u'_{\geq\mytheta_R}\colon [\mytheta_R,1]\to X$ is a slight pushoff of the reverse of $\u_{\leq\mytheta_R}$, so that it goes from $\u(\mytheta_R)$ to some point $x'_-\in\partial X$ near $x_-\in\partial X$, and so that $\u_-$ is a smooth neat arc from $x_-$ to $x'_-$. By abuse of notation, we consider $\u_-$ also as a loop based at $x_-$.
\begin{figure}[!htbp]
    \centering
    \begin{tikzpicture}[scale=0.8,every node/.style={scale=0.85},even odd rule]
        \clip (-0.7,-1.3) rectangle (5.55,1.6);
        \draw (0,1.5) -- (0,-1.1) node[left=-2pt]{\small$\partial X$}  (4.9,1.5) -- (4.9,-1.1)  node[right=-2pt]{\small$\partial X$};
        \draw[gray,dashed,thick,postaction=decorate,decoration = {markings, mark = at position 0.76 with {\arrow{stealth}}}]
            (0,-0.06) -- (4.9,-0.06) ;
        \draw[black,thick, postaction=decorate,decoration = {markings, mark = at position 0.6 with {\arrow{stealth}}}]
            (2.4,-0.01) to[out=0,in=-5, distance=1.3cm] (0,-0.6) ;
        \draw[black,thick,postaction=decorate,decoration = {markings, mark = at position 0.6 with {\arrow{stealth}}}]
            (0,-0.02) -- (2.4,-0.01) 
            circle (1.2pt) (2.5,-0.01) node[above]{\scriptsize$\mathsf{u}(\uptheta_R)$}  node[above,pos=0.4]{\small$\mathsf{u}_{\leq\uptheta_R}$}  ;
        \draw[black,thick]
            (0,-0.6) circle (1.2pt) 
            (0,-0.5) node[left]{$x'_-$} 
            (2,-0.4)  
            node[below]{\small$\mathsf{u}'_{\geq\uptheta_R}$} ;
        \draw[gray,thick]
            (3.8,0) node[below]{$\mathsf{u}$}
            ;
        \draw[black,thick]
            (0,-0.03) circle (1.2pt) node[left]{$x_-$} 
            (4.9,-0.06)  circle (1.2pt) node[right]{$x_+$};
    \end{tikzpicture}~\quad
\begin{tikzpicture}[scale=0.8,every node/.style={scale=0.85},even odd rule]
        \clip (-0.7,-1.3) rectangle (5.55,2.6);
        \draw (0,1.5) -- (0,-1.1) node[left=-2pt]{\small$\partial X$}  (4.9,1.5) -- (4.9,-1.1)  node[right=-2pt]{\small$\partial X$};
        \draw[gray,dashed,thick,postaction=decorate,decoration = {markings, mark = at position 0.76 with {\arrow{stealth}}}]
            (0,-0.05) -- (4.9,-0.05) ;
            
        \draw[blue,thick]
            (1,2.5) to[out=-50,in=160, distance=1cm]
            (3,1) to[out=-20,in=90, distance=1cm] (4.7,-1) 
            (0.6,1.6) node{$w$};
        \draw[blue] 
            (3,1) to[out=-10,in=-90, distance=1cm] (4,2.5) node[below=1cm]{$A$};
        \draw[blue, dotted] 
            (3,1) to[out=140,in=-90, distance=0.7cm] (2.9,2.6);
        \draw[thick,blue,postaction=decorate,decoration = {markings, mark = at position 0.6 with {\arrow{stealth}}}] 
            (0,0) to[out=20,in=-140, distance=0.7cm]  (1.25,2.2);
    
        \draw[black, thick]
            (0,-0.02) -- (1,-0.02)  
            ;
        \draw[black,thick, postaction=decorate,decoration = {markings, mark = at position 0.6 with {\arrow{stealth}}}]
            (2.6,-0.01) to[out=0,in=-5, distance=1.4cm] (0,-0.6) ;
        \draw[green!70!black,thick]
            (1,-0.02) to[out=0,in=140, distance=0.9cm] (1.9,1.6) 
            to[out=-40,in=-80, distance=0.4cm] (2.8,1.66);
        \fill[white]
            (2.5,1.46) rectangle (2.35,1.2);
        \draw[green!70!black,thick,postaction=decorate,decoration = {markings, mark = at position 0.6 with {\arrow{stealth}}}]    
            (2.8,1.66)
            to[out=100,in=180, distance=0.77cm] (2.6,-0.02) ;
        
        \draw[green!70!black,fill=green!70!black]
            (1.05,-0.02) circle (1.4pt) node[below]{\scriptsize$\mathsf{u}(\uptheta_L)$} 
            node[above=10pt]{$A_{\mathsf{u}}(\vec{t})$}
            (2.6,-0.02) circle (1.4pt) 
            (2.8,-0.02) node[above]{\scriptsize$\mathsf{u}(\uptheta_R)$}; 
        \draw[gray,thick]
            (3.8,0) node[below]{$\mathsf{u}$};

        \draw[black,thick]
            (0,-0.6) circle (1.2pt)
            (0,-0.5) node[left]{$x'_-$} 
            (2,-0.4)  
            node[below]{\small$\mathsf{u}'_{\geq\uptheta_R}$} 
            (0,-0.03) circle (1.2pt) node[left]{$x_-$} 
            (4.9,-0.06)  circle (1.2pt) node[right]{$x_+$};
    \end{tikzpicture}
    \caption{
        \emph{Left.} The arc $\u_-$ goes from $x_-$ to a nearby point $x_-'$. \emph{Right.} A $\u_-$-whiskered representative of $a$.
    }
    \label{fig:u-minus-whiskered}
\end{figure}
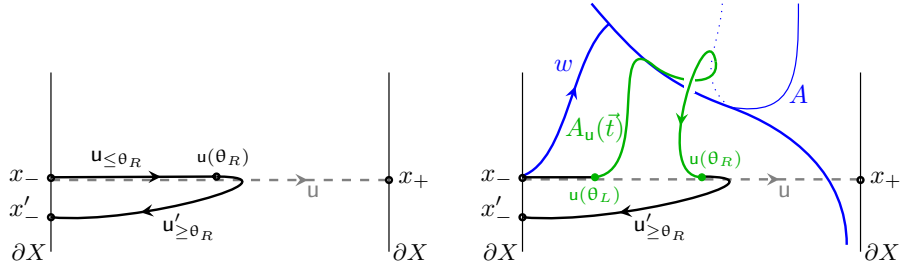
\begin{remark}\label{rem:u-minus-whiskered}
    Note that if $A_\u$ is a $\u$-whiskered representative of $a$, then it is also its $\u_-$-whiskered representative, since all condition of Definition~\ref{def:A-FA} are still fulfilled, {\color{black}see the right part of Figure~\ref{fig:u-minus-whiskered}}. Hence, the map $F_{A_\u}(\vec{t})=\u_{\leq\mytheta_L}\cdot A_\u(\vec{t}) \cdot\u_{\geq\mytheta_R}$ simply changes to $F_{A_{\u_-}}(\vec{t})=\u_{\leq\mytheta_L}\cdot A_\u(\vec{t}) \cdot\u'_{\geq\mytheta_R}$.
\end{remark}

\subsection{Formulae}\label{subsec:formulae}
In this section we express the values of $\dax_\u$ on certain elements in $\pi_{d-1}X$, and prove Theorem~\ref{thm:dax-compute} and Corollary~\ref{cor:dax-compute}. Let us first fix some notation.
\subsubsection{Notation}
Equip $\pi_{d-1}X=\pi_{d-1}(X,x_-)$ with the usual $\pi_1X=\pi_1(X,x_-)$-action, and the set
\[
    \pi_1(X,\partial X)\coloneqq\pi_1(X,\partial X,x_-)=\{k\colon\D^1\to X\mid k(-1)=x_-,\,k(+1)\in\partial X\}/\simeq
\]
with the action of $g\in\pi_1X$ by precomposition, $\bm{k}\mapsto g \bm{k}$. 
Moreover, we have the standard involution on the group ring $\Z[\pi_1X]\coloneqq\{\sum\varepsilon_ig_i:\varepsilon_i=\pm1,g_i\in\pi_1X\}$, which linearly extends $\ol{g}\coloneqq g^{-1}$. 

Next, we recall (see for example \cite{Ranicki}) the usual definition of the \emph{equivariant intersection pairing} 
\[
    \lambda\colon \pi_{d-1}X\times\pi_1(X,\partial X)\to\Z[\pi_1X].
\]
Given classes $a\in\pi_{d-1}X$ and $\bm{k}\in\pi_1(X,\partial X)$, pick smooth representatives $A\colon\S^{d-1}\to X$ and $k\colon(\D^1,\partial\D^1)\to(X,\partial X)$ that intersect transversely and in the interior of $X$, excluding the point $A(e)=k(-1)=x_-\in\partial X$, see Figure~\ref{fig:lambda-ex}. For a transverse intersection point $y\in A(\S^{d-1})\cap k(\D^1)$ define the double point loop
\begin{equation}\label{eq:loop-def-lambda}
    \lambda_y(A,k)\coloneqq\lambda_y(A)\cdot \lambda_y(k)^{-1},
\end{equation}
where paths $\lambda_y(A)$ and $\lambda_y(k)$ go from $x_-$ to $y$, respectively along $A$ and $k$ (the choice of a path along $A$ is irrelevant as $\S^{d-1}$ is simply connected). The sign $\varepsilon_y(A,k)$ is obtained by comparing the orientation of the tangent space $T_yX$ to that of
$dA|_y(T\S^{d-1})\oplus dk|_y(T\D^1)$, and we let
\[
    \lambda(a,\bm{k})\coloneqq\sum_{ y\in (A\cap k)\sm\{x_-\}}\varepsilon_y(A,k)[\lambda_y(A,k)].
\]
Note that we can also compute $\lambda(a,g)$ for $g\in\pi_1X$, using the canonical map $\pi_1X\to\pi_1(X,\partial X)$.
\begin{figure}[!htbp]
    \centering
    \begin{tikzpicture}[scale=0.8,every node/.style={scale=0.95},even odd rule]
        \clip (-0.7,-1.3) rectangle (5.55,2.6);
        \draw 
            (0,1.5) -- (0,-1.1) node[left=-2pt]{\small$\partial X$}  
            (4.9,1.5) -- (4.9,-1.1)  node[right=-2pt]{\small$\partial X$};
        \draw[blue,thick]
            (1,2.5) to[out=-50,in=160, distance=1cm]
            (3,1) to[out=-20,in=90, distance=1cm] (4.7,-1);
            
        \draw[red,thick,postaction=decorate,decoration = {markings, mark = at position 0.3 with {\arrow{stealth}}, mark = at position 0.6 with {\arrow{stealth}}}]
            (0,0) to[out=3,in=-140, distance=0.9cm]  (1.29,1.95) to[out=-55,in=160, distance=0.5cm]
            (2.65,1) to[out=-20,in=100, distance=0.35cm] (4.1,0.1) 
            -- (2.3,0.1) node[pos=0.7,above]{\small$\lambda_y(A,k)$}
            (0,0) -- (0.85,0.1) to[out=10,in=180,distance=0.6cm]
            (1.5,0.6) to[out=0,in=90,distance=0.2cm] 
            (1.9,0.4) to[out=-90,in=180,distance=0.5cm] 
            (1.7,-0.4) to[out=0,in=180,distance=0.9cm]
            (2.3,0.1)
            ;
        
        \draw[blue] 
            (3,1) to[out=-10,in=-90, distance=1cm] 
            (4,2.5) node[below=1cm]{$A$};
        \draw[blue, dotted] 
            (3,1) to[out=140,in=-90, distance=0.7cm] (2.9,2.6);
        \draw[blue,thick,postaction=decorate,decoration = {markings, mark = at position 0.6 with {\arrow{stealth}}}] 
            (0,0) to[out=20,in=-140, distance=0.7cm]  (1.25,2.2);

        \draw[black, thick, postaction=decorate,decoration = {markings, mark = at position 0.82 with {\arrow{stealth}}}]
            (0,0) -- (1,0) to[out=0,in=180,distance=0.5cm]
            (1.5,0.5) to[out=0,in=180,distance=0.8cm] 
            (1.7,-0.5) to[out=0,in=180,distance=1.3cm]
            (2.6,0) -- (4.9,0) node[pos=0.25,below]{$k$};

        \draw[black,thick]
            (0,0) circle (1.2pt) node[left]{$x_-$} 
            (4.9,0)  circle (1.2pt) node[right]{$x_+$};
        \draw[red,thick]
            (4.4,0)  circle (1.2pt) (4.3,0) node[below]{$y$};
    \end{tikzpicture}
    \caption{Associated to the intersection point $y\in A\cap k$ is the loop $\lambda_y(A,k)$, based at $x_-$.}
    \label{fig:lambda-ex}
\end{figure}
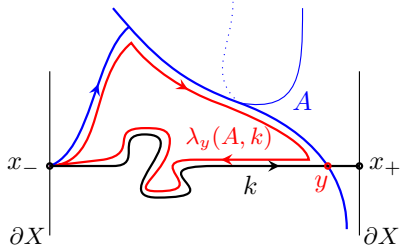

That $\lambda$ is a well-defined invariant of homotopy classes follows from its description using the Hurewicz map on the universal cover
\[
    \pi_{d-1}X\cong\pi_{d-1}\wt{X}\to H_{d-1}(\wt{X};\Z)\cong H_{d-1}(X;\Z[\pi_1X])
\]
and Poincar\'e duality with $\Z[\pi_1X]$-coefficients $H_{d-1}(X;\Z[\pi_1X])\times  H_1(X,\partial X;\Z[\pi_1X])\to \Z[\pi_1X]$.

Finally, one analogously defines a pairing $\lambda\colon\pi_1(X,\partial X)\times\pi_{d-1}X\to\Z[\pi_1X]$, and immediately has
    \begin{equation}\label{rem:other-lambda}
        \lambda(\bm{k},a)=(-1)^{d-1}\ol{\lambda(a,\bm{k})},
    \end{equation}
since $\lambda_y(k,A)=\lambda_y(k)\cdot\lambda_y(A)^{-1}=\lambda_y(A,k)^{-1}$ and exchanging the order of $dA|_y(T\S^{d-1})$ and $dK|_y(T\D^1)$ changes the orientation by $(-1)^{d-1}$. We use this pairing just as a shorter notation for the right hand side of \eqref{rem:other-lambda}.

The following standard properties of $\lambda$ are easy to check and will be useful in our computations.
\begin{lemma}\label{lem:lambda}
    For $\bm{k}\in\pi_1(X,\partial X)$, $g_1,g_2\in\pi_1X$ and $a_1,a_2\in\pi_{d-1}X$ we have
\begin{align*}
    \lambda(g_1 a_1+g_2 a_2,\;\bm{k}) &=g_1\lambda(a_1,\bm{k})+g_2\lambda(a_2,\bm{k}),\\
    \lambda(a,\;g\bm{k})&=\lambda(a,g)+\lambda(a,\bm{k})\ol{g}.
\end{align*}
\end{lemma}
In other words, for any $\bm{k}$ the map $\lambda(-,\bm{k})\colon\pi_{d-1}X\to\Z[\pi_1X]$ is $\Z[\pi_1X]$-linear, while for any $a$ the map $\lambda(a,-)\colon\pi_1(X,\partial X)\to\Z[\pi_1X]$ is a map of $\pi_1X$-sets which satisfies a Fox derivative rule.
Note that taking $\bm{k}=\ol{g}$ in the second equality, we obtain $0=\lambda(a,1)=\lambda(a,g)+\lambda(a,\ol{g})\ol{g}$.

Recall from the introduction that $\lambdabar(a,\bm{k})\in\Z[\pi_1X\sm1]$ is defined by forgetting the term given by $1$ in $\lambda(a,\bm{k})$. One needs to be careful when applying the formulae from Lemma~\ref{lem:lambda}, since, for example, $\lambdabar(ga,\bm{k})\neq g\lambdabar(a,\bm{k})$. However, one easily checks that the following does hold. 
\begin{lemma}\label{lem:lambdabar}
    $\lambdabar(ga,\;g\bm{k})=\lambdabar(ga,g)+g\lambdabar(a,\bm{k})\ol{g}$.
\end{lemma}

\subsubsection{Proofs of Theorem~\ref{thm:dax-compute} and Corollary~\ref{cor:dax-compute}}
We now relate homomorphisms $\dax_\u$ and $\dax\coloneqq \dax_{\u_-}$ from $\pi_{d-1}X$ to $\Z[\pi_1X\sm1]$, for the arc $\u_-$ as in \eqref{def:short-arc}. 
That is, we prove the formula from \ref{thm:dax-compute}\ref{eq:A}:
\[
    \dax_\u(a)=\dax(a)+\lambdabar(a,\bm{\u}),\quad \text{for every } a\in\pi_{d-1}X.
\]
\begin{proof}[Proof of Theorem~\ref{thm:dax-compute}\ref{eq:A}]
    Pick a $\u$-whiskered representative $A_\u$ for $a\in\pi_{d-1}X$ as in Definition~\ref{def:A-FA}. Then by Lemma~\ref{lem:computing-dax-u} we have
    \[
        \dax_\u(a)\coloneqq\Dax(F_{A_\u})=\Dax_x(F_{A_\u})+\Dax_y(F_{A_\u})
    \]
    for $F_{A_\u}\colon\S^{d-2}\to\ImArcs{X}$ given by immersed arcs $F_{A_\u}(\vec{t})\coloneqq\u_{\leq\mytheta_L}\cdot A_\u(\vec{t})\cdot\u_{\geq\mytheta_R}$ whose double points are either self-intersections of $A_\u(\vec{t})$ (type $x$), or intersections of $A_\u(\vec{t})$ with $\u$ (type $y$).

    On the other hand, recall from Remark~\ref{rem:u-minus-whiskered} that a $\u_-$-whiskered representative is given by
    \[
    F_{A_{\u_-}}(\vec{t})\coloneqq \u_{\leq\mytheta_L} \cdot A_\u(\vec{t}) \cdot \u'_{\geq\mytheta_R}.
    \]
    Therefore, $\dax(a)=\Dax(F_{A_{\u_-}})=\Dax_x(F_{A_\u})$. Namely, there is no change in the set of $x$-points and their associated loops, but now there are no $y$-points since we assume the arc $\u_-$ is ``short enough'', i.e.\  $A_\u(\vec{t})\cap\u'_{\geq\mytheta_R}=\emptyset$. Thus, we need only show that $\Dax_y(F_{A_\u})=\lambdabar(a,\u)$.
    
    We have
    $y_j= A(\vec{t}_j)(\theta^-_j)=\u(\theta_j^+)$,
    for some $\vec{t}_j\in\I^{d-2}$, $\theta_j^-\in [\mytheta_L,\mytheta_R]$, $\theta_j^+\in[\mytheta_R,1]$, see Figure~\ref{fig:u-whiskered}. By Lemma~\ref{lem:computing-dax-u}, the associated Dax double point loop is given by
    \[
        g_{y_j}=\u_{\leq\mytheta_L}\cdot A(\vec{t}_j)_{\leq\theta_j^-}\cdot\u_{\leq\theta_j^+}^{-1}.
    \]
    This is precisely equal to $\lambda_{y_j}(A,\u)$ from \eqref{eq:loop-def-lambda} since $\lambda_{y_j}(A)=\u_{\leq\mytheta_L}\cdot A(\vec{t}_j)_{\leq\theta_j^-}$ and $\lambda_{y_j}(\u)=\u_{\leq\theta_j^+}$ (cf.\ Figures~\ref{fig:realmap} and~\ref{fig:lambda-ex}). Moreover, any $g_{y_j}\simeq1$ is not counted towards $\dax_\u(a)$ by definition, see \eqref{eq:dax-def}, so we will have $\Dax_y(F_{A_\u})=\lambdabar(a,\bm{\u})$ once we check that the signs agree.
    
    To this end, the Dax sign $\varepsilon_{y_j}$ compares the orientation of the vector space
    \[
        dA|_{(\vec{t}_j,\theta_j^-)}(T\I^{d-2}\oplus T\D^1)\oplus d\const_{\u}|_{\theta_j^+}(T\I^{d-2}\oplus T\D^1)\quad\text{ to }\quad T_{\vec{t}_j}\I^{d-2}\oplus T_{y_j}X,
    \]
    while the sign $\varepsilon_{y_j}(A,\u)$ for $\lambda$ compares $dA|_{(\vec{t}_j,\theta_j^-)}(T\S^{d-1})\oplus d\u|_{\theta_j^+}(T\D^1)$ to $T_{y_j}X$, or equivalently 
    \[
        T_{\vec{t}_j}\I^{d-2}\oplus dA|_{(\vec{t}_j,\theta_j^-)}(T\S^{d-1})\oplus d\u|_{\theta_j^+}(T\D^1)\quad\text{ to }\quad
        T_{\vec{t}_j}\I^{d-2}\oplus T_{y_j}X.
    \]
    Two out of four displayed oriented spaces are the same, while the other two differ by $(d-2)(d-1)$ transpositions, which is an even number, so $\varepsilon_{y_j}=\varepsilon_{y_j}(A,\u)$.
\end{proof}
We immediately obtain Corollary~\ref{cor:dax-compute}\ref{eq:cor1}: if $A$ is embedded then all arcs in the family $F_{A_{\u_-}}$ are embedded so $\dax(a)=\Dax_x(F_{A_{\u_-}})=0$. Moreover, Corollary~\ref{cor:dax-compute}\ref{eq:cor2} is also immediate since
\[
    \dax_{g\u}(a)-\dax_\u(a)=\dax(a)+\lambdabar(a,g\bm{\u})-(\dax(a)+\lambdabar(a,\bm{\u}))=\lambdabar(a,g\bm{\u})-\lambdabar(a,\bm{\u}).
\]

Let us now prove part~\ref{eq:B} of Theorem~\ref{thm:dax-compute}: for any $a\in\pi_{d-1}X$ and $g\in\pi_1X$ we claim that
\[
    \dax(g a)\;=\;
    g\dax(a)\ol{g}\;-\;\lambdabar(g a,g)+\lambdabar(g,g a),
\]
recalling from~\eqref{rem:other-lambda} that $\lambdabar(g,g a)=(-1)^{d-1}\ol{\lambdabar(g a,g)}$.
\begin{proof}[Proof of Theorem~\ref{thm:dax-compute}\ref{eq:B}]
    Pick a $\u_-$-whiskered representative $A_{\u_-}\colon\I^{d-2}\times[\mytheta_L,\mytheta_R]\to X$ as in Definition~\ref{def:A-FA} so that $\Dax(F_{A_{\u_-}})=\Dax_x(F_{A_{\u_-}})$ is the sum over all self-intersections of $A_{\u_-}(\vec{t})$ (see the previous proof). We now describe a $\u_-$-whiskered representative of $ga$. 
    
    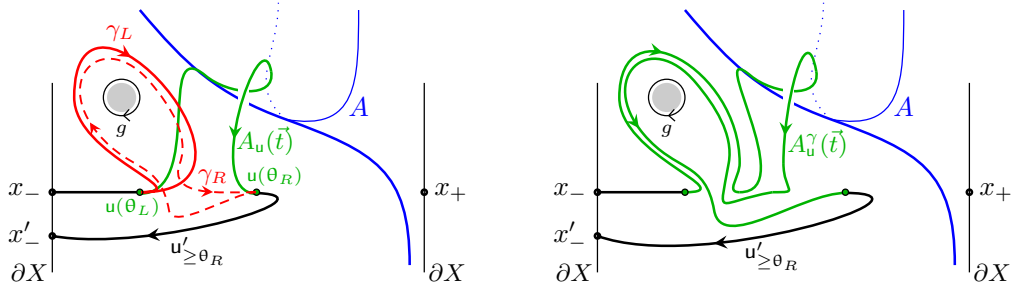
\begin{figure}[!htbp]
        \centering
        \begin{tikzpicture}[scale=0.8,every node/.style={scale=0.85},even odd rule]
        \clip (-0.9,-1.3) rectangle (5.55,2.6);
        \draw (-0.2,1.5) -- (-0.2,-1.1) node[left=-2pt]{\small$\partial X$}  (4.9,1.5) -- (4.9,-1.1)  node[right=-2pt]{\small$\partial X$};
        \draw[thick]
            (-0.2,0) circle (1pt) node[left]{$x_-$} 
            (4.9,0)  circle (1pt) node[right]{$x_+$};
            
        \draw[black,thick, postaction=decorate,decoration = {markings, mark = at position 0.7 with {\arrow{stealth}}}]
           (-0.2,0) -- (1,0) 
            (2.5,-0.01) to[out=0,in=-10, distance=1.4cm]
            (-0.2,-0.6) circle (1pt) node[left]{$x_-'$};
        \draw[black, thick]
            (1.8,-0.8) node[]{\footnotesize$\mathsf{u}'_{\geq\uptheta_R}$};

        \fill[color=black!20!white] (0.75,1.3) circle (0.2cm);
        \draw[postaction = decorate,decoration = {markings, mark = at position 0.76 with {\arrowreversed{>}} }] (0.75,1.3) circle (0.25cm) node[below=0.23cm]{\scriptsize$g$};
        
        \draw[blue,thick]
            (1,2.5) to[out=-50,in=160, distance=1cm] (3,1) 
            to[out=-20,in=90, distance=1.3cm] (4.7,-1) ;
        \draw[blue] 
            (3,1) to[out=-10,in=-90, distance=1cm] (4,2.5) node[below=1cm]{$A$};
        \draw[blue, dotted] 
            (3,1) to[out=140,in=-90, distance=0.7cm] (2.9,2.6);
    
        \draw[green!70!black,thick]
            (1,-0.02) to[out=0,in=140, distance=0.9cm] (1.9,1.6) 
            to[out=-40,in=-80, distance=0.4cm] (2.8,1.66);
        \fill[white]
            (2.5,1.46) rectangle (2.35,1.2);
        \draw[green!70!black,thick,postaction=decorate,decoration = {markings, mark = at position 0.6 with {\arrow{stealth}}}]  
            (2.8,1.66)
            to[out=100,in=180, distance=0.77cm] (2.6,-0.02) ;
            
        \draw[green!70!black,fill=green!70!black]
            (1,0) circle (1pt) 
            (0.9,0.05) node[below]{\scriptsize$\mathsf{u}(\uptheta_L)$} 
            (2.6,0) circle (1pt) 
            (2.85,-0.05) node[above]{\scriptsize$\mathsf{u}(\uptheta_R)$}
            (2.72,0.38) node[above]{\small$A_{\mathsf{u}}(\vec{t})$}; 
        \draw[red,thick,postaction = decorate,decoration = {markings, mark = at position 0.55 with {\arrow{stealth}} }]
            (1.02,0) 
            to[out=0,in=-90, distance=0.8cm] (0.1,1.3) 
            to[out=90,in=0, distance=2.1cm] (1.02,0) ;
        \draw[red,thick]    
            (0.4,2.18)  node[right]{\small$\gamma_L$}
            (2,-0.04) node[above]{\small$\gamma_R$};
        \draw (1,0) circle (1.2pt) (2.6,0) circle (1.2pt);
        \draw[red,densely dashed,semithick,postaction = decorate,decoration = {markings, mark = at position 0.08 with {\arrowreversed{stealth}}, mark = at position 0.595 with {\arrowreversed{stealth}}}]
            (2.59,0)  to[out=180,in=-110, distance=0.5cm] (1.55,0.4)
             to[out=70,in=-50, distance=0.4cm] (1.15,1.5)
             to[out=130,in=130, distance=1cm] (0.43,0.8)
             to[out=-50,in=100, distance=0.7cm] (1.4,-0.25)
             to[out=-80,in=180, distance=0.3cm] (2.59,0);
    \end{tikzpicture}~\quad\quad
\begin{tikzpicture}[scale=0.8,every node/.style={scale=0.85},even odd rule]
        \clip (-0.9,-1.3) rectangle (5.55,2.6);
        \draw (-0.2,1.5) -- (-0.2,-1.1) node[left=-2pt]{\small$\partial X$}  (4.9,1.5) -- (4.9,-1.1)  node[right=-2pt]{\small$\partial X$};
        \draw[thick]
            (-0.2,0) circle (1pt) node[left]{$x_-$} 
            (4.9,0)  circle (1pt) node[right]{$x_+$};
    
        \draw[black,thick, postaction=decorate,decoration = {markings, mark = at position 0.68 with {\arrow{stealth}}}]
           (-0.2,0) -- (1,0) 
           (3.2,-0.01) to[out=0,in=-20, distance=1.5cm] 
           (-0.2,-0.6) circle (1pt) node[left]{$x_-'$} ;

        \fill[color=black!20!white] (0.75,1.3) circle (0.2cm);
        \draw[postaction = decorate,decoration = {markings, mark = at position 0.76 with {\arrowreversed{>}} }] (0.75,1.3) circle (0.25cm) node[below=0.23cm]{\scriptsize$g$};
    
        \draw[blue,thick]
            (1,2.5) to[out=-50,in=160, distance=1cm]
            (3,1) to[out=-20,in=90, distance=1.3cm] (4.7,-1) ;
        \draw[blue] 
            (3,1) to[out=-10,in=-90, distance=1cm] (4,2.5) node[below=1cm]{$A$};
        \draw[blue, dotted] 
            (3,1) to[out=140,in=-90, distance=0.7cm] (2.9,2.6);
    
        \draw[green!70!black,thick,postaction = decorate,decoration = {markings, mark = at position 0.55 with {\arrow{stealth}} }]
            (1,0)  to[out=0,in=-90, distance=0.8cm] (0.1,1.3)
            to[out=90,in=90, distance=1.3cm] (1.75,0.4) 
            to[out=-90,in=180, distance=0.2cm] (1.75,0);
        \draw[green!70!black,thick]
            (2.82,0.33) node[above]{\small$A_{\mathsf{u}}^\gamma(\vec{t})$};
        \draw[green!70!black,thick,postaction = decorate,decoration = {markings, mark = at position 0.55 with {\arrow{stealth}} }]
            (2.2,0) to[out=-170,in=-100, distance=0.8cm] (1.6,0.4)
            to[out=80,in=-50, distance=0.4cm] (1.15,1.5)
            to[out=130,in=130, distance=1cm] (0.43,0.8)
            to[out=-50,in=100, distance=0.7cm] (1.4,-0.25)
            to[out=-80,in=180, distance=0.6cm] (3.2,-0.01);
             
        \draw[green!70!black,thick]
            (1.75,0) to[out=0,in=140, distance=0.85cm] (1.9,1.6) 
            to[out=-40,in=-80, distance=0.4cm] (2.8,1.66)
            ;
        \fill[white]
            (2.5,1.46) rectangle (2.35,1.2);
        \draw[green!70!black,thick,postaction=decorate,decoration = {markings, mark = at position 0.8 with {\arrow{stealth}}}]  
            (2.8,1.66)
            to[out=100,in=90, distance=0.8cm] (2.35,0.4)
            ;
        \draw[green!70!black,thick]
            (2.35,0.4) to[out=-90,in=0, distance=0.4cm] (2.2,0);
            
        \draw[green!70!black,fill=green!70!black]
            (1,0) circle (1pt)
            (3.2,0) circle (1pt) 
            ; 
        \draw (1,0) circle (1.2pt) (3.2,0) circle (1.2pt);
        \draw[black, thick]
            (2.2,-0.87) node[]{\footnotesize$\mathsf{u}'_{\geq\uptheta_R}$};
    \end{tikzpicture}
    
        \caption{\emph{Left:} A loop $\gamma_L$ {\color{black}based at $\u(\mytheta_L)$}, and its pushoff $\gamma_R$ {\color{black}based at $\u(\mytheta_R)$}. \emph{Right:} A $\u_-$-whiskered representative $A^\gamma_{\u_-}(\vec{t})$ of $ga$ consist of the old arc $A_{\u_-}(\vec{t})$ and the ``new whisker'', made of $\gamma_L$ and $\gamma_R^{-1}$.}
        \label{fig:gu-whiskered}
    \end{figure}
    Firstly, we can represent $g\in\pi_1X$ by $\gamma=\u_{\leq\mytheta_L}\cdot\gamma_L\cdot\u_{\leq\mytheta_L}^{-1}$ for an embedded loop $\gamma_L$ based at $\u(\mytheta_L)$, such that $\gamma_L\times pr_2\colon \D^1\times \I^{d-2}\to X\times \I^{d-2}$ intersects the interior of $A_{\u_-}$ transversely and in a finite number of points $(z_j,\vec{t}_j)\in X\times \I^{d-2}$. Similarly, let $\gamma_R$ be a copy of $\gamma_L$ based at $\mytheta_R$ instead, i.e.\ $\gamma_R\simeq\u|_{[\mytheta_L,\mytheta_R]}^{-1}\cdot\gamma_L\cdot\u|_{[\mytheta_L,\mytheta_R]}$, see the left part of Figure~\ref{fig:gu-whiskered}.
    We define $A^\gamma_{\u_-}(\vec{t})$ as
    \[
        \gamma_L\cdot A_{\u_-}(\vec{t})\cdot\gamma_R^{-1},
    \]
    modified into an immersed arc in the obvious way, see the right part of Figure~\ref{fig:gu-whiskered}. We claim that $A^\gamma_{\u_-}$ is a $\u_-$-whiskered representative of $ga$. Namely, recall that the action of $g\in\pi_1X$ on $a\in\pi_{d-1}X$ agrees with the conjugation action of $\pi_0\Omega X$ on $\pi_{d-2}\Omega X$ and note that $A^\gamma_{\u_-}$ is precisely homotopic to such a pointwise conjugate.

    Therefore, $\dax(g a)\coloneqq \Dax(F_{A^\gamma_{\u_-}})\in\Z[\pi_1X\sm1]$ is the sum of signed nontrivial double point loops associated to double points of the immersed arcs 
    \[
        F_{A^\gamma_{\u_-}}(\vec{t})\coloneqq\u_{\leq\mytheta_L}\cdot A^\gamma_{\u_-}(\vec{t})\cdot\u'_{\geq\mytheta_R},\quad \text{for }\vec{t}\in\I^{d-2}.
    \]
    Such points can be divided into two groups.
    \begin{enumerate}
    \item[(old)]
        Each double point $z_i$ of some $F_{A_{\u_-}}(\vec{t}_i)$ appears also in $F_{A^\gamma_{\u_-}}(\vec{t}_i)$.
    \item[(new)]
        Each intersection point $z_j$ that contributes to $\lambda(a,g)$, gives a pair of intersection points: $z_j^L\in A_{\u_-}\cap\gamma_L$ and $z_j^R\in A_{\u_-}\cap\gamma_R^{-1}$, see Figure~\ref{fig:gu-final}.
    \end{enumerate}
    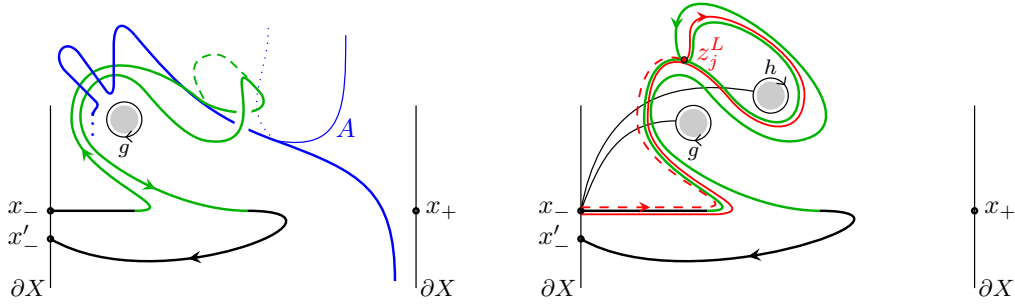
\begin{figure}[!htbp]
        \centering
        \begin{tikzpicture}[scale=0.8,every node/.style={scale=0.85},even odd rule]
        \clip (-0.9,-1.3) rectangle (5.62,3.05);
        \draw (-0.2,1.5) -- (-0.2,-1.1) node[left=-2pt]{\small$\partial X$}  (5,1.5) -- (5,-1.1)  node[right=-2pt]{\small$\partial X$};
        \draw[black,thick]
            (-0.2,0) circle (1pt) node[left]{$x_-$} 
            (5,0)  circle (1pt) node[right]{$x_+$};
            
        \draw[black,thick, postaction=decorate,decoration = {markings, mark = at position 0.6 with {\arrow{stealth}}}]
           (-0.2,0) -- (1,0) (2.6,0) to[out=0,in=-30, distance=1.8cm] (-0.2,-0.4) circle (1pt) node[left]{$x_-'$} ;

        \fill[color=black!20!white] (0.85,1.3) circle (0.2cm);
        \draw[black,postaction = decorate,decoration = {markings, mark = at position 0.76 with {\arrowreversed{>}} }] (0.85,1.3) circle (0.25cm) node[below=0.2cm]{\footnotesize$g$};
            
        \draw[blue,thick]
            (1,2.5) to[out=-50,in=160, distance=1cm]
            (3,1) to[out=-20,in=90, distance=1.3cm] (4.7,-1) ;
        \draw[blue] (3,1) to[out=-10,in=-90, distance=1cm] (4,2.5) node[below=1cm]{$A$};
        \draw[blue, dotted] (3,1) to[out=140,in=-90, distance=0.7cm] (2.9,2.6);

        \draw[green!70!black,thick,postaction = decorate,decoration = {markings, mark = at position 0.35 with {\arrow{stealth}} }]
            (1,0)  to[out=0,in=-90, distance=0.8cm] (0.1,1.3)
            to[out=90,in=140, distance=1cm] (1.9,1.6) ;
        \draw[green!70!black,thick,postaction = decorate,decoration = {markings, mark = at position 0.75 with {\arrow{stealth}} }]
            (2,1) to[out=170,in=-50, distance=0.3cm] (1.15,1.7)
            to[out=130,in=140, distance=0.9cm] (0.5,0.8)
            to[out=-40,in=180, distance=0.7cm] (2.6,0);
        
        \draw[green!70!black,thick]
            (1.9,1.6) to[out=-40,in=-60, distance=0.4cm] (2.8,1.66);
        \draw[green!70!black,semithick,densely dashed]
            (1.9,1.6) to[out=-40,in=-160, distance=0.5cm] (2,2.2) 
            to[out=20,in=140, distance=0.3cm] (2.6,1.9)
            ;
        
        \fill[white]
            (2.4,1.15) -- (2.6,1.05) -- (2.7,1.66) -- (2.5,1.7);
            
        \draw[green!70!black,thick]
        (2.8,1.66) to[out=130,in=-10, distance=1.1cm] (2,1);

            \draw[blue,thick] 
                (1,2.5) to[out=130,in=90, distance=0.3cm] (0.7,2.35) -- (0.7,1.8)
                to[out=-90,in=-50, distance=0.1cm]  (0.6,1.7) -- (0.2,2.25) 
                to[out=130,in=130, distance=0.3cm] (0,2) -- (0.45,1.55)
                to[out=-50,in=90, distance=0.1cm] (0.4,1.35) node{$\vdots$};
    \end{tikzpicture}~\quad\quad
    \begin{tikzpicture}[scale=0.8,every node/.style={scale=0.85},even odd rule]
        \clip (-1.5,-1.3) rectangle (5.4,3.05);
        \draw (-0.8,1.5) -- (-0.8,-1.1) node[left=-2pt]{\small$\partial X$}  (4.8,1.5) -- (4.8,-1.1)  node[right=-2pt]{\small$\partial X$};
        \draw[black,thick]
            (-0.8,0) circle (1pt) node[left]{$x_-$} 
            (4.8,0)  circle (1pt) node[right]{$x_+$};
            
        \draw[black,thick, postaction=decorate,decoration = {markings, mark = at position 0.6 with {\arrow{stealth}}}]
           (-0.8,0) -- (1,0) 
           (2.6,0) to[out=0,in=-30, distance=1.8cm] (-0.8,-0.4) circle (1pt) node[left]{$x_-'$} ;

        \fill[color=black!20!white] (0.8,1.25) circle (0.2cm);
        \draw[black,postaction = decorate,decoration = {markings, mark = at position 0.76 with {\arrowreversed{>}} }] (0.8,1.25) circle (0.25cm) node[below=0.2cm]{\footnotesize$g$};
        \fill[color=black!20!white] (1.9,1.64) circle (0.2cm);
        \draw[black,postaction = decorate,decoration = {markings, mark = at position 0.1 with {\arrowreversed{>}} }] (1.9,1.64) circle (0.25cm) node[above=0.16cm]{\footnotesize$h$};
        \draw 
            (-0.8,0) to[out=70,in=180, distance=0.8cm]  (0.55,1.28) 
            (-0.8,0) to[out=80,in=170, distance=1.4cm]  (1.66,1.7);
    \draw[green!70!black,thick]
            (1,0)  to[out=0,in=-90, distance=0.8cm] (0.1,1.3)
            to[out=90,in=130, distance=1cm]  (1.31,1.66) 
            to[out=-50,in=-90, distance=0.6cm] (2.3,1.66) ;
        \draw[red,dashed,semithick,postaction = decorate,decoration = {markings, mark = at position 0.2 with {\arrow{stealth}} }]
            (-0.8,0.05) -- (0.9,0.05)  
            to[out=0,in=-90, distance=0.8cm] (0,1.3)
            to[out=90,in=160, distance=0.5cm]  (0.67,2.15) ;
        \draw[red,semithick,postaction = decorate,decoration = {markings, mark = at position 0.186 with {\arrowreversed{stealth}} }]
            (2.38,1.66) to[in=90,out=90, distance=0.7cm] (0.75,2.5) 
            to[out=-90,in=0, distance=0.1cm]  (0.67,2.15)
            
            (-0.8,-0.05) -- (1.15,-0.05)  
            to[out=5,in=-90, distance=0.8cm] (0.15,1.3)
            to[out=90,in=130, distance=1cm]  (1.31,1.6) 
            to[out=-50,in=-90, distance=0.65cm] (2.38,1.66) ;

        \draw[green!70!black,thick]
            (2.3,1) to[out=-170,in=-50, distance=0.5cm] (1,1.7)
            to[out=130,in=140, distance=0.9cm] (0.5,0.8)
            to[out=-40,in=180, distance=0.7cm] (2.6,0);

        \draw[green!70!black,thick,postaction = decorate,decoration = {markings, mark = at position 0.45 with {\arrowreversed{stealth}} }]
            (2.3,1.66) to[in=90,out=90, distance=0.7cm] (0.8,2.4) 
            to[out=-90,in=-90, distance=0.45cm]  (0.55,2.5) 
            to[out=90,in=10, distance=1.5cm] (2.3,1);
        
        \fill[red] 
            (0.67,2.15) circle (1pt) 
            (1.05,2.23) node{$z^L_j$};
        \draw (0.67,2.15) circle (1.2pt);
    \end{tikzpicture}
        \caption{\emph{Left:} For every intersection point of $a$ and $g$ there are two double points $z_j^L,z_j^R$ of $A^\gamma_{\u_-}$. Namely, the dashed arc (part of $A_{\u_-}(\vec{t})$) moves across the sphere $a$ and for some $\vec{t}$ it hits back into $\gamma_L$ and $\gamma_R$. \emph{Right:} The associated loop of $z_j^L$ goes along $\gamma$  {\color{black}(dashed)}, then jumps to $A^\gamma_{\u_-}(\vec{t})$  {\color{black}(solid red)}. It is homotopic to {\color{black}$h$}.}
        \label{fig:gu-final}
    \end{figure}
        In the first case the associated loop $g_{z_i}$ for $\dax(g a)$ is the one which contributed to $\dax(a)$ but now conjugated by~$g$, as the new whisker is used for both paths $F_{A^\gamma_{\u_-}}(\vec{t}_i)_{\leq\theta_i^-}$ and $F_{A^\gamma_{\u_-}}(\vec{t}_i)_{\leq\theta_i^+}$ that comprise $g_{z_i}$. This gives $g\dax(a)\ol{g}$, which is exactly the first term on the right hand side of \ref{eq:B}.

    In the second case, the associated Dax double point loops are respectively
    \begin{align*}
        g_{z_j^L}=[\lambda_{z_j}(\gamma_L)\cdot \lambda_{z_j}(A^\gamma_{\u_-})^{-1}]=[\lambda_{z_j}(\gamma,A^\gamma_{\u_-})],\\
        g_{z_j^R}=[\lambda_{z_j}(A^\gamma_{\u_-})\cdot\lambda_{z_j}(\gamma_R)^{-1}]=[\lambda_{z_j}(A^\gamma_{\u_-},\gamma)],
    \end{align*}
    where we used the definition of $\lambda$ from \eqref{eq:loop-def-lambda}. Thus, we will exactly obtain the two remaining terms of \ref{eq:B}, namely $\lambda(g,ga)=(-1)^{d-1}\lambda(ga,g)$ and $-\lambda(ga,g)$ respectively, once we check the signs.
    
    Locally around $z_j^R$ the first strand (the one with $\theta_-$) moves across $A^\gamma_{\u_-}$, while the second (the one with $\theta_+$) is constantly $\gamma^{-1}$, which is exactly the sign of $\varepsilon_{z_j^R}(A^\gamma_{\u_-}, \gamma^{-1})$ so the opposite of the sign for $\lambda(g a,g)$. On the other hand, for $z_j^L$ the first strand is constantly $\gamma$ and the second moves across $A_{\u_-}$, so indeed $\varepsilon^{\Dax}_{z_j^L}=\varepsilon_{z_j^L}(\gamma,A^\gamma_{\u_-})=(-1)^{d-1}\varepsilon_{z_j^L}(A^\gamma_{\u_-},\gamma)$. This completes the proof of~\ref{eq:B}.
\end{proof}

\begin{proof}[Proof of Corollary~\ref{cor:dax-compute}]
For the parts \ref{eq:cor1} and \ref{eq:cor2} see the paragraph after the proof of Theorem~\ref{thm:dax-compute}\ref{eq:A} above. The part ~\ref{eq:cor3} follows from
\begin{align*}
    \dax_{g\u}(g a)-g\dax_\u(a)\ol{g} 
    &=\dax(ga)+\lambdabar(ga,g\bm{\u})-
    g\big(\dax(a)+\lambdabar(a,\bm{\u})\big)\ol{g}\\
    &=\big(\dax(ga)-g\dax(a)\ol{g}\big)
    +\lambdabar(ga,g\bm{\u})-g\lambdabar(a,\bm{\u})\ol{g}\\
    &=\lambdabar(g,ga)-\lambdabar(ga,g)+\big(\lambdabar(ga,g)+g\lambdabar(a,\bm{\u})\ol{g}\big)-g\lambdabar(a,\bm{\u})\ol{g}\\
    &=\lambdabar(g,ga),
\end{align*}
where we applied \ref{eq:A} twice, reordered the terms, then used \ref{eq:B} and Lemma~\ref{lem:lambdabar}, and finally cancelled some terms. We prove~\ref{eq:cor4} by applying \ref{eq:cor2} and then \ref{eq:cor3}:
\begin{align*}
    \dax_{\u}(g a)
    &=\dax_{g\u}(ga)-\lambdabar(ga,g\bm{\u})+\lambdabar(ga,\bm{\u})\\
    &=g\dax_\u(a)\ol{g} + \lambdabar(g,ga)-\lambdabar(ga,g\bm{\u})+\lambdabar(ga,\bm{\u}).
\end{align*}
Finally, ~\ref{eq:cor5} follows from~\ref{eq:cor4}, \ref{eq:cor1} and Lemma~\ref{lem:lambdabar}.
\end{proof}

\section{Knotted circles}\label{sec:circles}
In this section we study the space $\Emb(\S^1,N)$ of smooth embeddings of a circle into a compact manifold $N$ of dimension $d\geq3$ (possibly with $\partial N\neq\emptyset$). The inclusion into the space of immersions  $\incl_N\colon\Emb(\S^1,N)\hra\Imm(\S^1,N)$ induces for any basepoint $\s\in\Emb(\S^1,N)$ isomorphisms 
\begin{equation}\label{eq:stable-range-circles}
    \pi_n\incl_N\colon\pi_n(\Emb(\S^1,N),\s)\xrightarrow{\cong}\pi_n(\Imm(\S^1,N),\s),\quad 0\leq n\leq d-4,
\end{equation}
similarly as in \eqref{eq:arcs-stable-range} for arcs.
Thus, the lowest homotopy group potentially distinguishing embedded from immersed circles is again $\pi_{d-3}$, and $\ker(\pi_{d-3}\incl_N)$ is isomorphic to the cokernel of the connecting map $\delta_{\incl_N}\colon\pi_{d-2}(\Imm(\S^1,N),\s)\to\pi_{d-2}(\Imm(\S^1,N),\Emb(\S^1,N),\s)$.

The first part of Theorem~\ref{thm:circles-main} is a computation of $\pi_n(\Imm(\S^1,N),\s)$ for $n\leq d-3$, which we carry out in Section~\ref{subsec:immersed-circles}. In Section~\ref{subsec:circles-to-arcs} we relate the space $\Emb(\S^1,N)$ to $\Emb_\partial(\D^1,N\sm D^d)$, and use this to prove Theorem~\ref{thm:circles-main}\ref{thm:circles-mainII} which computes $\ker(\pi_{d-3}\incl_N)$ (restated as Theorem~\ref{thm:circles-main2}), see Section~\ref{subsec:computing-the-kernel}.

\subsection{Immersed circles}\label{subsec:immersed-circles}
Let $\S(N)$ denote the unit sphere subbundle of the tangent bundle of $N$, and $\Lambda\S(N)\coloneqq\Map(\S^1,\S(N))$ its free loop space, i.e.\ the space of all (nonbased) loops in $\S(N)$. Fix a neighbourhood $[e,e+\varepsilon]\subseteq\S^1$ of a point $e\in\S^1$ and en embedding $\ol{e}\colon[e,e+\varepsilon]\hra N$.
\begin{prop}\label{prop:free-loops}
    Taking unit derivatives gives a weak homotopy equivalence of fibration sequences
    \[
    \begin{tikzcd}
        r^{-1}(\ol{e})\ar[hook]{r}{}\ar[]{d}{\simeq}[swap]{\deriv} & \Imm(\S^1,N)\ar[]{r}{r}\ar[]{d}{\simeq}[swap]{\deriv}  & \Imm([e,e+\varepsilon],N)\ar[]{d}{\simeq}[swap]{\deriv_e}\\
        \Omega\S(N)\ar[hook]{r}{} & \Lambda\S(N)\ar[]{r}{\ev_e}  & \S(N)
    \end{tikzcd}
    \]
    where $r$ restricts an immersion to $[e,e+\varepsilon]\subseteq\S^1$ and $\ev_e$ evaluates a map at $e$. 
    Moreover, in the bottom sequence the connecting map sends $\gamma\in\Omega\S(N)$ to the concatenated loop $\gamma\cdot\deriv\s\cdot\gamma^{-1}{\color{black}\in}\Omega\S(N)$.
\end{prop}
\begin{proof}
    The map $\deriv_e\colon\Imm([e,e+\varepsilon],N)\to\S(N)$ taking the unit derivative at $e\in\S^1$ is a weak equivalence, since the arc can be arbitrarily shorten (see~\cite[Lem.~6.2]{Smale}). The restriction map $r$ is a Serre fibration by Smale~\cite{Smale}, so the fibre $r^{-1}(\ol{e})$ over any basepoint is weakly equivalent to the fibre $\Imm_*(\S^1,N)\coloneqq (\deriv_e\circ r)^{-1}(\ol{e})$, the space of \emph{based immersions}. Now, Smale also shows that $\deriv\colon\Imm_*(\S^1,N)\to\Omega\S(N)$ is a weak equivalence. Thus, in the above diagram the leftmost vertical map $\deriv$ is also a weak equivalence, implying the middle map $\deriv$ is as well.
    
    For the last claim, first note that the connecting map goes from the space $\Omega\S(N)$ based at $\const_*$ to $\Omega\S(N)$ based at $\deriv\s$. By definition, it takes a loop $\gamma$ in $\S(N)$, lifts it (using the homotopy lifting property of $\ev_e$) to a path in $\Lambda\S(N)$ which at time $t=1$ equals $\deriv\s$, and then evaluates at the other endpoint\footnote{We use this convention as we prefer to define the group commutator by $[a,b]=aba^{-1}b^{-1}$. If lifts would start at $t=0$, we would get $\gamma^{-1}\cdot\deriv\s\cdot\gamma$, so we would have to use $a^{-1}b^{-1}ab$ instead.} $t=0$. We simply observe that the path $t\mapsto\gamma|_{[t,1]}\cdot\deriv\s\cdot\gamma|_{[t,1]}^{-1}$ is one such lift, since it equals $\deriv\s$ for $t=1$, and for each $t$ its value at $e$ is $\gamma(t)$.
\end{proof}
Corresponding to a basepoint $\s\in\Imm(\S^1,N)$ is the point $\deriv\s\in\Lambda\S(N)$. Note that it is important to keep track of basepoints, since the components of $\Lambda\S(N)$ are in general not even homotopy equivalent (so neither are components of $\Imm(\S^1,N)$). However, the components of $\Omega\S(N)$ are: the postconcatenation $-\cdot(\deriv\s)^{-1}\colon\Omega\S(N)\to\Omega\S(N)$, $\gamma\mapsto\gamma\cdot(\deriv\s)^{-1}$ is a homotopy equivalence, with an obvious inverse $\gamma\mapsto\gamma\cdot(\deriv\s)$ (note that $(\deriv\s)^{-1}\neq\deriv(\s^{-1})$). 

Thus, writing $(-\cdot(\deriv\s)^{-1})\circ\deriv=(\deriv-)\cdot(\deriv\s)^{-1}$ we have a diagram of fibration sequences
    \begin{equation}\label{eq:imm-fib-seq}
    \begin{tikzcd}
        r^{-1}(\ol{e})\ar[hook]{r}{}\ar[]{d}{\simeq}[swap]{(\deriv-)\cdot(\deriv\s)^{-1}} & \Imm(\S^1,N)\ar[]{r}{r}\ar[]{d}{\simeq}[swap]{\deriv}  & \Imm([e,e+\varepsilon],N)\ar[]{d}{\simeq}[swap]{\deriv_e}\\
        \Omega\S(N)\ar{r}{\cdot(\deriv\s)} & \Lambda\S(N)\ar[]{r}{\ev_e}  & \S(N)
    \end{tikzcd}
    \end{equation}
so that now the basepoint of $\Omega\S(N)$ is $\const_*$, and the bottom connecting map is $\gamma\mapsto \gamma\cdot{\deriv\s}\cdot\gamma^{-1}\cdot{\deriv\s}^{-1}$. 

Using the canonical isomorphisms $\pi_n(\Omega\S(N),\const_*)\cong\pi_{n+1}(\S(N),e)$ and the correspondence of Samelson and Whitehead products (see, for example, \cite[App.B]{K-thesis-paper}), we view this connecting map as the self-map of $\pi_{n+1}\S(N)$ given by
\[
    a\mapsto[a,\bm{\deriv\s}]_W=\begin{cases} 
        a\,(\bm{\deriv\s}\,a)^{-1},  & n=0,\\
        a-\bm{\deriv\s}\,a,          & n>0.
    \end{cases}
\]
This is the Whitehead product with the homotopy class $\bm{\deriv\s}\in\pi_1\S(N)$ of $\deriv\s$. As before, $\bm{\deriv\s}\,a$ denotes the usual action of $\bm{\deriv\s}\in\pi_1\S(N)$ on $a\in\pi_n\S(N)$ (``the change of whisker''). For $n=0$ this is the conjugation action, so the formula $a\,(\bm{\deriv\s}\,a)^{-1}$ is exactly the commutator $[a,\bm{\deriv\s}]\in\pi_1\S(N)$.

\begin{cor}\label{cor:imm-circles-htpy-gps}
   There is a bijection $\pi_0\Imm(\S^1,N)\cong\faktor{\pi_1\S(N)}{[a,x]\sim1}$ and group extensions
    \[\begin{tikzcd}
        \faktor{\pi_{n+1}\S(N)}{\langle a-\bm{\deriv\s}\,a\rangle}\ar[tail]{r} &  \pi_n(\Imm(\S^1,N),\s)\ar[two heads]{r} & \big\{b\in\pi_n\S(N)\mid b=\bm{\deriv\s}\,b\big\}.
    \end{tikzcd}\]
    for all $n\geq1$.
    Moreover, for $n\leq d-3$ the space $\S(N)$ can be replaced by $N$ and $\bm{\deriv\s}$ by $\bm{\s}\in\pi_1N$.
\end{cor}
The first part follows from the long exact sequence in homotopy groups and \eqref{eq:imm-fib-seq}, noting that the set $\pi_0\Lambda\S(N)$ is in bijection with the set of orbits of the action of $\pi_1\S(N)$ on itself via the connecting map $a\mapsto[a,x]$ for varying $x=\bm{\deriv\s}$. The last sentence follows since the projection $\S(N)\to N$ is $(d-1)$-connected (as the fibre is $\S^{d-1})$.

\begin{remark}\label{rem:nullhomotopic-has-section}
    If $\bm{\deriv\s}=1\in\pi_1\S(N)$ (equivalently $\bm{\s}=1\in\pi_1N$), we have split extensions
    \[\begin{tikzcd}
        \pi_{n+1}\S(N)\ar[tail]{r} & \pi_n(\Imm(\S^1,N);1)\ar[two heads,shift left]{r} & \pi_n\S(N).\ar[shift left]{l}
    \end{tikzcd}
    \] 
    The splitting comes from the fact that $\ev_e$ has a section, sending $x\in\S(N)$ to the constant loop $\const_x$. Note that this section is basepoint preserving if and only if we choose a constant loop for a basepoint of $\Lambda\S(N)$ (which we can do in this case since $\s$ is nullhomotopic).
\end{remark}

\subsection{Reducing circles to arcs}\label{subsec:circles-to-arcs}
We next give a fibration for embedded circles analogous to \eqref{eq:imm-fib-seq}.

\begin{prop}\label{prop:closed-to-boundary}
    Taking the unit derivative at $e\in\S^1$ gives a fibration sequence
    \[\begin{tikzcd}
        \Arcs{N\sm D^d}\ar[]{r}{-\cdot \ol{e}} & \Emb(\S^1,N)\ar[]{r}{\deriv_e} & \S(N),
    \end{tikzcd}
    \]
    where the first map glues together along the boundary the neat arc $\ol{e}\colon[e,e+\varepsilon]\hra D^d\subseteq N$ with the given neat arc $K\colon\D^1=\S^1\sm[e,e+\varepsilon]\hra N\sm D^d$.
\end{prop}
The boundary condition for $K\in\Arcs{N\sm D^d}$ is given by the derivative of $\ol{e}$ at the boundary, so that gluing them gives a smooth embedding $K\cdot\ol{e}\colon\S^1\hra N$.
\begin{proof}
    The restriction map $r\colon\Emb(\S^1,N)\to\Emb([e,e+\varepsilon],N)$ is a fibration by the Cerf--Palais theorem \cite{Cerf-plongements,Palais}. For the base space of this fibration we have weak homotopy equivalences $\deriv_e\colon\Emb([e,e+\varepsilon],N)\simeq\Imm([e,e+\varepsilon],N)\simeq\S(N)$, since the shrinking from the argument in the proof of Proposition~\ref{prop:free-loops} can be done through embeddings. 
    
    Finally, for the fibre $r^{-1}(\ol{e})$ over the basepoint $\ol{e}\colon[e,e+\varepsilon]\hra N$, we claim that there is an equivalence $r^{-1}(\ol{e})\simeq\Arcs{N\sm D^d}$. We claim that there is a retraction of $r^{-1}(\ol{e})$ onto its subspace consisting of those $K\colon\S^1\hra N$ which intersect an open tubular neighbourhood $\nu_{\ol{e}}\coloneqq \im\ol{e}\times\D^{d-1}_{\epsilon}\cong\D^{d}\subseteq N$ \emph{precisely} in $K([e,e+\varepsilon])=\im\ol{e}$, since then we will have
    \[
        r^{-1}(\ol{e})\simeq\Emb_\partial(\S^1\sm[e,e+\varepsilon],N\sm\nu_{\ol{e}})\cong\Arcs{N\sm D^d}.
    \] 
    The retraction can be constructed by integrating a vector field defined on each punctured normal disk $\ol{e}(p)\times(\D^{d-1}_{\epsilon}\sm\{0\})$ for $p\in[e,e+\varepsilon]$ by pointing radially outwards. This gives a homotopy from the identity map on $\nu_{\ol{e}}\sm\im\ol{e}$ to a smooth self-map $\varphi$ with the image $\varphi(\nu_{\ol{e}}\sm\im\ol{e})\subseteq\im\ol{e}\times(\D^{d-1}_{\epsilon}\sm\D^{d-1}_{\epsilon/3})$. Since $K\colon\S^1\hra N$ belongs to $r^{-1}(\ol{e})$ if it agrees with $\ol{e}$ on $[e,e+\varepsilon]$, does not intersect $\ol{e}$, the composite $\varphi\circ K$ is well defined and belongs to the desired subspace. 
\end{proof}
Let $\s\colon\S^1\hra N$ be the basepoint in $\Emb(\S^1,N)$, so that $\u\coloneqq \s|_{\S^1\sm[e,e+\varepsilon]}$ is the basepoint in the fibre $\Arcs{N\sm D^d}$, and $\u\cdot\ol{e}=\s$.

Using the inclusion $\incl_N\colon\Emb(\S^1,N)\hra\Imm(\S^1,N)$ we combine the last fibration sequence with those from \eqref{eq:imm-fib-seq}, to obtain a commutative diagram
\begin{equation}\label{eq:combining}
\begin{tikzcd}
    \Arcs{N\sm D^d}\ar[hook]{r}{-\cdot \ol{e}}\ar{d}[swap]{\jincl_N\deriv_\u} & \Emb(\S^1,N)\ar[]{r}{\deriv_e}\ar{d}[swap]{\deriv\incl_N} & \S(N)\ar[equals]{d}{}\\
    \Omega\S(N)\ar[]{r}{-\cdot(\deriv\s)} & \Lambda\S(N)\ar[]{r}{\ev_e}  & \S(N)
\end{tikzcd}
\end{equation}
Let us explain the label of the left vertical map. It is obtained by precomposing $(\deriv-)\cdot(\deriv\s)^{-1}$ from \eqref{eq:imm-fib-seq} with $-\cdot\ol{e}$. Since $\ol{e}=\u^{-1}\cdot \s$, this takes $K\in\Arcs{N\sm D^d}$ to
\begin{align*}
    \left((\deriv-)\cdot(\deriv\s)^{-1}\right)(K\cdot\ol{e})
    &=\deriv(K\cdot\ol{e})\cdot(\deriv\s)^{-1}\\
    &=(\deriv K\cdot (\deriv\u)^{-1}\cdot \deriv \s)\cdot (\deriv \s)^{-1}\\
    &=\deriv K\cdot (\deriv\u)^{-1}.
\end{align*}
Therefore, this is the composite of the map $\deriv_\u\colon\ImArcs{N\sm D^d}\to\Omega\S(N\sm D^d)$, defined by $\deriv_\u(K)=(\deriv K)\cdot(\deriv\u)^{-1}$, and the map $\jincl_N\colon \Omega\S(N\sm D^d)\hra\Omega\S(N)$ induced by the inclusion $N\sm D^d\subseteq N$.

From the vertical and horizontal long exact sequences of homotopy groups, we see that $-\cdot \ol{e}$ induces a weak equivalence $\hofib_{\const_*}(\jincl_N\deriv_\u)\simeq\hofib_{\deriv \s}(\deriv\incl_N)$, for the basepoint $*=\ol{e}(e)\in N$.

Moreover, {\color{black}by} Corollary~\ref{cor:imm-circles-htpy-gps} the following commutative diagram has exact rows and columns.
\begin{equation}\label{diag:main-circles-hofibs}
    \begin{tikzcd}[column sep=large]
        \pi_{d-1}(\S N)\ar{r}{a\mapsto a-\bm{\deriv\s}\,a}\ar{d} & 
        \pi_{d-1}(\S N)\ar{r}\ar{d}{\delta_{\jincl_N\deriv_\u}} & 
        \pi_{d-2}\Lambda\S(N)\ar{d}{\delta_{\deriv\incl_N}}\ar[two heads]{r}{} & \big\{b\in\pi_{d-2}N\mid b=\bm{\s}\,b\big\}\ar{d}
        \\
        0\ar{d}\ar{r} &    \pi_{d-3}\hofib_*(\jincl_N\deriv_\u)\ar{d}\ar[tail,two heads]{r}{\cong} & 
        \pi_{d-3}\hofib_{\deriv\s}(\deriv\incl_N)\ar{d} \ar{r}{} & 0\ar{d}
        \\
        \pi_{d-2}N\ar[equals]{d}\ar{r}{\delta_{\deriv_e}} &
        \pi_{d-3}(\Arcs{N\sm D^d},\u)\ar[two heads]{d}{\pi_{d-3}(\jincl_N\deriv_\u,\u)}\ar{r}{\pi_{d-3}(-\cdot \ol{e})} & \pi_{d-3}(\Emb(\S^1,N),\s)\ar[two heads]{d}{\pi_{d-3}(\deriv\incl_N,\s)} \ar{r}{\pi_{d-3}\deriv_e} & \pi_{d-3}N\ar[equals]{d}
        \\
        \pi_{d-2}N\ar[]{r}{b\mapsto b-\bm{\s}\,b} & \pi_{d-2}N\ar[]{r}{\pi_{d-3}(-\cdot \deriv\s)} & \pi_{d-3}(\Lambda\S(N),\deriv\s) \ar{r}{} & \pi_{d-3}N.
    \end{tikzcd}
\end{equation}
The desired group $\ker\pi_{d-3}(\deriv\incl_N,\s)$ is the cokernel of the connecting map $\delta_{\deriv\incl_N}$. We apply the snake (aka kernel--cokernel sequence) lemma to the top two rows with the leftmost top term omitted.
\begin{theorem}\label{thm:whisk}
There is an exact sequence of groups
\[
\begin{tikzcd}[column sep=large]
        \ker(\delta_{\deriv\incl_N})\ar{r} & \big\{b\in\pi_{d-2}N\mid b=\bm{\s}\,b\big\}
        \ar{r}{\delta^{whisk}_{\s}} &
        \ker\pi_{d-3}(\jincl_N\deriv_\u,\u)\ar[two heads]{r}{} & \ker\pi_{d-3}(\deriv\incl_N,\s).
\end{tikzcd}
\]
where $\delta^{whisk}_{\s}$ is a ``parametrized change of the whisker'': for $d=3$ it sends $b=\bm{\s}b\bm{\s}^{-1}\in\pi_1N$ to an embedded concatenation $b\cdot\bm{\s}\cdot b^{-1}$, while for $d\geq4$ it is a family version of this.

Moreover, if $\s$ is nullhomotopic then $\delta^{whisk}_{\s}$ is trivial.
\end{theorem}
\begin{proof}
    It remains to describe the connecting map $\delta^{whisk}_{\s}$ in the kernel--cokernel sequence, by definition given as a restriction of $\delta_{\deriv_e}$. This in turn has a description similar to the connecting map for $\ev_e$ given in the proof of Proposition~\ref{prop:free-loops}. Namely, for $\gamma\in\Omega\S(N)$ we let $\delta_{\deriv_e}(\gamma)\coloneqq B_0|_{\S^1\sm[e,e+\varepsilon]}$ where we lift the loop $\gamma$ to a path $B_t\in\Emb(\S^1,N)$, i.e.\ an isotopy of ``non-based'' embedded circles, such that $B_1=\s$ and for every $t\in[0,1]$ the unit derivative at $e$ is equal to $\deriv_e(B_t)=\beta_t$.
    Note that by the commutativity of \eqref{eq:combining}
    each $B_t\in\Emb(\S^1,N)$ is homotopic to the loop $\gamma|_{[t,1]}\cdot\deriv\s\cdot\gamma|_{[t,1]}^{-1}\in\Lambda\S(N)$, which was used in the definition of $\delta_{\ev_e}$.

    Thus, to obtain $\delta^{whisk}_{\s}(b)$ we represent $b\in\pi_{d-2}N$ by a map $\beta\colon\I^{d-3}\times\I\to\S(N)$, lift it to an isotopy $B\colon\I^{d-3}\times\I\to\Emb(\S^1,N)$ (i.e.\ $\deriv_e\circ B=\beta$) with $B(\vec{t},t)=\s$ for $\vec{t}\in\partial\I^{d-3}$ or $t=1$, and let  
    \[
        \delta^{whisk}_{\s}(b)\coloneqq[B(-,0)|_{\S^1\sm[e,e+\varepsilon]}\colon(\I^{d-3},\partial\I^{d-3})\to (\Arcs{N\sm D^d},\u)].
    \]
    Then each $B(\vec{t},t)$ is homotopic to $\beta_{\vec{t},\geq t}\cdot\s\cdot\beta_{\vec{t},\geq t}^{-1}$, so the family $B(-,0)$ is an embedded version of the map $\beta_{-,\geq0}\cdot\s\cdot\beta_{-,\geq0}^{-1}\colon\S^{d-3}\to\Omega\S(N)$, and can thus be viewed as a ``parametrized whisker change''. Moreover, since we assume $b-\bm{\s}\,b=0$ it follows that $B(-,0)$ is based homotopic to $\const_\s$, so that $\delta^{whisk}_{\s}(b)$ is indeed in the kernel of $\pi_{d-3}(\jincl_N\deriv_\u,\u)$.
    
    In particular, for $d=3$ this means that we simply isotope $\s$ by dragging its part $\ol{e}=\s|_{[e,e+\varepsilon]}$ around a loop $\beta\in\Omega N$ representing $b\in\pi_1N$, until this part comes back to its initial position (we will have to slightly push away the rest of $\s$). The resulting embedded circle $B(0)$ is clearly freely isotopic to $\s$, and based homotopic to $\beta\cdot\s\cdot\beta^{-1}$ which is based homotopic to $\s$. Thus, the restriction of $B(0)$ to $\S^1\sm[e,e+\varepsilon]$ is an arc homotopic to $\u$ rel.\ endpoints, but possibly not isotopic to it. Note that for $\bm{\s}=1$ every $\delta^{whisk}_{\s}(b)$ is clearly isotopic to $\u$.
    
    In fact, for nullhomotopic $\s$ and any $d\geq3$ the map $\pi_n\deriv_e$ has a section for any $n\geq0$, implying that $\delta_{\deriv_e}=0$, so $\delta^{whisk}_{\s}$ is trivial in this case. Namely, any family $\S^n\to\S(N)$ lifts by the parametrised ambient isotopy extension theorem to a family $\I^n\to\Emb(\S^1,N)$, which is constantly $\s$ on the boundary except possibly on $\{1\}\times\I^{n-1}$. However, we can choose the ambient isotopy so that it encompasses a ball containing $\s$. Thus, the rest of $\s$ is also being dragged along (not only $\s([e,e+\e])$) in the isotopy, giving a desired lift $\S^k\to\Emb(\S^1,N)$.
\end{proof}

\subsection{Computing the kernel}\label{subsec:computing-the-kernel}
It remains to prove the following statement from Theorem~\ref{thm:circles-main}.
\begin{theorem}\label{thm:circles-main2}
    For $d\geq4$ the Dax invariant for arcs induces an isomorphism from $\ker\pi_{d-3}(\incl_N,\s)$ to the abelian group with generators $g\in\pi_1N\sm1$ and relations:
      $\Dax(\delta^{whisk}_\s(c))=0$ for all $c\in\pi_{d-2}N$ such that $c=\bm{\s}\,c$, and
      $\dax_\u(a)=0$ for all $a\in\pi_{d-1}(N\sm D^d)$.
\end{theorem}
We will use that $\ker(\pi_{d-3}\incl_N,\s)\cong\ker(\pi_{d-3}\deriv\incl_N,\s)$ (as $\deriv$ is a homotopy equivalence by Proposition~\ref{prop:free-loops}) and Theorem~\ref{thm:whisk}: we will first compute $\ker(\pi_{d-3}\jincl_N\deriv_\u,\u)$ and then the image of $\delta^{whisk}_{\s}$. 

\begin{lemma}\label{lem:kernels-agree}
    The subgroups $\ker(\pi_{d-3}\jincl_N\deriv_\u)$ and $\ker(\pi_{d-3}p_\u)$ of $\pi_{d-3}\Arcs{N\sm D^d}$ agree.
\end{lemma}
\begin{proof}
    First recall that $\deriv_\u\colon\Arcs{N\sm D^d}\to\Omega\S(N\sm D^d)$, is given by $K\mapsto \deriv K\cdot(\deriv\u)^{-1}$, while $p_\u\colon\Arcs{N\sm D^d}\to\Omega(N\sm D^d)$ is given by $K\mapsto K\cdot\u^{-1}$. Thus, these maps agree on $\pi_{d-3}$, since $\pi_{d-2}\S(N\sm D^d)=\pi_{d-2}(N\sm D^d)$. The result then follows from $\pi_{d-3}(\jincl_N\deriv_\u)= \pi_{d-3}(\jincl_N)\circ\pi_{d-3}(\deriv_\u)$, and the fact that the map $\pi_{d-3}(\jincl_N)\colon\pi_{d-2}(N\sm D^d)\to\pi_{d-2}(N)$ is an isomorphism by the first part of the following standard result {\color{black}(see for example \cite[Lem.4.38]{Hatcher-AT})}. 
\begin{lemma}\label{lem:hatcher}
    The natural inclusion induces isomorphisms $\pi_n(N\sm D^d)\to \pi_nN$ for all $n\leq d-2$. Moreover, $\pi_d(N,N\sm D^d)\cong\Z[\pi_1N]\{\bm{\Phi}\}$, generated by the homotopy class of the attaching map $\Phi\colon( D^d,\S^{d-1})\to(N,N\sm D^d)$, so there is an exact sequence
    \[\begin{tikzcd}[column sep=large]
    \pi_d(N\sm D^d)\rar & 
    \pi_dN\rar & \Z[\pi_1N]\{\bm{\Phi}\}\rar & 
    \pi_{d-1}(N\sm D^d)\arrow[two heads]{r}{} & 
    \pi_{d-1}N.
    \end{tikzcd}\qedhere
    \]
\end{lemma}
\end{proof}
Note that Lemma~\ref{lem:kernels-agree} immediately implies Corollaries~\ref{cor:iso-of-circles-from-arcs} and~\ref{cor:iso-of-circles-and-arcs} from the introduction.

As a side remark, in most cases $\Z[\pi_1N]\{\bm{\Phi}\}\to\pi_{d-1}(N\sm D^d)$ is injective, thanks to the following.
\begin{lemma}\label{lem:QHS}
    The map $\pi_d(N\sm D^d)\to\pi_dN$ is not surjective if and only if the universal cover $\wt{N}$ is a rational homology sphere. 
    Moreover, if for some $g\in\pi_1N$ the class $g\bm{\Phi}$ is trivial in $\pi_{d-1}(N\sm D^d)$, then $N$ is simply connected. 
\end{lemma}
\begin{proof}
    We first claim that the image of $\pi_d(N\sm D^d)\to\pi_dN$ precisely consists of degree $0$ maps. Namely, the classes in $\pi_dN$ that come from $\pi_d(N\sm D^d)$ are represented by nonsurjective maps, which have degree $0$. Conversely, any degree $0$ map $f\colon\S^d\to N$ is homotopic to a map missing a point in $N$, so after a homotopy we can assume it misses a neighbourhood $D^d$ of the point $\s(e)$, and hence lifts to a map $\S^d\to N\sm D^d$.
    
    Next, we show that there exists a class in $\pi_dN$ of nontrivial degree if and only if the universal cover $\wt{N}$ is a rational homology sphere. Namely, for $f\colon\S^d\to N$ and $\alpha\in H^k(N;\Z)$ with $k\neq0,d$ we have $f^*\alpha=0$, so $0=f_*([\S^d]\cap f^*\alpha)=n([N]\cap\alpha)$, where $n=\deg(f)\in\Z$
    and $[N]\cap-$ is the Poincar\'e isomorphism. Thus, $n\neq0$ implies that $N$ is a rational homology sphere, and so is $\wt{N}$, by the same argument applied to a lift $\wt{f}\colon\S^d\to\wt{N}$, which also must have nontrivial degree (equal to ${n}/{|\pi_1N|}$, so $\pi_1N$ is finite and actually trivial for even $d$, as seen using the Euler characteristic). 
    Conversely, if $\wt{N}$ is a rational homology sphere then $\deg\colon\pi_n\wt{N}\otimes\Q\to H_d(\wt{N};\Q)\cong\Q$ is an isomorphism by the rational Hurewicz theorem, so there exists $\wt{f}\colon\S^d\to\wt{N}$ whose degree $n$ is a rational generator. Then $f\colon\S^d\to\wt{N}\to N$ has degree $n|\pi_1X|\neq0$.
    
    The relative Hurewicz map takes $g\bm{\Phi}$ to $1\in\Z\cong H_d(N,N\sm D^d)$, so if $f\in\pi_dN$ maps to $g\bm{\Phi}$, then $\deg(f)=1$. But since $|\pi_1N|$ has to divide $\deg(f)$, we get $|\pi_1N|=1$, proving the last claim.
\end{proof}

Combining Theorem~\ref{thm:KT-arcs}, which computes $\ker(\pi_{d-3}p_\u,\u)\subseteq\pi_{d-3}\Arcs{N\sm D^d}$ using the Dax invariant, with our kernel-cokernel sequence and Lemma~\ref{lem:kernels-agree}, for $d\geq4$ we obtain:
\[\begin{tikzcd}[column sep=small]
    \{b\in\pi_{d-2}N\mid b=\bm{\s}\,b\}\ar{r}{\delta^{whisk}_{\s}} 
    & \ker(\pi_{d-3}p_\u,\u)\ar[tail,two heads,shift right=6pt]{d}{\cong}[swap]{\Dax}\ar[two heads]{r}{} 
    & \ker(\pi_{d-3}\deriv\incl_N,\s)\\
    &
    \faktor{\Z[\pi_1(N\sm D^d)\sm1]}{\dax_\u(\pi_{d-1}(N\sm D^d))}\ar[tail,two heads,shift right=6pt]{u}[swap]{\partial\realmap} &
\end{tikzcd}
\]
Thus, $\ker(\pi_{d-3}\deriv\incl_N,\s)$ is isomorphic to the quotient of $\ker(\pi_{d-3}p_\u,\u)$ by the image of $\delta^{whisk}_{\s}$, which is equal to the quotient of $\Z[\pi_1(N\sm D^d)\sm1]$ by the subgroup of relations
\[
    rel_\s\coloneqq \dax_\u(\pi_{d-1}(N\sm D^d))\oplus\{\Dax(\delta^{whisk}_{\s}(b))\mid b=\bm{\s}\,b\in\pi_{d-2}N\}.
\]
This finishes the proof of Theorem~\ref{thm:circles-main2}.\hfill\qedsymbol


\subsection{Circles in dimension three}\label{subsec:dim-3}

The discussion in this section so far gives us also the diagram
\begin{equation}\label{eq:proofD}
\begin{tikzcd}[column sep=small]
    \{b\in\pi_1N\mid b=\bm{\s}\,b\}\ar{r}{\delta^{whisk}_{\s}} 
    & \KK(N\sm\D^3,\bm{\u})\ar[,two heads]{d}{}[swap]{\Dax_\u}\ar[two heads]{r}{} 
    & \KK(N,\bm{\s})\\
    &
    \faktor{\Z[\pi_1(N\sm D^d)\sm1]}{\dax_\u(\pi_2(N\sm D^d))} &
\end{tikzcd}
\end{equation}
where $\KK(N\sm D^3,\bm{\u})$ and $\KK(N,\bm{\s})$ are just sets, of respectively arcs homotopic to $\u$ and circles homotopic to $\s$, and the other two terms are groups.

Since in this dimension $\bm{\s}\,b\coloneqq\bm{\s}\cdot b\cdot \bm{\s}^{-1}$, the leftmost term is the centralizer $\zeta(\bm{\s})$ of $\bm{\s}\in\pi_1N$. It acts on $\KK(N\sm\D^3,\bm{\u})$ by sending $K\colon\D^1\hra N\sm\D^3$ to the arc $\delta^{whisk}_K(b)$ which is an embedded change of $K$ by the whisker $b$. Similarly as above, $\Dax_\u(\delta^{whisk}_K(b))$ counts double point loops in a homotopy from this arc to $\u$, which exists since $b\s b^{-1}$ is homotopic to $\s$.

Moreover, recall from Corollary~\ref{cor:dim3} that we computed $\dax_\u(\pi_2(N\sm D^d))$ in terms of the set $\mathcal{S}(N\sm D^d)$ of generators of the $\Z[\pi_1N]$-module $\pi_2(N\sm D^d)$, and that in Corollary~\ref{cor:gPhi} we saw that the sphere $\partial\D^3$ gives $\ol{g}-g\ol{\bm{\s}}$. We also observe that $\lambdabar(ga,\bm{\u})=\lambdabar(ga,\bm{\s})$. 

Recall that Theorem~\ref{thm:circles-3d} asserts that there is a well-defined invariant $\Dax_\s$ from the set $\KK(N,\bm{\s})$ to the set of equivalence classes of the set-theoretic action of $b\in\zeta(\bm{\s})$ via $b\star r = b\cdot r\cdot b^{-1} + \Dax(\delta^{whisk}_{\s}(b))$ on the group
    \[\faktor{\Z[\pi_1N\sm1]}{
        \big\langle
        \ol{g}-g\ol{\bm{\s}},\; 
        \lambdabar(ga,\bm{\s})-\lambdabar(ga,g)+\ol{\lambdabar(ga,g)}\mid g\in\pi_1N,\, a\in \mathcal{S}(N\sm\D^3)\;\big\rangle
        }.
    \]
\begin{proof}[Proof of Theorem~\ref{thm:circles-3d}]
    It remains to check that two different lifts $K,K'\in\KK(N\sm\D^3,\bm{\s})$ of a knot $L\in\KK(N,\bm{\s})$ have the same Dax invariant modulo the given action. By the exactness of the sequence \eqref{eq:proofD} we have $K'=\delta^{whisk}_{K}(b)$ for some $b=\bm{\s}\,b\in\pi_1N$. A homotopy from $\delta^{whisk}_{K}(b)$ to $\u$ can be written as a homotopy $h_1$ from $\delta^{whisk}_{K}(b)$ to $\delta^{whisk}_{\u}(b)$, where we use any homotopy from $K$ to $\u$, then followed by a homotopy $h_2$ from $\delta^{whisk}_{\u}(b)$ to $\u$. Therefore, using the additivity of the Dax invariant (see {\color{black}the proof of} Lemma~\ref{lem:Dax-dim3}) we have
    \[
    \Dax_\u(K')=\Dax_\u(\delta^{whisk}_{K}(b))=\Dax(h_1)+\Dax(h_2)=b\cdot\Dax_\u(K)\cdot b^{-1}+\Dax_\u(\delta^{whisk}_{\s}(b)).
    \]
    Therefore, $\Dax_\u(K)$ and $\Dax_\u(K')$ are in the same orbit of the action of $\zeta(\bm{\s})$, as desired.
\end{proof}

\begin{remark}\label{rem:Vassiliev-circles2}
    In Theorem~\ref{thm:Vassiliev} we showed that $\Dax_\u$ is the universal type $\leq1$ Vassiliev invariant of $\KK(X,\bm{\u})$, by writing any $v\colon\KK(X,\bm{\u})\to A$ as $v(K)-v(\u)=w_v\circ\Dax(K)$. We had that $w_v$ is well defined since it vanished on the image of $\dax_\u(\pi_2X)$. 
    
    The proof that for any $v\colon\KK(N,\bm{\s})\to A$ we have $v(K)-v(\s)=w_v\circ\Dax_\s(K)$ follows the same steps. Since $w_v$ vanishes on $\dax_\u(\pi_2(N\sm\D^d))$ by the same argument as for arcs, it remains to show $w_v$ is well defined modulo the action of $\zeta(\bm{\s})$. Observe that for any $r$ we have $r=\Dax_\u(K)$ for some $K\in\KK(N\sm D^d,\bm{\u})$, and we saw that $brb^{-1}+\Dax_\u(\delta^{whisk}_\s(b)$ is then equal to $\Dax_\u(K')$, for some arc $K'$ that closes up to the same knotted circle as $K$. Since $v$ is an invariant of circles, we conclude that $w_v(r)=w_v(\Dax_\s(K))=w_v(\Dax_\s(K'))=w_v(brb^{-1}+\Dax_\u(\delta^{whisk}_\s(b)){\color{black})}$.
\end{remark}


\printbibliography

\vspace{1em}
\hrule

\end{document}